\RequirePackage{fix-cm}
\documentclass[smallextended]{svjour3}       
\smartqed  
\usepackage{amsmath}
\usepackage{amsbsy}
\usepackage{amssymb}
\usepackage{textcomp}
\usepackage{graphicx}
\usepackage{amsfonts}
\usepackage{longtable}
\usepackage{multirow}
\usepackage{lscape}
\usepackage{url}
\usepackage{citeref}
\usepackage{subfig}
\usepackage{longtable}
\usepackage{lscape}
\usepackage{hhline}
\usepackage[utf8]{inputenc}
\usepackage{colortbl,array} 
\usepackage{multirow,bigdelim}
\usepackage{booktabs}
\usepackage[ruled]{algorithm2e}
\usepackage{epstopdf}
\def\eeq{\end{equation}} 
\def\lbeq#1{\begin{equation} \label{#1}} 
\def\Rz{\mathbb{R}}

\def\<{\langle} 				
\def\>{\rangle} 				
\def\eps{\varepsilon} 
\newtheorem{thm}{Theorem.\nopagebreak}[section]
\newtheorem{alg}[thm]{Algorithm.}
\def\bealg#1{\begin{alg}{\bf #1}\index{Algorithm!#1}\nopagebreak}
\def\ealg{\end{alg}}
\def\idx#1{#1\index{#1}} 
\def\D{\displaystyle}				
\def\ol{\overline}

\def\fct#1{\mathop{\rm #1}}	

\def\argmin{\fct{argmin}}

\def\half{\frac{1}{2}}
\newtheorem{prop}[thm]{Proposition.\nopagebreak}

\def\bary{\begin{array}}
\def\eary{\end{array}}

\def\att{					
    \marginpar[ \hspace*{\fill} \raisebox{-0.2em}{\rule{2mm}{1.2em}} ]
    {\raisebox{-0.2em}{\rule{2mm}{1.2em}} }
        }

\def\st{\fct{s.t.~}}

\begin{document}

\title{Optimal subgradient algorithms with application to large-scale linear inverse problems
}
\author{Masoud Ahookhosh}


\institute{Faculty of Mathematics, University of Vienna,
Oskar-Morgenstern-Platz 1, 1090 Vienna, Austria\\
\email{masoud.ahookhosh@univie.ac.at}
}

\date{Received: date / Accepted: date}

\maketitle

\begin{abstract}
This study addresses some algorithms for solving structured unconstrained convex optimization problems using first-order information where the underlying function includes high-dimensional data. The primary aim is to develop an implementable algorithmic framework for solving problems with multi-term composite objective functions involving linear mappings using the optimal subgradient algorithm, OSGA, proposed by {\sc Neumaier} in \cite{NeuO}. To this end, we propose some prox-functions for which the corresponding subproblem of OSGA is solved in a closed form. Considering various inverse problems arising in signal and image processing, machine learning, statistics, we report extensive numerical and comparisons with several state-of-the-art solvers proposing favourably performance of our algorithm. We also compare with the most widely used optimal first-order methods for some smooth and nonsmooth convex problems. Surprisingly, when some Nesterov-type optimal methods originally proposed for smooth problems are adapted for solving nonsmooth problems by simply passing a subgradient instead of the gradient, the results of these subgradient-based algorithms are competitive and totally interesting for solving nonsmooth problems. Finally, the OSGA software package is available.

\keywords{Structured convex optimization \and Linear inverse problems \and
Nonsmooth optimization \and Sparse optimization \and First-order black-box oracle \and Optimal complexity \and Subgradient methods \and High-dimensional data \and Compressed sensing \and Image restoration}
 \subclass{90C25 \and 90C60 \and 49M37 \and 65K05}
\end{abstract}

\section{Introduction}

In many applications, e.g., those arising in signal and image processing, machine learning, compressed sensing, geophysics and statistics, key features cannot be studied by straightforward investigations, but must be indirectly inferred from some observable quantities. Thanks to this characteristic, they are usually referred as 
{\bf \emph{inverse problems}}. If a linear relevance between the features of interest and observed data can be prescribed, this leads to linear inverse problems. If $y \in \Rz^m$ be an indirect observation of an original object $x \in \Rz^n$ and $\mathcal{A}:\Rz^n \rightarrow \Rz^m$ is a  linear operator, then the linear inverse problem is defined using
\lbeq{e.inv2}
y = \mathcal{A}x + \nu,
\eeq 
where $\nu \in \Rz^m$ represents an additive noise about which little is known apart from which qualitative knowledge. \\

{\bf Motivation \& history.} In practice, the systems (\ref{e.inv2}) is typically {\bf \emph{underdetermined}}, {\bf \emph{rank-deficient}} or {\bf \emph{ill-conditioned}}. The primary difficulty with linear inverse problems is that the inverse object is extremely sensitive to $y$
due to small or zero singular values of $\mathcal{A}$ meaning ill-conditioned systems behave like singular cases. Indeed, in the case that $\mathcal{A}^{-1}$ for square problems or {\bf \emph{pseudo-inverse}} $\mathcal{A}^\dagger = (\mathcal{A}^*\mathcal{A})^{-1} \mathcal{A}^*$ for full rank over-determined systems is exist, then analyzing the singular value decomposition shows that $\tilde{x} = \mathcal{A}^{-1}y$ or $\tilde{x} = \mathcal{A}^\dagger y$ be inaccurate and meaningless approximation for $x$, see \cite{NeuI}. Furthermore, when the vector $\delta$ is not known, one cannot solve (\ref{e.inv2}) directly. 

From underdetermined or rank-deficient feature of inverse problems, we know that if a solution exists, then there exist infinitely many solutions. Hence some additional information is required to determine a satisfactory solution of (\ref{e.inv2}). It is usually interested to determine a solution by minimizing the residual $\|\mathcal{A}x - y\|_2$, leading to the well-known least-square problem
\begin{equation}\label{e.lea}
\min_{x \in \Rz^n} ~~ \frac{1}{2} \|\mathcal{A}x - y\|_2^2,\\
\end{equation}
where $\|.\|_2$ denotes $l_2$-norm. In view of ill-condition feature of the problem (\ref{e.inv2}), the solution of (\ref{e.lea}) is usually improper. Hence Tikhonov in \cite{Tik} proposed the following penalized minimization problem
\begin{equation}\label{e.tikh}
\min_{x \in \Rz^n} ~~  \frac{1}{2} \|\mathcal{A}x - y\|_2^2 + \frac{\lambda}{2} \|x\|_2^2,\\
\end{equation}
where $\lambda$ is the regularization parameter that controls the trade-off between the least square data fitting term and the regularization term. The reformulated problem is also smooth and convex, and selecting a suitable regularization parameter leads to a well-posed problem. In many applications, we seek the sparsest solution of (\ref{e.inv2}) among all the solutions provided that $y$ is acquired from a highly sparse observation. To do so, the following problem is proposed
\[
    \bary{ll}
    \mathrm{min} & \|x\|_0\\
    \st  & \|\mathcal{A}x - y\| \leq \epsilon
    \eary
\] 
where $\|x\|_0$ is the number of all nonzero elements of $x$ and $\epsilon$ is a nonnegative constant. It is well-known that the objective function is nonconvex, however its convex relaxation is much more preferable. As investigated in \cite{BruDE}, this is achievable by considering the problem
\[
    \bary{ll}
    \mathrm{min} & \|x\|_1\\
    \st  & \|\mathcal{A}x - y\| \leq \epsilon.
    \eary
\] 
or its unconstrained version
\begin{equation}\label{e.BPD}
\min_{x \in \Rz^n} ~~ \frac{1}{2} \|\mathcal{A}x - y\|_2^2 + \lambda \|x\|_1,\\
\end{equation}
which is referred as {\bf \emph{basis pursuit denoising}} or {\bf \emph{lasso}}. Due to non-differentiability of the $l_1$-norm, this problem is nonsmooth and its solving is more difficult than solving (\ref{e.tikh}). 

In general, linear inverse problems can be modelled as a special case of the following {\bf \emph{affine composite}} minimization problem
\begin{equation}\label{e.genf}
\min_{x \in \Rz^n} ~~ f(\mathcal{A}x) + \varphi(\mathcal{W}x).
\end{equation}
Usually, $f: \Rz^n \rightarrow \Rz$ is a smooth convex {\bf \emph{loss function}} defined by a black-box oracle and $\varphi: \Rz^n \rightarrow \Rz$ is a {\bf \emph{regularizer}} or {\bf \emph{regularization function}}, which is usually nonsmooth and possibly nonconvex. This means that the general underlying regularizer can be smooth or nonsmooth and convex or nonconvex. In this paper, we consider {\bf \emph{structured convex optimization}} problems of the form (\ref{e.genf}) when regularizers can be smooth or nonsmooth convex terms, so nonconvex regularizers are not the matter of this study.

Over the past few decades, because of emerging many applications of inverse problems in applied sciences and engineering, solving structured composite optimization problems of the form (\ref{e.genf}) attracts many interesting researches. To exploit different features of objective functions, a variety of techniques for solving this class of problems are introduced, where proposed optimization schemes employ particular structures of objective functions. It is also known that solving nonsmooth problems are usually more difficult than solving smooth problems, even for unconstrained cases. In the meantime, black-box convex optimization is theoretically and computationally regarded as a mature area in optimization and frequently employed to solve problems with large number of variables appearing in various applications. Since problems of the form (\ref{e.genf}) mostly include high-dimensional data, optimization schemes should avoid of exploiting costly operations and also high amount of memory. Therefore, the {\bf \emph{first-order black-box oracle}} is intensively used to construct appropriate schemes for handling large-scale problems, where just function and subgradient evaluations are needed. Some popular first-order references are gradient-based algorithms \cite{CheZ,Nes.E,Nes.Co,Nes.Un}, subgradient-based schemes \cite{BecT1,BenN,NemY,Nes.Bo,Nes.PD}, bundle-type procedures \cite{LanB}, smoothing techniques \cite{BecT2,Nes.Sm} and proximal gradient methods \cite{BecT2,BecCG,ParB}. 

In the context of convex optimization, the {\bf \emph{analytical complexity}} refer to the number of calls of an oracle, for example iterations or function and subgradient evaluations, to reach an $\eps$-solution of the optimum, for more information see \cite{NemY,Nes.Bo}. In 1983, {\sc Nemirovski} \& {\sc Yudin} in \cite{NemY} proved that the lower complexity bound of first-order methods for finding an $\varepsilon$-solution of the optimum for smooth convex functions with Lipschitz continuous gradient is $O(1/\sqrt{\eps})$, while for Lipschitz continuous nonsmooth convex function it is $O(1/\eps^2)$. Indeed, the interesting feature of these error bounds is that they are independent of the problem's dimension. For the established class of problems, an algorithm called {\bf \emph{optimal}} if it can achieve these worst-case complexities. In a case of nonsmooth convex problems, the well-known subgradient decent and mirror descent methods are optimal, see \cite{NemY}. The seminal optimal first-order method for smooth convex problems with Lipschitz continuous gradient is introduced by {\sc Nesterov} \cite{Nes.83} in 1983, and since then he has developed the idea of optimal methods for various classes of problems in \cite{Nes.88,Nes.Bo,Nes.Co,Nes.Un}. It is clear that the gap between the complexity of smooth and nonsmooth convex problems is considerably big. This fact motivates {\sc Nesterov} to develop some schemes with better performance for nonsmooth convex problems so called smoothing techniques  where the worst-case complexity of these methods is in order $O(1/\eps)$ and can be improved using the idea of optimal gradient methods, see \cite{Nes.Sm,Nes.Se,Nes.E}. The theoretical analysis and surprising computational results of optimal methods are totally interesting, especially for solving large-scale problems. These appealing features along with increasing the interest of solving large-scale problems attract many authors to develop the idea of the optimal first-order methods, for examples {\sc Beck} \& {\sc Teboulle} \cite{BecT1}, {\sc Becker} et al.\ \cite{BecCG}, {\sc Devolder} et al.\ \cite{DevGN}, {\sc Gonzaga} et al.\ \cite{GonK,GonKR}, {\sc Lan} \cite{LanB}, {\sc Meng \& Chen} \cite{MenC}, {\sc Nemirovski} \& {\sc Nesterov} \cite{NemNes} and {\sc Tseng} \cite{Tse}.  

In most of the studies mentioned, the first difficulty in implementation is the algorithms need the knowledge of global parameters such as the Lipschitz constant of objective function for nonsmooth problems and the Lipschitz constant for its gradient in smooth cases. {\sc Nesterov} in \cite{Nes.Un} proposed an adaptive procedure to approximate the Lipschitz constant, however it still needs a suitable initial guess. Recently {\sc Lan} in \cite{LanB} based on employing bundle techniques avoids of requiring the global parameters. More recently, {\sc Neumaier} in \cite{NeuO} introduced a novel optimal algorithm on the basis of incorporating a linear relaxation of the objective function and a prox-function into a fractional subproblem to construct a framework that can be employed for both smooth and nonsmooth convex optimization in the same time and to avoid of needing Lipschitz constants. It called OSGA, where it only needs the first-order black-box oracle. In a case that no strong convexity is assumed, it shares the uniformity, freeness of global parameters and complexity properties of Lan's approach, but has a far simpler structure and derivation. Thanks to achieving the optimal complexity of first-order methods, it can be regarded as a fully adaptive alternative of Nesterov-type optimal methods. Furthermore, low memory requirements of this algorithm make it appropriate for solving large-scale practical problems. \\

{\bf Contribution.} The main focus of this paper is to extend and develop the basic idea of the optimal subgradient algorithm for minimizing general composite functions to deal with structured convex optimization problems in applications. There exist many applications of linear inverse problems with a complicated objective function consisting of various linear operators and regularizers, which are not supported by algorithms and solvers using the proximal mapping in their structures. For example, consider the scaled lasso problem that is defined the same as (\ref{e.BPD}) in which the regularization term $\|x\|_1$ is substituted by $\|Wx\|_1$ for a general linear operator $W$. In this case, the following proximal mapping should be solved as an auxiliary subproblem
\begin{equation*}
\mathrm{prox}_{\lambda}(y) = \argmin_{x \in \Rz^n} \frac{1}{2} \|x-y\|_2^2 + \|\mathcal{W} x\|_1.
\end{equation*}
Indeed, this proximity problem cannot be solved explicitly using {\bf \emph{iterative shrinkage-thresholding}} except for the case that $W$ is orthonormal, i.e., it cannot be solved using usual proximal-based algorithms. Thanks to accessible requirements of OSGA needing just function and subgradient evaluations, solving most of such complicated problems are tractable in reasonable costs. Furthermore, some appropriate prox-functions are proposed for unconstrained problems of the form (\ref{e.genf}) such that the corresponding subproblem of OSGA can be solved explicitly. We also adapt some Nesterov-type optimal methods originally proposed for smooth problems to solve nonsmooth problems by passing subgradient instead of gradient. The surprising results indicate that the adapted algorithms perform competitively or even better than optimal methods specifically proposed for solving nonsmooth problems. Finally, the paper is accompanied
with a software release.\\

{\bf Contents.} The rest of the paper is organized as follows. In next section, we establish the basic idea of the optimal subgradient algorithm for the general affine composite functions. Section 3 devotes to implementing OSGA and comparing it with some popular first-order methods and state-of-the-art solvers for some practical problems in applications. The OSGA software package is described in Section 4. Finally, we conclude our results in Section 5.

\vspace{-3mm}
\section{Optimal subgradient framework}
In this section, we shall give a short sketch of the optimal subgradient framework introduced by {\sc Neumaier} in \cite{NeuO} for composite functions in the presence of linear operators which is appropriate for solving considered applications of linear inverse problems. 

This paper considers composite functions of the form (\ref{e.genf}) or a more general form that will be discussed in the sequel, however we can still solve constrained problems by defining an indicator function. Let us to consider the following convex constrained problem
\begin{equation}\label{e.con}
\min_{x \in C}~ f(\mathcal{A} x),
\end{equation}
where $f:C \to \Rz$ is a convex function defined on a nonempty convex subset $C$ of a finite-dimensional real vector space $\mathcal{X}$ with a bilinear pairing $\<h,z\>$, where $h$ is in the dual space $\mathcal{X}^*$ which is formed by all linear functions on $\Rz$. The indicator function of the set $C$ is defined by
\begin{equation}\label{e.ind}
  I_C(x) = \left\{
  \begin{array}{ll}
    0\ \ & \hbox{$x \in C$}; \\
    \infty \ \ & \hbox{otherwise}. \\
  \end{array}  \right. 
\end{equation}
Since $C$ is convex set, it is straightforward to show $I_C$ is a convex function. Thus, by setting $\varphi(x) = I_C(x)$, we can reformulate the problem (\ref{e.con}) in an equivalent unconstrained form (\ref{e.genf}), which is desired problem of the current paper. 

In some applications, one has no easy access to the matrices used in (\ref{e.genf}), instead it is assembled in a subroutine without ever forming explicit matrices so that it is preferred to employ linear operators in (\ref{e.genf}) instead of their matrix representations. We generally consider functions of the form
\lbeq{e.genf1}
\Psi(x) = \sum_{i=1}^{n_1} f_i(\mathcal{A}_i x) + \sum_{j=1}^{n_2} \varphi_j(\mathcal{W}_j x),
\eeq
where $f_i: \mathcal{U}_i \to \Rz$ for $i = 1,2,\cdots,n_1$ are convex smooth functions, $\varphi_j: \mathcal{V}_j \to \Rz$ for $j = 1,2,\cdots,n_2$ are smooth or nonsmooth convex functions and $\mathcal{A}_i: \mathcal{X} \to \mathcal{U}_i$ for $i = 1,2,\cdots,n_1$ and $\mathcal{W}_j: \mathcal{X} \to \mathcal{V}_j$ for $j = 1,2,\cdots,n_2$ are linear operators. Here, $\Psi$ is called the {\bf \emph{multi-term affine composite function}}. Since each affine term $\mathcal{A}_i x$ or $\mathcal{W}_j x$ and each function $f_i$ or $\varphi_j$ for $i = 1, \cdots,n_1$ and $j = 1, \cdots,n_2$ is convex and the domain of $\Psi$ defined by
\[
\mathrm{dom}~ \Psi = \left( \bigcap_{i = 1}^{n_1} \mathrm{dom}~ f_i(\mathcal{A}_i x) \right) \bigcap \left( \bigcap_{j = 1}^{n_2} \mathrm{dom}~ \varphi_j (\mathcal{W}_j x) \right)
\]
is a convex set, then the function $\Psi$ is convex on its domain. Considering (\ref{e.genf1}), we are generally interested in solving the following composite convex optimization problem
\lbeq{e.genf2}
\Psi^* := \min_{x \in \mathcal{X}}~ \Psi(x) 
\eeq
Under considered properties of $\Psi$, the problem (\ref{e.genf2}) has a global optimizer denoted by $x^*$ and $\Psi^* = \Psi(x^*)$ is its global optimum. The functions of the form $\Psi$ are frequently appeared in the recent interests of employing hybrid regularizations or mixed penalty functions for solving problems in the fields such as signal and image processing, machine learning and statistics.  In the sequel, we will see that many well studied structured optimization problems are special cases of (\ref{e.genf1}). It is also clear that (\ref{e.genf2}) is a generalization of (\ref{e.genf}). For example, the following {\bf \emph{scaled elastic net}} problem
\lbeq{e.elan}
\Psi(x) = \frac{1}{2} \|\mathcal{A}x - b\|_2^2 + \lambda_1 \|\mathcal{W}_1x\|_1 + \frac{\lambda_2}{2} \|\mathcal{W}_2x\|_2^2
\eeq
can be considered in the form (\ref{e.genf1}) by setting $\Psi(x) = f_1(\mathcal{A}_1 x) + \varphi_1(\mathcal{W}_1 x) + \varphi_2(\mathcal{W}_2 x)$ where $f_1(\mathcal{A}_1 x) = \frac{1}{2} \|\mathcal{A}_1 x - b\|_2^2$, $\varphi_1(\mathcal{W}_1 x) = \lambda_1 \|\mathcal{W}_1 x\|_1$ and $\varphi_2(\mathcal{W}_2 x) = \frac{\lambda_2}{2} \|\mathcal{W}_2 x\|_2^2$. \\

{\bf Basic idea of optimal subgradient scheme.} This section summarizes the basic idea of OSGA \cite{NeuO} which is adapted for the considered unconstrained problem (\ref{e.genf2}). Let us to consider an arbitrary convex function $\Psi: \mathcal{X} \rightarrow \Rz$ of the form (\ref{e.genf1}). A vector $g_{\Psi}(x) \in \mathcal{X}^* $ is called a {\bf \emph{subgradient}} of $\Psi$ at a given point $x \in \mathrm{dom}~ \Psi$ if for any $z \in \mathrm{dom}~ \Psi$ we have that
\[
\Psi(z) \geq \Psi(x) + \langle g_{\Psi}(x),z - x \rangle. 
\]
The set of all subgradients of the function $\Psi$ at $x$ is called {\bf \emph{subdifferential}} of $\Psi$ at $x$ denoted by $\partial \Psi(x)$. The subdifferential of a general convex function $\Psi$ at any point $x$ in its domain is nonempty, and it is unique if $\Psi$ is differentiable.

We now consider a minimization problem of the form (\ref{e.genf1}) and suppose that the first-order black-box oracle 
\[
\mathcal{O}(x) = (\Psi(x),g_{\Psi}(x)),
\]
is available, where $g_{\Psi}(x) \in \partial \Psi(x)$ at $x \in \mathcal{X}$. The primary objective of the optimal subgradient algorithm (OSGA) is to construct a sequence of iterations such that the error bound $\Psi(x_b)-\Psi^*$ is monotonically reducing to be smaller than an accuracy parameter $\eps$, i.e.
\[
0 \leq \Psi_b - \Psi^* \leq \eps,
\]
where $\Psi_b := \Psi(x_b)$ is the function value in the best known point $x_b$. 

To derive a framework for OSGA, let us to define the following {\bf \emph{affine global underestimator}}, which is a linear relaxation of $\Psi$,
\lbeq{e.f1}
\Psi(z) \geq \gamma +\<h,z\>, 
\eeq
for all $z\in \mathrm{dom}~ \Psi$ with $\gamma\in\Rz$ and $h\in \mathcal{X}^*$, which is clearly available due to the first-order condition for the convex function $\Psi$. We also consider a continuously differentiable {\bf \emph{prox-function}} $Q: \mathcal{X} \to \Rz$, which is a strongly convex function with the convexity parameter set to $\sigma$ satisfying 
\lbeq{e.strc}
Q(\lambda x + (1 - \lambda) z) + \lambda (1 - \lambda) \frac{\sigma}{2} \|x - z\|^2 \leq \lambda Q(x) + (1 - \lambda)Q(z), 
\eeq
for all $x,z \in \mathcal{X}$ and $\lambda \in (0,1)$ where $\|.\|$ is an arbitrary norm of the vector space $\mathcal{X}$. This is equivalent to 
\lbeq{e.strc}
Q(z)\ge Q(x)+\<g_Q(x),z-x\>+\frac{\sigma}{2}\|z-x\|^2, 
\eeq
for all $x,z\in \mathcal{X}$ where $g_Q(x)$ denotes the gradient of $Q$ at $x$. Furthermore, we suppose that
\lbeq{e.Qinf}
Q_0: = \inf_{z\in \mathcal{X}} Q(z) >0.
\eeq
The point $z_0 \in \mathcal{X}$ minimizing the problem (\ref{e.Qinf}) is called the {\bf \emph{center}} of $Q$. At the end of this section, we will propose some prox-functions satisfying these conditions. Afterwards, the following minimization subproblem is defined
\lbeq{e.Eeta}
E(\gamma,h):=-\inf_{z\in ~\mathrm{dom}~ \Psi} \frac{\gamma+\<h,z\>}{Q(z)}.
\eeq
By the definition of prox-functions, it is clear that the denominator $Q(z)$ is positive and bounded away from zero on the feasible region $\mathrm{dom}~ \Psi$. The solution of this subproblem (\ref{e.Eeta}) is denoted by $u := U(\gamma,h)$, and it is supposed that $e := E(\gamma,h)$ and $u$ can be effectively computed. In particular, for the unconstrained problem (\ref{e.genf1}), we will show that this subproblem can be solved explicitly and cheaply for some suitable prox-functions. \\

{\bf Theoretical analysis.} We here establish a bound on error $\Psi(x_b)-\Psi^*$ for the proposed scheme. The definition of $e = E(\gamma,h)$ in (\ref{e.Eeta}), for arbitrary $\gamma_b \in \Rz$ and $h\in \mathcal{X}^*$, simply implies that
\lbeq{e.bou}
\gamma_b + \<h,z\> \ge -E(\gamma_b,h) Q(z), 
\eeq
for all $z\in \mathcal{X}$. From (\ref{e.f1}) and setting $\gamma_b=\gamma-\Psi(x_b)$, we have
\[
\Psi(z) - \Psi(x_b) \geq \gamma_b + \<h,z\>,
\] 
for all $z\in \mathcal{X}$. By this fact, assuming $E(\gamma_b,h)\le \eta$, setting $z = x^*$ and (\ref{e.bou}), we conclude the following upper bound on the function value error
\lbeq{e.error}
0 \leq \Psi(x_b)-\Psi^* \leq \eta Q(x^*).
\eeq
It is clear that if an upper bound for $Q(x^*)$ is known or assumed, this bound translates into a computable error estimate for the minimal function value. But even in the absence of such an upper bound, the optimization problem (\ref{e.genf1}) can be solved with a target accuracy
\lbeq{e.target}
0 \leq \Psi(x_b)-\Psi^* \leq \eps Q(x^*)
\eeq
if one manages to decrease the error factor $\eta$ from its initial value until $\eta \leq \eps$, for some target accuracy tolerance $\eps$. This bound clearly means that the rate of decrease in the error bound $\Psi(x_b)-\Psi^*$, is at least the same as convergence rate of the sequence $\eta$. Indeed, this fact is the main motivation of OSGA. {\sc Neumaier} in \cite{NeuO} considers a problem of the form (\ref{e.con}) and derive the following complexity bounds for OSGA matching to the optimal complexity bounds of first-order methods for smooth and nonsmooth problems, see \cite{NemY,Nes.Bo}.

\begin{theorem}
Suppose that $f$ is a convex function of the form (\ref{e.con}) and let $\{x_k\}$ is generated by OSGA. Then complexity bounds for smooth and nonsmooth problems are as follows: \\
(i) (Nonsmooth complexity bound) If the point generated by OSGA stay in bounded region of the interior of $C$, or $f$ is Lipschitz continuous in $C$, then the total number of iterations needed is $O\left(1/\varepsilon^2\right)$. Thus the asymptotic worst case complexity is $O\left(1 / \varepsilon^2 \right)$. \\
(ii) (Smooth complexity bound) If $f$ has Lipschitz continuous gradient, the total number of iterations needed for the algorithm is $O\left(1 / \sqrt{\varepsilon} \right)$.
\end{theorem}

{\bf Algorithmic framework.} Before presenting the OSGA algorithm, we first look at the first-order oracle of the problem (\ref{e.genf1}) which is important in the presence of linear mappings, and then we outline the adapted algorithm. 

On the one hand, practical problems are commonly involving high-dimensional data, i.e. function and subgradient evaluations are quite expensive. On the other hand, in many applications the most computational cost of function and subgradient evaluations relates to applying direct and adjoint linear operators, originated from an existence of affine terms in (\ref{e.genf1}). This facts suggest that we should apply linear operators and their adjoints as less as possible during our implementations. For the sake of simplicity in the rest of the paper, we employ $\Psi_x$ and $g_x$ instead of $\Psi(x)$ and $g_{\Psi}(x)$, respectively. Considering the structure of (\ref{e.genf1}), we see that affine terms $\mathcal{A}_i x$ for $i = 1,2, \cdots,n_1$ and $\mathcal{W}_j x$ for $j = 1,2, \cdots,n_2$ are appearing in both function and subgradient evaluations. As a result, in the oracle, we set $v_x^i = \mathcal{A}_i x$ for $i = 1,2, \cdots,n_1$ and $w_x^j = \mathcal{W}_j x$ for $j = 1,2, \cdots,n_2$. Therefore, the function value and subgradient of $\Psi$ at $x$ are computed in the following scheme:\\

\begin{algorithm}[H] \label{a.nfofg}
\DontPrintSemicolon 
\KwIn{global parameters: $\mathcal{A}_i~ \mathrm{for}~ i = 1, \cdots,n_1$;~ $\mathcal{W}_j~ \mathrm{for}~ j = 1, \cdots,n_2$;
    ~ local parameters: $x$; }
\KwOut{$\Psi_x$;~ $g_{\Psi}(x)$;}
\Begin{
    $v_x^i \gets \mathcal{A} x~ \mathrm{for}~ i = 1, \cdots,n_1$;\; 
    $w_x^j \gets \mathcal{W} x~ \mathrm{for}~ j = 1, \cdots,n_2$; \;
    $\Psi_x \gets \sum_{i=1}^{n_1} f_i(v_x^i) +  \sum_{j=1}^{n_2} \varphi_j(w_x^j)$;\;
    $g_{\Psi}(x) \gets \sum_{i=1}^{n_1} \mathcal{A}_i^* ~\partial f_i(v_{x}^i) + 
              \sum_{j=1}^{n_2} \mathcal{W}_i^* ~\partial \varphi_j(w_{x}^j)$;
}
\caption{ NFO-FG nonsmooth first-order oracle}
\end{algorithm}

\vspace{6mm}
The structure of NFO-FG immediately implies that any call of the oracle $\mathcal{O}(x) = (\Psi_x, g_{\Psi}(x))$ requires $n_1 + n_2$ calls of each direct and adjoint linear operator. It is also clear that by this scheme, one can avoid of double applying of expensive linear operators in the computation of function values and subgradients. In cases that algorithms need only a function value or a subgradient, we consider two special cases of NFO-FG that just return a function value or a subgradient, which are respectively called NFO-F and NFO-G. We also emphasize that in the cases that total computational cost of the oracle dominated by applying linear mappings, the complexity of an algorithm can be measured by counting the number of direct and adjoint linear operators used to achieve a prescribed accuracy $\eps$.

Considering what established in this section and those introduced in Section 2 of \cite{NeuO} about linear relaxation constructions and the case of strongly convex functions, we outline a version of OSGA for multi-term affine composite function (\ref{e.genf1}) as follows:\\\\

\begin{algorithm}[H] \label{a.osga}
\DontPrintSemicolon 
\KwIn{global tuning parameters: $\delta$;~ $\alpha_{\max}\in{]0,1[}$;~ $0<\kappa'\le\kappa$;~~ local parameters: $x_0$;~$\mu \geq 0$;~  $\Psi_{\mathrm{target}}$;}
\KwOut{$x_b$; $\Psi_{x_b}$;}
\Begin{
    choose an initial best point $x_b$;\;
    compute $\Psi_{x_b}$ and $g_{\Psi}(x_b)$ using NFO-FG;\;
    \eIf{$\Psi_{x_b} \leq \Psi_{\mathrm{target}}$} {
        stop;\;
    }{
        $h \gets g_{x_b}-\mu g_Q(x_b)$;~~ $\gamma \gets \Psi_{x_b}-\mu Q(x_b)-\<h,x_b\>$;\;
        $\gamma_b \gets \gamma - \Psi_{x_b}$;~~ $u \gets U(\gamma_b,h)$;~~ $\eta \gets 
        E(\gamma_b,h)-\mu$;\;
    }
    $\alpha \gets \alpha_{\max}$;\\
    \While {stopping criteria do not hold}{
        $x \leftarrow x_b+\alpha(u-x_b)$;\;
        compute $\Psi_x$ and $g_{\Psi}(x)$ using NFO-FG;\;
        $g \leftarrow g_x-\mu g_Q(x)$;~~ $\ol h \leftarrow h+\alpha(g-h)$;\;
        $\ol\gamma \leftarrow \gamma+\alpha(\Psi_x - \mu Q(x)-\<g,x\>-\gamma)$;\;
        $x_b' \leftarrow \argmin_{z \in \{x_b,x\}} \Psi(z, v_z)$;~~ $\Psi_{x_b'} \leftarrow 
        \min \{\Psi_{x_b}, \Psi_x\}$;\;
        $\gamma_b' \leftarrow \ol\gamma-f_{x_b'}$;~~ $u' \leftarrow U(\gamma_b',\ol h)$;~~ 
        $x' \leftarrow x_b+\alpha(u'-x_b)$;\;
        compute $\Psi_{x'}$ using NFO-F;\;
        choose $\ol x_b$ in such a way that $\Psi_{\ol x_b}\le \min\{\Psi_{x_b'},\Psi_{x'} \}$;\;
        $\ol\gamma_b \leftarrow \ol\gamma - \Psi_{\ol x_b}$;~~ $\ol u \leftarrow U(\ol \gamma_b,\ol h)$;~~
        $\ol \eta \leftarrow E(\ol\gamma_b,\ol h)-\mu$;\;
        $x_b \leftarrow \ol x_b$; ~ $\Psi_{x_b} = \Psi_{\ol x_b}$;\;
        \eIf {$\Psi_{x_b} \leq \Psi_{\mathrm{target}}$}{
            stop;\;
        }{
            update the parameters $\alpha$, $h$, $\gamma$, $\eta$ and $u$ using PUS;
        }
    }
}
\caption{ {\sc OSGA} for multi-term affine composite functions}
\end{algorithm}

\vspace{6mm}
Pay attention to the structure of problem (\ref{e.genf1}) and the OSGA framework, we only require to calculate the following simple operations to implement the algorithm. 
{\renewcommand{\labelitemi}{$\bullet$}
\begin{itemize}
\item $x + \beta y$, where $x, y \in \mathcal{U}_i$ for $i = 1, \cdots, n_1$ and $\beta \in \Rz$;
\item $r + \beta s$, where $r, s \in \mathcal{V}_j$ for $j = 1, \cdots, n_2$ and $\beta \in \Rz$;
\item $\mathcal{A}_i x$ for $i = 1, \cdots, n_1$ and $\mathcal{W}_j x$ for $j = 1, \cdots, n_2$ where $x \in \mathcal{X}$;
\item $\mathcal{A}_i^* y$ with $y \in \mathcal{U}_i$ for $i = 1, \cdots, n_1$ and $\mathcal{W}_j^* y$ with $y \in \mathcal{V}_j$ for $j = 1, \cdots, n_2$.
\end{itemize}
As a result of (\ref{e.target}), a natural stopping criterion for the algorithm is the condition $\eta \leq \epsilon$, but our experiments shows that satisfying this condition is slow and takes many iterations. Then more sophisticated stopping criteria for the algorithm is needed. For example a bound on the number of iterations, function evaluations, subgradient evaluations or the running time or a combination of them can be used.

As discussed in \cite{NeuO}, OSGA uses the following scheme for updating the given parameters $\alpha$, $h$, $\gamma$, $\eta$ and $u$:\\
   
\begin{algorithm}[H] \label{a.par}
\DontPrintSemicolon 
\KwIn{global tuning parameters: $\delta$;~ $\alpha_{\max}\in{]0,1[}$;~ $0<\kappa'\le\kappa$;~~ local parameters: $\alpha$; $\eta$; $\bar{h}$; $\bar{\gamma}$; $\bar{\eta}$; $\bar{u}$;}
\KwOut{$\alpha$;~ $h$;~ $\gamma$;~ $\eta$;~ $u$;}
\Begin{
    $R \gets \left(\eta-\ol \eta)/(\delta\alpha \eta \right)$;\;
    \eIf{$R<1$} {
        $h \gets \ol h$;\;
    }{
        $\ol\alpha \leftarrow \min(\alpha e^{\kappa' (R-1)},\alpha_{\max})$;
    }
    $\alpha \gets \ol \alpha$;\;
    \If{$\ol \eta<\eta$}{
        $h \gets \ol h$;~ $\gamma \gets \ol \gamma$;~ $\eta \gets \ol \eta$;~ $u \gets \ol u$;\;
    }
}
\caption{ PUS parameters updating scheme}
\label{algo:max}
\end{algorithm}

\vspace{6mm}
We here emphasize that the composite version of optimal subgradient algorithms inherits the main characteristics of the original OSGA framework. Thus, we expect the same theoretical analysis to be the same as that reported for OSGA in \cite{NeuO} which means that the current version remains optimal. \\

{\bf Prox-functions and solving the auxiliary subproblem.} In the rest of this section, we first propose some prox-functions and then derive a closed form for the corresponding subproblem (\ref{e.Eeta}).

As mentioned, the prox-function $Q(z)$ should be a strongly convex function that is analytically known. This means that it takes its unique global minimum which is positive by the definition of $Q$. Suppose $z_0$ is the global minimizer of $Q$. By the definition of center of $Q$ and the first-order optimality condition for (\ref{e.Qinf}), we have that $g_Q(z_0) = 0$. This fact along with (\ref{e.strc}) imply
\lbeq{e.proin}
Q(z) \geq Q_0 + \frac{\sigma}{2} \|z - z_0\|^2,
\eeq
which simply means that $Q(z)>0$ for all $z \in \mathcal{X}$. In addition, it is interested that prox-functions be  separable. Taking this fact into account and (\ref{e.proin}), appropriate choices of prox-functions for unconstrained problem (\ref{e.genf1}) can be
\lbeq{e.prox1}
Q_1(z) := Q_0 + \frac{\sigma}{2} \|z - z_0\|_2^2
\eeq
and
\lbeq{e.prox2}
Q_2(z) := Q_0 + \frac{\sigma}{2} \sum_{i = 1}^n w_i (z_i - (z_0)_i)^2,
\eeq
where $Q_0 > 0$ and $w_i \geq 1$ for $i = 1, 2, \cdots, n$. It is not hard to show that both (\ref{e.prox1}) and (\ref{e.prox2}) are strongly convex functions satisfying (\ref{e.proin}). 

To introduce a more general class of prox-functions, let us to define the {\bf \emph{quadratic norm}} on vector space $\mathcal{X}$ using
\[
\|z\|_{\mathcal{D}}:=\sqrt{\<\mathcal{D}z,z\>}
\]
by means of a {\bf \emph{preconditioner}} $\mathcal{D}$, where $\mathcal{D}$ is symmetric and positive definite. The associated dual norm on $\mathcal{X}^*$ is then given by
\[
\|h\|_{*\mathcal{D}} := \|\mathcal{D}^{-1} h\|_{\mathcal{D}} = \sqrt{\<h,\mathcal{D}^{-1} h\>},
\]
where $\mathcal{D}^{-1}$ denotes an inverse of $\mathcal{D}$. Given a symmetric and positive definite preconditioner $\mathcal{D}$, then it is natural to consider the quadratic function
\lbeq{e.prox3}
Q(z):= Q_0 + \frac{\sigma}{2} \|z-z_0\|_\mathcal{D}^2,
\eeq
where $Q_0$ is a positive number and $z_0\in \mathcal{X}$ is the center of $Q$. Now, we should prove that $Q$ is a prox-function and satisfies (\ref{e.proin}). The fact that $g_Q(x) = \sigma \mathcal{D} (x - z_0)$ leads to 
\begin{equation*}
\begin{split}
Q(x) + \langle g_Q(x),z - x \rangle + \frac{\sigma}{2} \|z - x\|_{\mathcal{D}}^2 &= 
Q(x) + \langle g_Q(x), z - x \rangle + \frac{\sigma}{2} \langle \mathcal{D} (z - x), z - x \rangle\\
& = Q_0 + \frac{\sigma}{2} \langle \mathcal{D} (x - z_0), x - z_0 \rangle+ \langle g_Q(x), z - x \rangle + \frac{\sigma}{2} \langle \mathcal{D} (z - x), z - x \rangle\\
& = Q_0 + \frac{\sigma}{2} \langle \mathcal{D} (x - z_0), z - z_0 \rangle + \frac{\sigma}{2} \langle \mathcal{D} (z - z_0), z - x \rangle\\
& = Q_0 + \frac{\sigma}{2} \langle \mathcal{D} (x - z_0), z - z_0 \rangle + \frac{\sigma}{2} \langle \mathcal{D} 
(z - x), z - z_0\rangle\\
& = Q_0 + \frac{\sigma}{2} \langle \mathcal{D} (z - z_0), z - z_0 \rangle \\
& =  Q(z).
\end{split}
\end{equation*} 
This clearly means that $Q$ is strongly convex by the convexity parameter $\sigma$, i.e., (\ref{e.prox3}) is a prox-function. Moreover, $z_0$ is the center of $Q$ and $g_Q(z_0) = 0$. This fact together with (\ref{e.strc}) and (\ref{e.Qinf}) imply that (\ref{e.proin}) holds.

At this point, we emphasize that an efficient solving of the subproblem (\ref{e.Eeta}) is highly related to the selection of prox-functions. While using a suitable prox-function allows a very efficient way of computing $u$, employing other selections may significantly slow down the process of solving the auxiliary subproblem. Here, it is assumed that by appropriate selections of prox-functions the subproblem (\ref{e.Eeta}) can be solved explicitly. In the following result, we verify the solution of (\ref{e.Eeta}) using the prox-function (\ref{e.prox3}). 

\begin{prop}\label{p.2}
Suppose $Q$ is determined by (\ref{e.prox3}) and $Q_0 > 0$. Then $Q$ is a prox-function with the center $z_0$ and satisfies (\ref{e.proin}). Moreover, the subproblem (\ref{e.Eeta}) corresponds to this $Q$ is explicitly solved as follows
\lbeq{e.h0}
u = z_0 - E(\gamma,h)^{-1} \sigma^{-1} \mathcal{D}^{-1} h
\eeq
with
\lbeq{e.E0}
E(\gamma,h)
=\frac{-\beta_1+\sqrt{\beta_1^2 + 4 Q_0 \beta_2}}{2Q_0}
=\frac{2 \beta_2}{\beta_1+\sqrt{\beta_1^2 + 4 Q_0 \beta_2}},
\eeq
where $\beta_1 = \gamma + \langle h,z_0 \rangle$ and $\beta_2 = \left( \frac{1}{2} \sigma^{-2} - \sigma^{-1} \right) \|h\|_*^2$.
\end{prop}
\begin{proof} 
We already proved that $Q$ is a prox-function and (\ref{e.proin}) holds. The rest of the proof is similar to that discussed in \cite{NeuO} for $\sigma = 1$. \qed
\end{proof}

It is obvious that (\ref{e.prox1}) and (\ref{e.prox2}) are special cases of (\ref{e.prox3}), so the solution (\ref{e.h0}) and (\ref{e.E0}) derived for (\ref{e.prox3}) can be simply adapted for the situation  (\ref{e.prox1}) or (\ref{e.prox2}) used as a prox-function. Furthermore, notice that the error bound (\ref{e.error}) is proportional to $Q(x^*)$ implying that an acceptable choosing of $x_0$  make the term $Q(x^*) = Q_0 + \frac{1}{2} \|x^*-x_0\|_\mathcal{D}^2$ enough small. Hence, selecting a suitable starting point $x_0$ close to the optimizer $x^*$ as much as possible has a positive effect on giving a better complexity bound. This also propose that a reasonable choice for $Q_0$ can be $Q_0 \approx \half \| x^*-x_0 \|^2$. This fact will be considered along numerical experiments. 

We conclude this section by considering cases that the variable domain is a set of all $m \times n$ matrices, $\mathcal{X} = \Rz^{m \times n}$. In this case, we introduce a suitable prox-function such that the subproblem (\ref{e.Eeta}) is solved in a closed form. Details are described in the following proposition.

\begin{prop}\label{p.2}
Suppose $Q_0 > 0$ and $Q$ is determined by 
\lbeq{e.prox4}
Q(Z):= Q_0 + \frac{\sigma}{2} \|Z - Z_0\|_F^2,
\eeq
where $\|.\|_F$ is the Frobenius norm. Then $Q$ is a prox-function with the center $Z_0$ and satisfies (\ref{e.proin}). Moreover, the subproblem (\ref{e.Eeta}) corresponded to this $Q$ is explicitly solved by
\lbeq{e.h1}
U = Z_0 - E(\gamma,H)^{-1} \sigma^{-1} H
\eeq
with
\lbeq{e.E1}
E(\gamma,H)
=\frac{-\xi_1+\sqrt{\xi_1^2 + 4 Q_0 \xi_2}}{2Q_0}
=\frac{2 \xi_2}{\xi_1+\sqrt{\xi_1^2 + 4 Q_0 \xi_2}},
\eeq
where $\xi_1 = \gamma + Tr(H^TZ_0)$ and $\xi_2 = \left( \frac{1}{2} \sigma^{-2} - \sigma^{-1} \right) \|H\|_F^2$.
\end{prop}
\begin{proof}
By the features of the Frobenius norm, it is clear that $Z = Z_0$ is the minimizer of (\ref{e.Qinf}). Also, $Q$ is continuously differentiable $g_Q(x) = \sigma (X - Z_0)$. Thus, we have 
\begin{equation*}
\begin{split}
Q(X) + \langle g_Q(X), Z - X \rangle + \frac{\sigma}{2} \|Z - X\|_F^2 &= 
Q_0 + \frac{\sigma}{2} \langle X - Z_0, X - Z_0 \rangle + \sigma \langle X - Z_0, Z - X \rangle + \frac{\sigma}{2} \|Z - X\|_F^2\\
& = Q_0 + \frac{\sigma}{2} \langle (X - Z_0), Z - Z_0 \rangle + \frac{\sigma}{2} \langle (Z - Z_0), Z - X \rangle\\
& = Q_0 + \frac{\sigma}{2} \langle (Z - Z_0), Z - Z_0 \rangle \\
& =  Q(Z).
\end{split}
\end{equation*} 
This means that $Q$ is strongly convex by the convexity parameter $\sigma$, i.e., (\ref{e.prox3}) is a prox-function. Also by (\ref{e.strc}) and (\ref{e.Qinf}), the condition (\ref{e.proin}) is clearly hold. 

Now, we consider the subproblem (\ref{e.Eeta}) and drive its minimizer. Let us to define the function $E_{\gamma,H}:\mathcal{X} \rightarrow \Rz$ using
\[
E_{\gamma,H}(Z):=-\frac{\gamma+\<H,Z\>}{Q(Z)},
\]
where it is continuously differentiable and attains its supremum at $Z = U$. Therefore, in the point $Z = U$, we obtain 
\lbeq{e.etae2}
E(\gamma,H)Q(U) = -\gamma - \<H,U\>.
\eeq
It follows from $g_Q(Z) = \sigma (Z - Z_0)$, (\ref{e.etae2}) and the first-order necessary optimality condition that
\[ 
E(\gamma,H) \sigma (U - Z_0) + H = 0
\]
or equivalently
\[
U = Z_0 - E(\gamma,H)^{-1} \sigma^{-1} H.
\]
Setting $e = E(\gamma,H)$ and substituting it into (\ref{e.etae2}) straightforwardly lead to

\begin{equation*}
\begin{split}
e Q(U) + \gamma + \langle H,U\rangle & = e \left(Q_0 + \frac{1}{2} \| e^{-1} \sigma^{-1} H \|_F^2\right) + \gamma + \langle H,Z_0 - e^{-1} \sigma^{-1} H \rangle \\
& = e \left( Q_0 + \frac{1}{2} e^{-2} \sigma^{-2} \|H\|_F^2 \right) + \gamma + Tr(H^T Z_0) + e^{-1} \sigma^{-1} \|H\|_F^2  \\
& = Q_0 e^2 + \left( \gamma + Tr(H^T Z_0) \right) e + \left( \frac{1}{2} \sigma^{-2} - \sigma^{-1}\right) \|H\|_F^2\\
& = Q_0 e^2 + \xi_1 e + \xi_2 = 0,
\end{split}
\end{equation*}
where $\xi_1 = \gamma + Tr(H^T Z_0)$ and $\xi_2 = \left( \frac{1}{2} \sigma^{-2} - \sigma^{-1} \right) \|H\|_F^2$. Solving this quadratic equation and selecting the larger solution lead to 
\[
E(\gamma,H)
=\frac{-\xi_1+\sqrt{\xi_1^2 + 4 Q_0 \xi_2}}{2Q_0}
=\frac{2 \xi_2}{\xi_1+\sqrt{\xi_1^2 + 4 Q_0 \xi_2}}.
\]
This completes the proof. \qed
\end{proof}
\section{Numerical experiments}\label{s.pra}
This section reports extensive numerical experiments of the proposed algorithm dealing with some famous inverse problems in applications compared with some popular algorithms and software packages. We assess the convergence speed of OSGA and some widely used first-order methods for solving such practical problems. The experiment is divided into two parts. We first consider some imaging problems and give results of experiments with OSGA and some state-of-the-art solvers. In the second part, we adapt some smooth optimization algorithms for nonsmooth problems and report encouraging results and comparisons.

All codes are written in MATLAB, and default values of parameters for the algorithms and packages are used. For OSGA, we used the prox-function (\ref{e.prox1}) with $Q_0 = \frac{1}{2} \|x_0\|_2 + \epsilon$, where $\epsilon$ is the machine precision in MATLAB. We also set the following parameters for OSGA:
\begin{equation*}
\delta = 0.9;~~ \alpha_{max} = 0.7;~~ \kappa = \kappa' = 0.5;~~ \Psi_{\mathrm{target}} = - \infty. 
\end{equation*}
All implementations are executed on a Toshiba Satellite Pro L750-176 laptop with Intel Core i7-2670QM Processor and 8 GB RAM.
\subsection{Image restoration}\label{s.imag}
Image reconstruction, also called image restoration, is one of the classical linear inverse problems dating back to the 1960s, see for example \cite{AndH,BerB}. The goal is to reconstruct images from some observations. Let $y \in \mathcal{U}$ be a noisy indirect observation of an image $x \in \mathcal{X}$ with an unknown noise $\delta \in \mathcal{U}$. Suppose that $\mathcal{A}:\mathcal{X} \rightarrow \mathcal{U}$ is a linear operator. To recover the unknown vector $x$, we use the linear inverse model (\ref{e.inv2}). As discussed, according to the ill-conditioned or singular nature of the problem, some regularizations are needed. Recall (\ref{e.tikh}), (\ref{e.BPD}) or the more general model (\ref{e.genf1}) to derive a suitable solution for the corresponding inverse problem and refer to \cite{KauN1,KauN2,NeuI} for more regularization options. The strategy employed for solving such penalized optimization problems depends on features of objective functions and regularization terms like differentiability or nondifferentiability and convexity or nonconvexity. The quality of the reconstructed image highly depends on these features. A typical model for image reconstruction consists of a smooth objective function penalized by smooth or nonsmooth regularization penalties. These penalties are selected based on expected features of the recovered image. 

Suppose that the known image $x \in \mathcal{X}$ is represented by $x = \mathcal{W} \tau$ in some basis domain, e.g., wavelet, curvelet, or time-frequency domains, where the operator $\mathcal{W}:\mathcal{U} \rightarrow \mathcal{V}$ can be orthogonal or any dictionary. 
There are two main strategies for restoring the image $x$ from an indirect observation $y$, namely {\bf \emph{synthesis}} and {\bf \emph{analysis}}. To recover a clean image, the synthesis approach reconstructs the image by solving the following minimization problem
\begin{equation}\label{e.syn}
\min_{\tau \in \mathcal{U}} ~~ \D \frac{1}{2} \|\mathcal{A}\mathcal{W} \tau - y\|_2^2 + \lambda \varphi(\tau),\\
\end{equation}
where $\varphi$ is a convex regularizer. If $\tau^*$ is an optimizer of (\ref{e.syn}), then an approximation of the clean image $x$ will be reconstructed using $x^* = \mathcal{W} \tau^*$. Alternatively, the analysis strategy recovers the image by solving 
\begin{equation}\label{e.ana}
\min_{x \in \mathcal{X}} ~~ \D \frac{1}{2} \|\mathcal{A}(x) - y\|_2^2 + \lambda \varphi(x).
\end{equation}
Similarity of the problems (\ref{e.syn}) and (\ref{e.ana}) suggests that they are strongly related. These two approaches are equivalent in some assumptions, while variant in others, for more details see \cite{ElaMR}. Although various regularizers like $l_p$-norm are popular in image restoration, it is arguable that the best known and frequently employed regularizer in analysis approaches for image restoration is {\bf \emph{total variation}} (TV), which will be discussed later in this section. In the sequel, we report our numerical experiments based on the analysis strategy.

The pioneering work on total variation as a regularizer for image denoising and restoration was proposed by {\sc Rudin}, {\sc Osher}, and {\sc Fatemi} in \cite{OshRF}. Total variation regularizations are able to restore discontinuities of images and recover the edges. This makes TV appealing and widely used in applications. TV is originally defined in the infinite-dimensional Hilbert space, however for the digital image processing it is interested to consider it in a finite-dimensional space like $\Rz^n$ of pixel values on a two-dimensional lattice. In practice, some discrete versions of TV are favoured to be used. Two standard choices of discrete TV-based regularizers, namely {\bf \emph{isotropic total variation}} and {\bf \emph{anisotropic total variation}}, are popular in signal and image processing, defined by
\[
    \bary{lll}
    \|X\|_{ITV} &=& \sum_i^{m-1} \sum_j^{n-1} \sqrt{(X_{i+1,j} - X_{i,j})^2+(X_{i,j+1} - X_{i,j})^2 }\\
            &+& \sum_i^{m-1} |X_{i+1,n} - X_{i,n}| + \sum_i^{n-1} |X_{m,j+1} - X_{m,j}|,
    \eary
\]
and
    \[
    \bary{lll}
    \|X\|_{ATV} &=& \sum_i^{m-1} \sum_j^{n-1} \{|X_{i+1,j} - X_{i,j}| + |X_{i,j+1} - X_{i,j}| \}\\
            &+& \sum_i^{m-1} |X_{i+1,n} - X_{i,n}| + \sum_i^{n-1} |X_{m,j+1} - X_{m,j}|,
    \eary
\]
for $X \in \Rz^{m\times n}$, respectively. The image restoration problem with the discrete TV-based regularization can be formulated by
\begin{equation}\label{e.TV}
\min_{X \in \Rz^{m \times n}} ~~ \D \frac{1}{2} \| \mathcal{A}(X) - Y\|_F^2 + \lambda~ \varphi(X),\\
\end{equation}
where $\mathcal{A}: \Rz^{m \times n} \rightarrow \Rz^{r \times m}$ is a linear operator, $\varphi = \|.\|_{ITV}$ or $\varphi = \|.\|_{ATV}$ and $\|X\|_F = \sqrt{\mathrm{Tr}~ {XX^T}}$ is the Frobenius norm in which $\mathrm{Tr}$ stands for the trace of a matrix. 

In what follows, we consider some prominent classes of imaging problems, more specifically denoising, inpainting and deblurring, and conduct our tests on numerous images. \\

\subsubsection{Denoising with total variation}\label{s.den}   
Image denoising is one of the most fundamental image restoration problems aiming to remove noise from images, in which the noise comes from some resources like sensor imperfection, poor illumination, or communication errors. While this task itself is a prominent imaging problem, it is important as a component in other processes like deblurring with the proximal mapping. The main objective is to denoise images while the edges  are preserved. To do so, various models and ideas have been proposed. Among all of the regularizers, total variation proposed a model preserving discontinuities (edges). 

The particular case of (\ref{e.TV}) when $\mathcal{A}$ is the identity operator, is called {\bf \emph{denoising}} and is very popular with total variation regularizers. Although the original total variation introduced in \cite{OshRF} is an infinite-dimensional continuous constrained problem, we here consider the following unconstrained discrete finite-dimensional version 
\lbeq{e.pro}
X_{Y,\lambda} = \argmin_{X \in \Rz^{m \times n}} \frac{1}{2} \|X - Y\|_F^2 + \lambda~ \varphi(X),
\eeq
where $\varphi = \|.\|_{ITV}$ or $\varphi = \|.\|_{ATV}$. It is clear that both $\|.\|_{ITV}$ and $\|.\|_{ATV}$ are nonsmooth semi-norms and  the proximal operator (\ref{e.pro}) cannot be solved explicitly for a given $Y$. Rather than solving the problem (\ref{e.pro}) explicitly, it has been approximately solved by iterative processes similar to that proposed by {\sc Chambolle} in \cite{Cha}. We note that $\|.\|_{ITV}$ and $\|.\|_{ATV}$ are nonsmooth and convex, and their subdifferentials are available. Therefore, OSGA can be employed for solving (\ref{e.pro}).  

We consider the isotropic version of problem (\ref{e.pro}) where $Y$ denotes a noisy image and $X_{Y,\lambda}$ is denoised approximation of the original $X$ corresponding to a regularization parameter $\lambda$. We run IST, TwIST \cite{BioF} and FISTA \cite{BecT2,BecT3} in order to compare their performances with OSGA's results. For IST, TwIST and FISTA, we use original codes and only adding more stopping criteria are needed in our comparisons. We set $\lambda = 0.05$ for all algorithms. As IST, TwIST and FISTA need to solve their subproblems iteratively, we consider three different cases by limiting the number of internal iterations to $chit = 5$, $10$ and $20$ iterations. Let us consider the recovery of the $1024 \times 1024$ Pirate image by IST, TwIST, FISTA and OSGA to verify the results visually and in more details. The noisy image is generated by using the MATLAB internal function $\mathtt{awgn}$ with $\mathrm{SNR} = 15$, and the algorithms are stopped after 50 iterations. The results are summarized in Figures 1 and 2. 

To see details of this experiment, we compare the results of implementations in different measures in Figure 2. The performances are measured by some relative errors for points and function values and also so-called signal-to-noise improvement (ISNR). Here, ISNR and PSNR are defined by
\lbeq{e.isnr}
\mathrm{ISNR} = 20 \log_{10} \left( \frac{\|Y - X_0\|_F}{\|X - X_0\|_F} \right)~~~ \mathrm{and} ~~~ \mathrm{PSNR} = 20 \log_{10} \left( \frac{\sqrt{mn}}{\|X - X_0\|_F} \right),
\eeq
where $X_0$ denotes the $m \times n$ clean image, $Y$ is the observed image and pixel values are in $[0, 1]$. Generally, these ratios measure the quality of the restored image $X$ relative to the blurred or noisy observation $Y$. From Figure 2, it is observed that OSGA outperforms IST, TwIST and FISTA with respect to all considered measures. Figure 2, shows that the denoised image looks good for all considered algorithms, where the best function value and the second best PSNR are attained by OSGA requiring 25.46 seconds to be stopped. However, it can be seen that by increasing the number of Chambolle's iterations, the running time is dramatically increased for IST, TwIST and FISTA, however, a better function value or PSNR (compared with OSGA) is not attained.

\begin{figure}[p]\label{f.den1}
\centering
\subfloat[][$rel.~ 1~ vs.~ iterations~ (chit = 5)$]{\includegraphics[width=4.9cm]{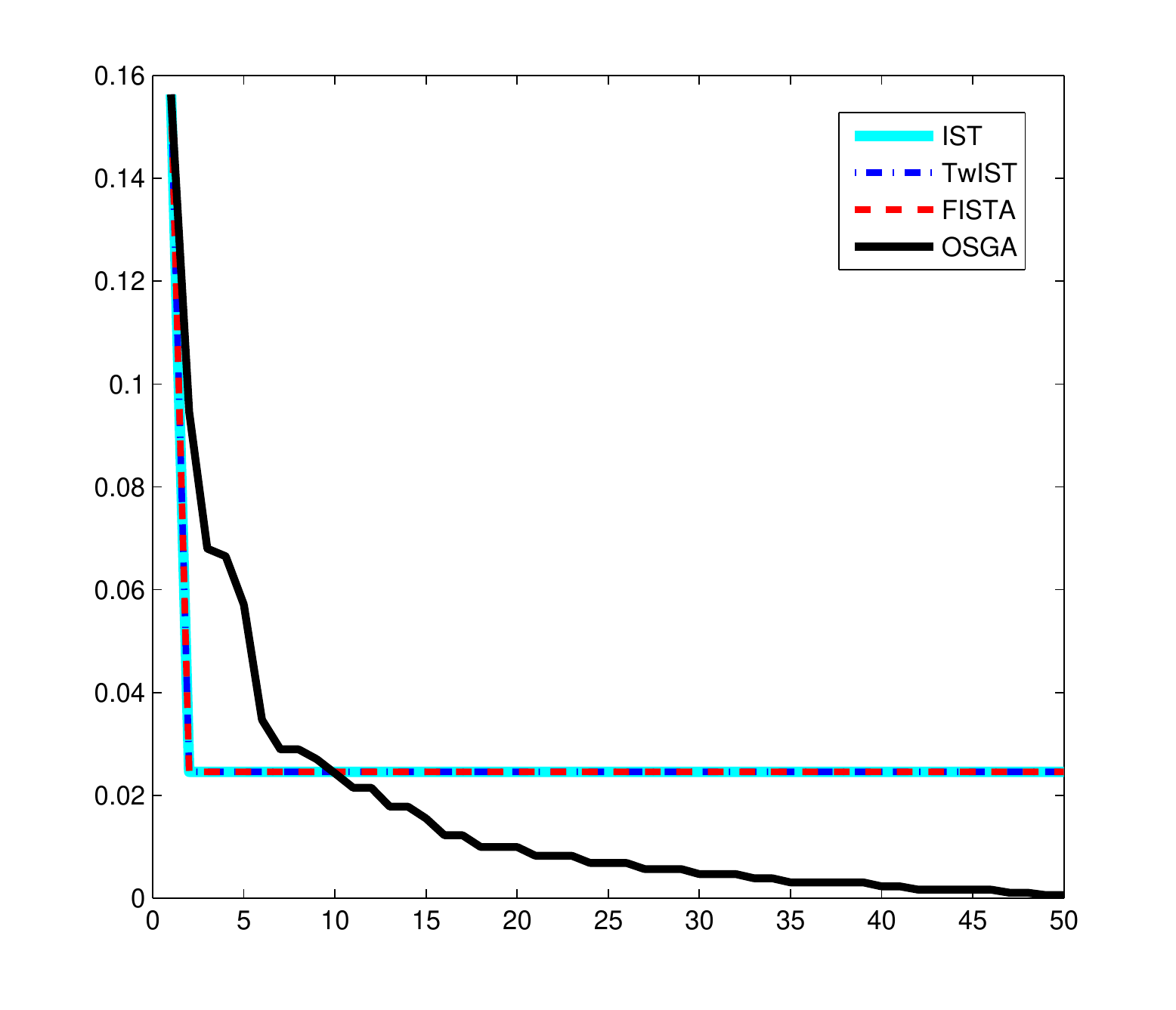}}%
\qquad
\subfloat[][$rel.~ 1~ vs.~ iterations~ (chit = 10)$]{\includegraphics[width=4.9cm]{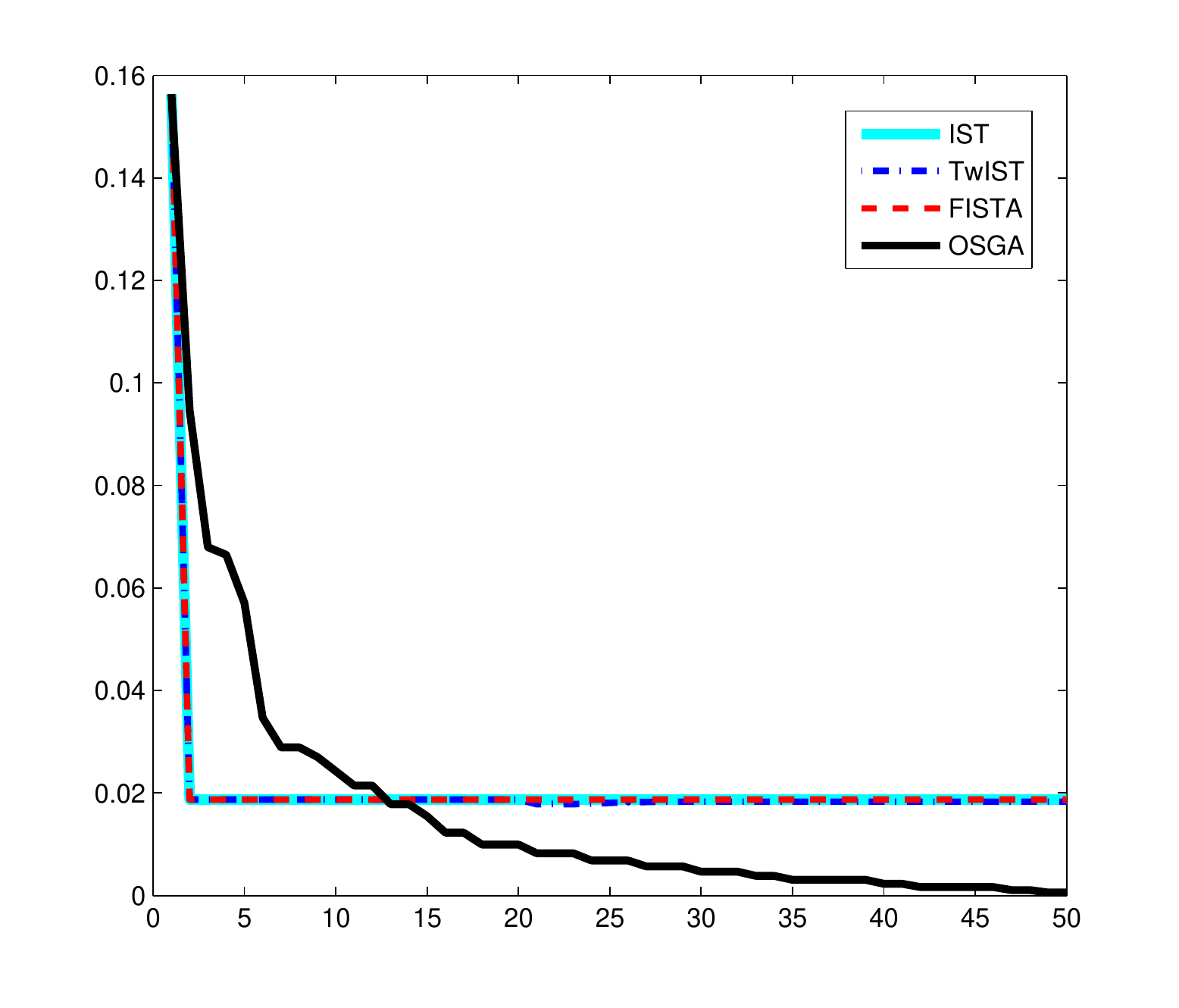}}
\qquad
\subfloat[][$rel.~ 1~ vs.~ iterations~ (chit = 20)$]{\includegraphics[width=4.9cm]{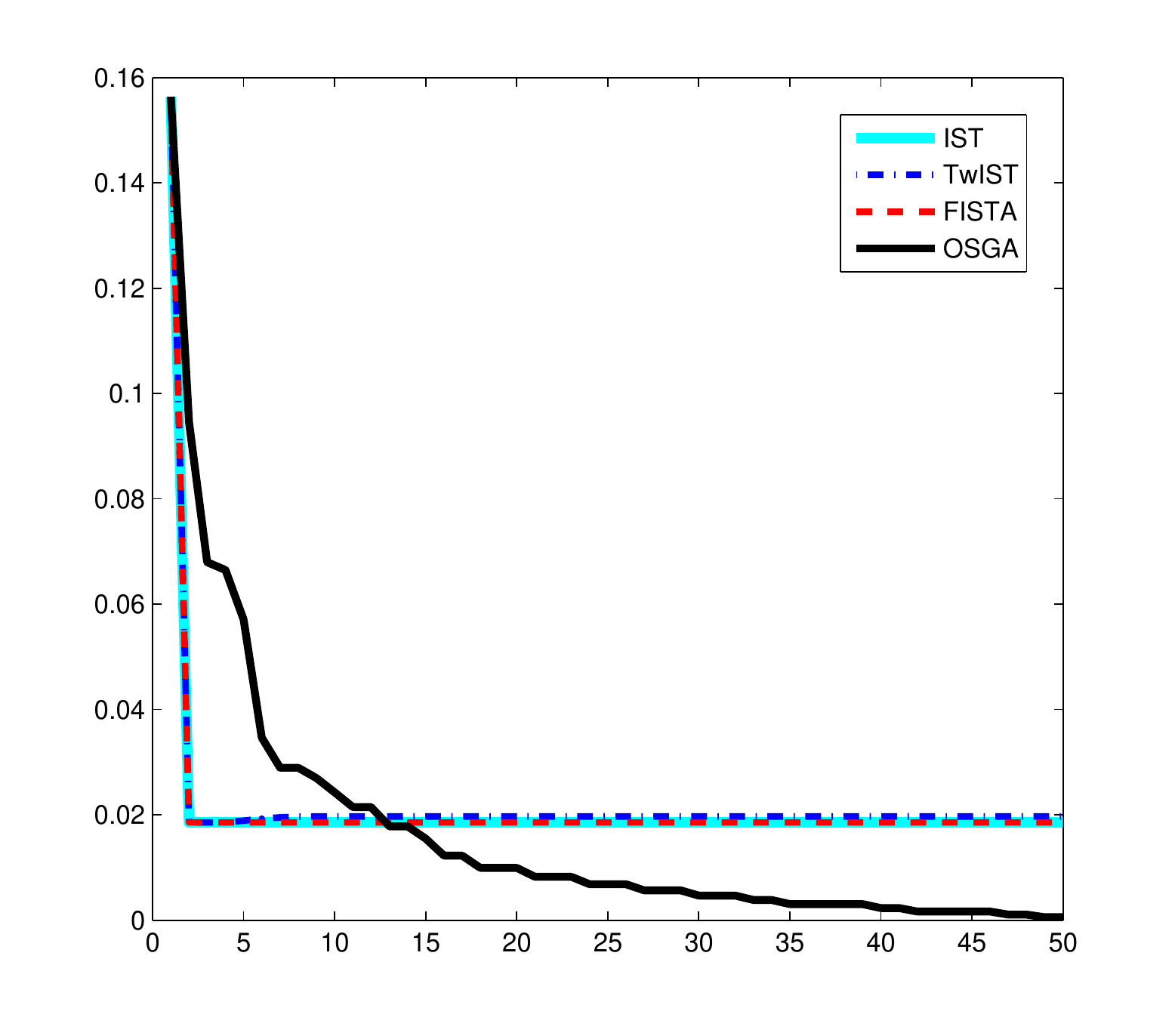}}%
\qquad
\subfloat[][$rel.~ 2~ vs.~ time~ (chit = 5)$]{\includegraphics[width=4.9cm]{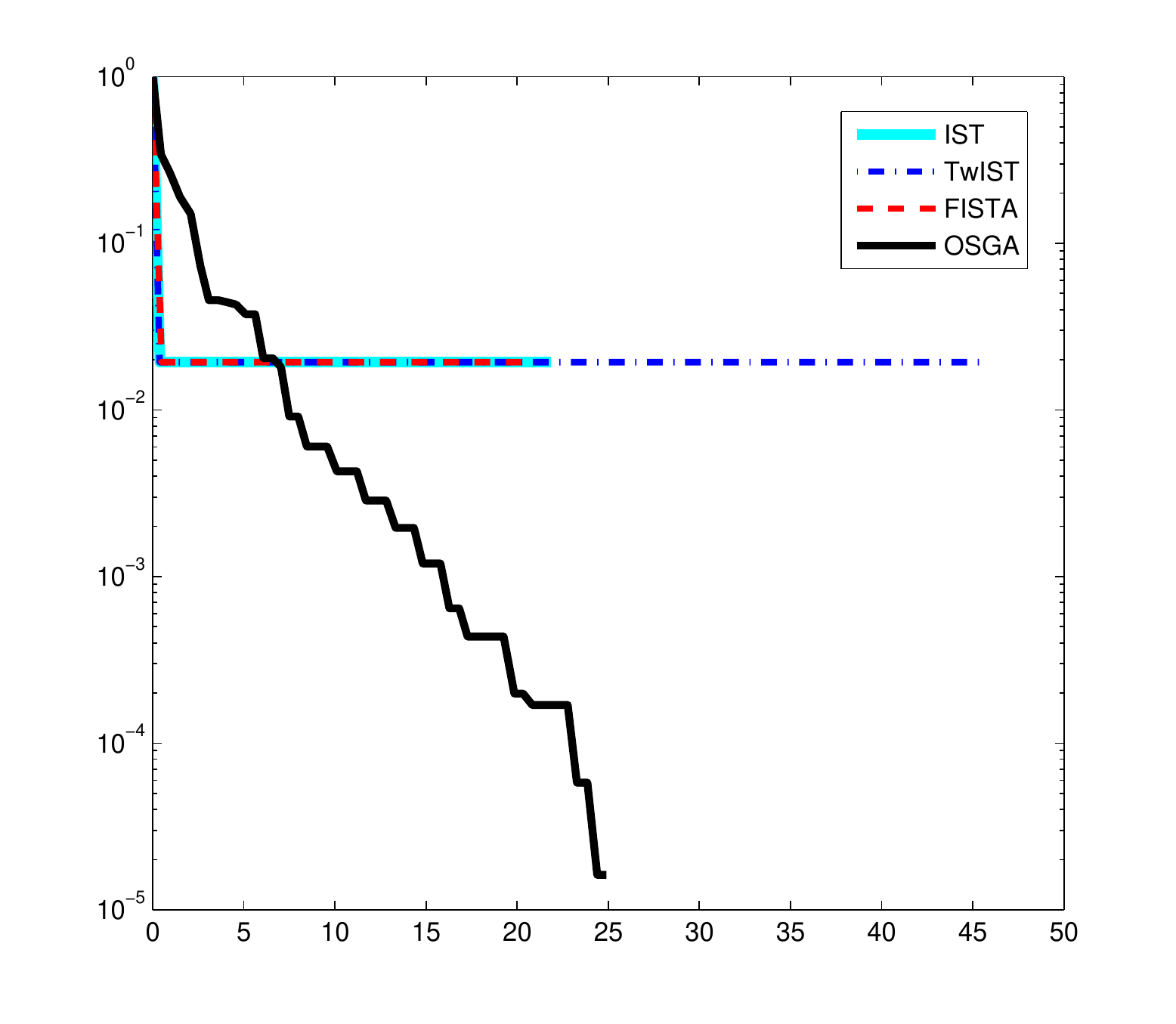}}
\qquad
\subfloat[][$rel.~ 2~ vs.~ time~ (chit = 10)$]{\includegraphics[width=4.9cm]{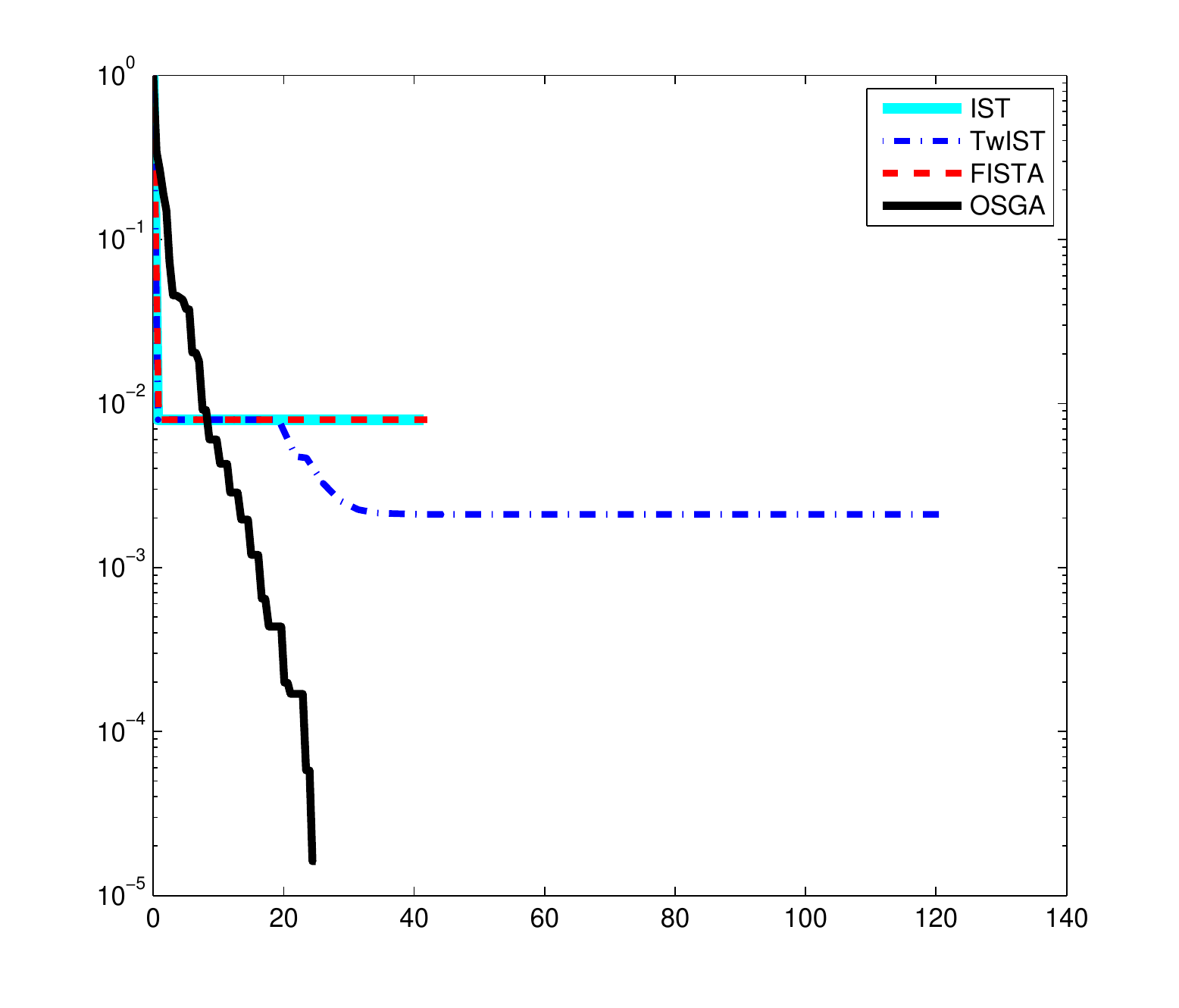}}%
\qquad
\subfloat[][$rel.~ 2~ vs.~ time~ (chit = 20)$]{\includegraphics[width=4.9cm]{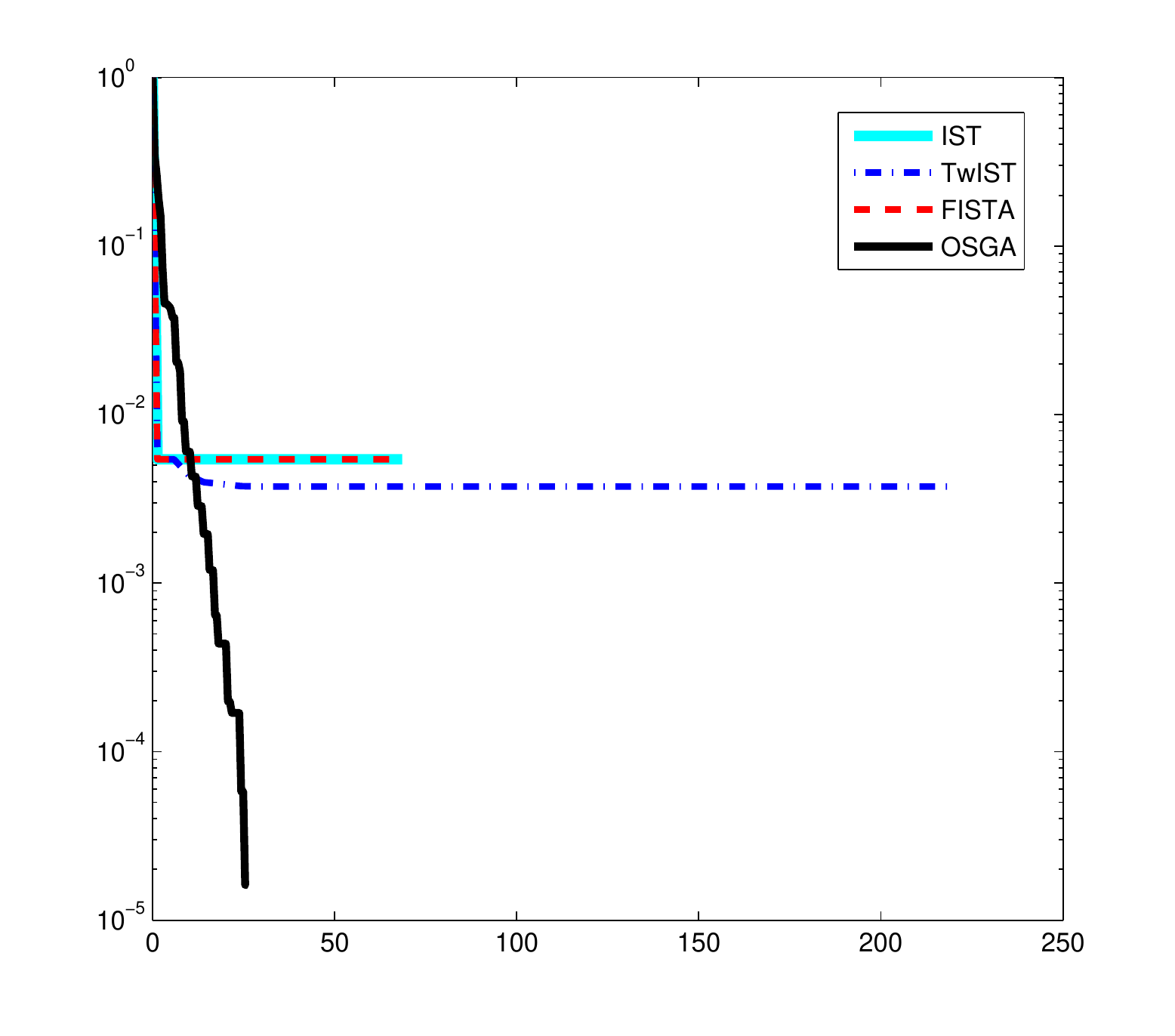}}
\qquad
\subfloat[][$rel.~ 2~ vs.~ iterations~ (chit = 5)$]{\includegraphics[width=4.9cm]{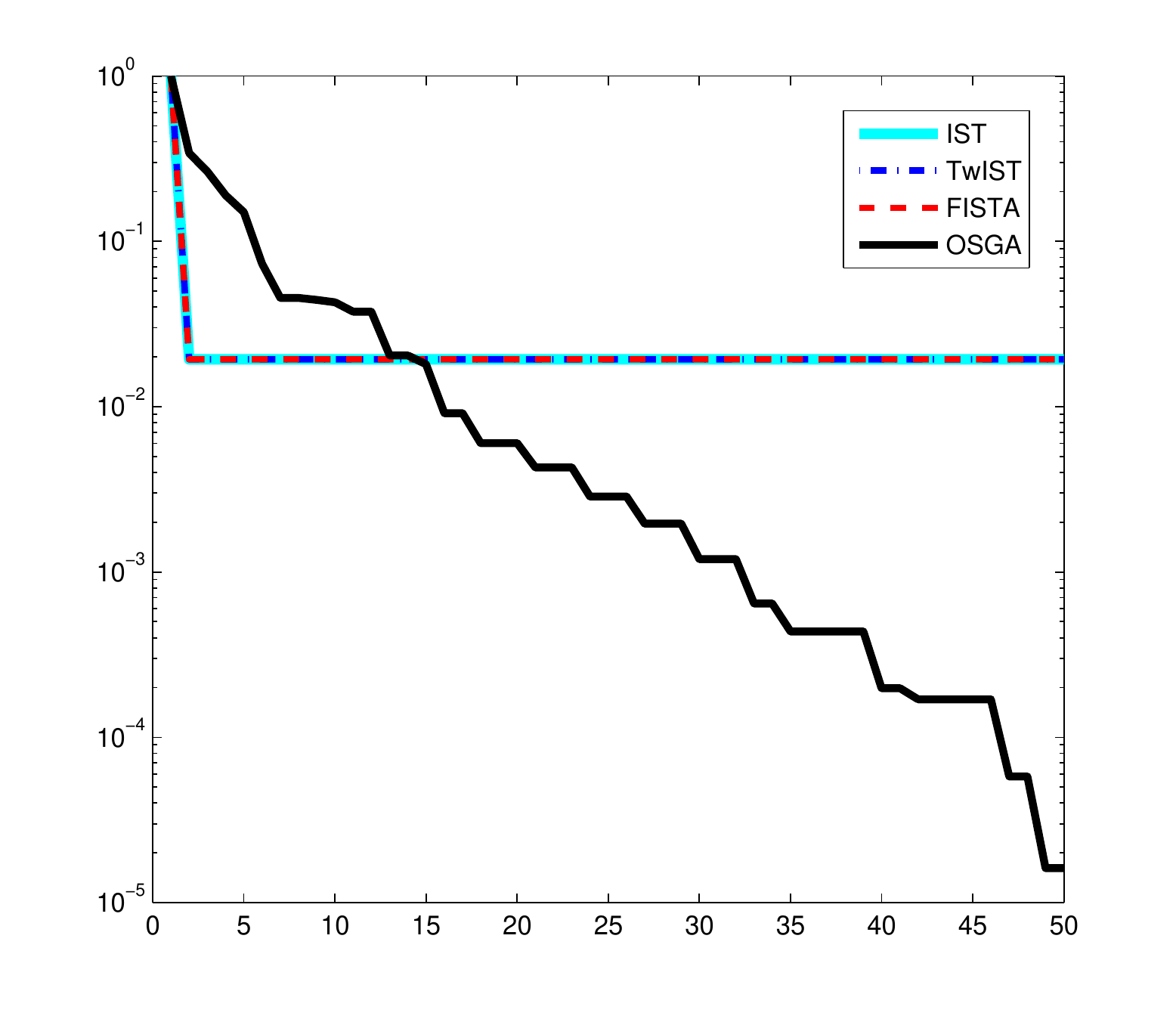}}
\qquad
\subfloat[][$rel.~ 2~ vs.~ iterations~ (chit = 10)$]{\includegraphics[width=4.9cm]{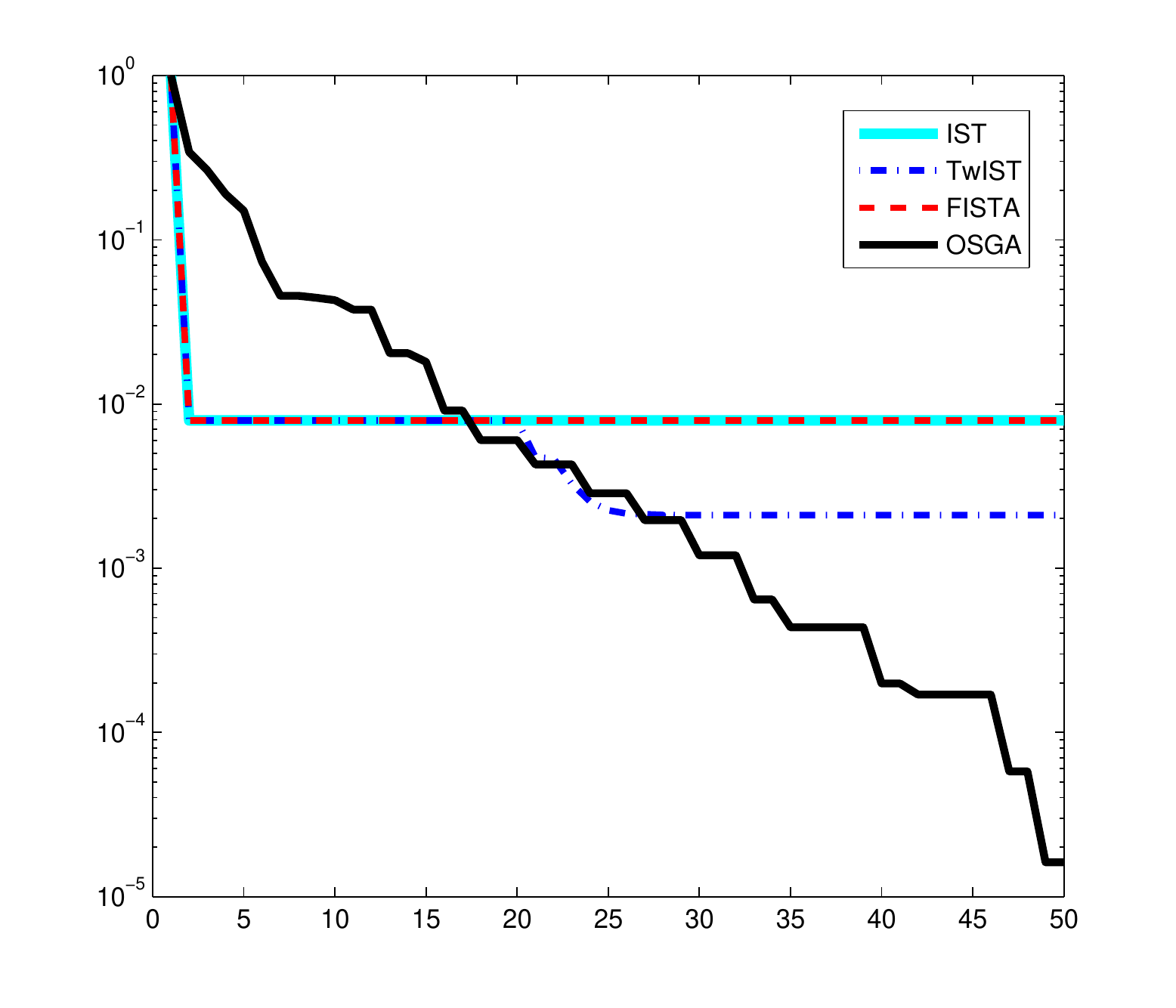}}%
\qquad
\subfloat[][$rel.~ 2~ vs.~ iterations~ (chit = 20)$]{\includegraphics[width=4.9cm]{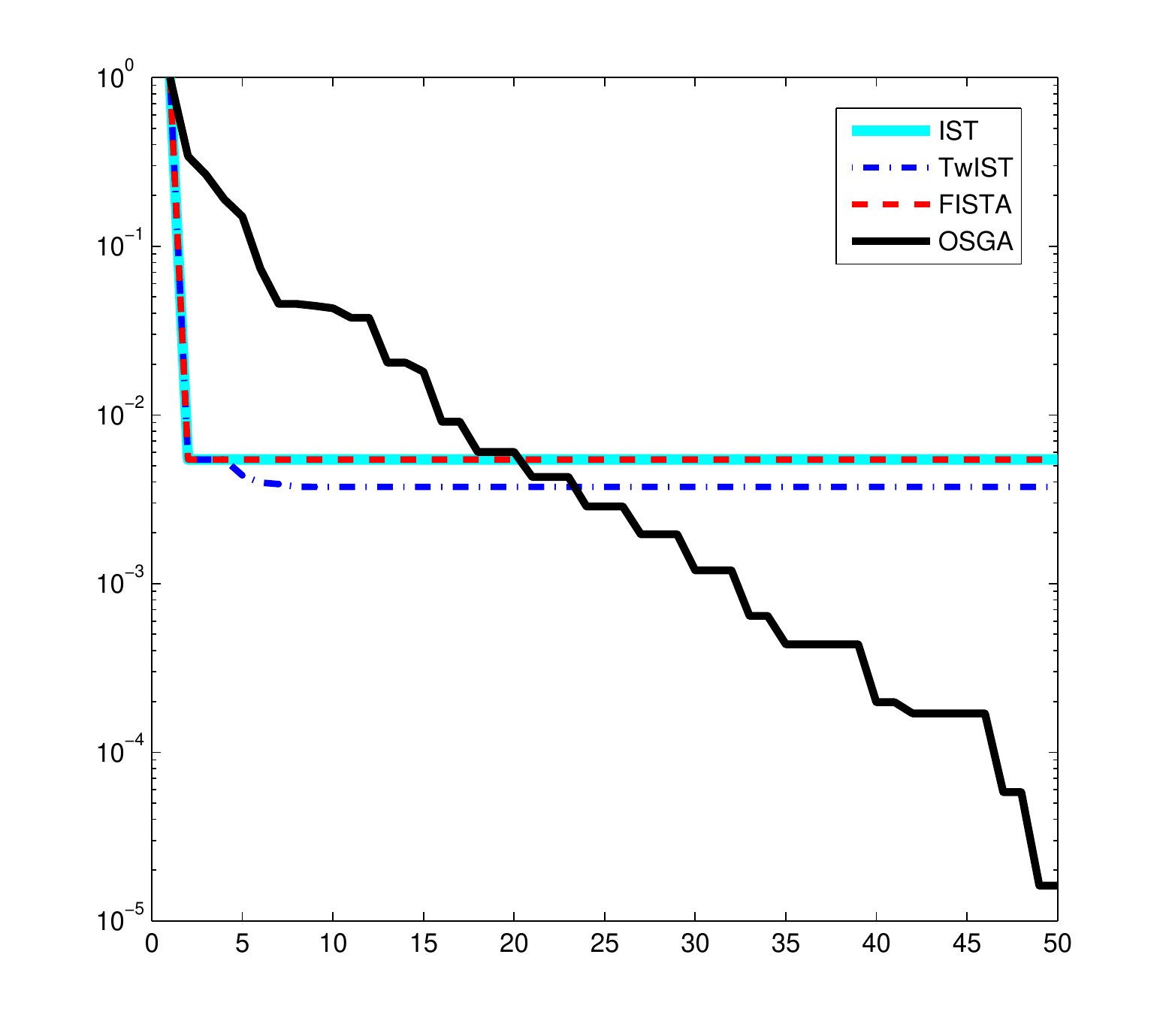}}
\qquad
\subfloat[][$ISNR~ vs.~ iterations~ (chit = 5)$]{\includegraphics[width=4.9cm]{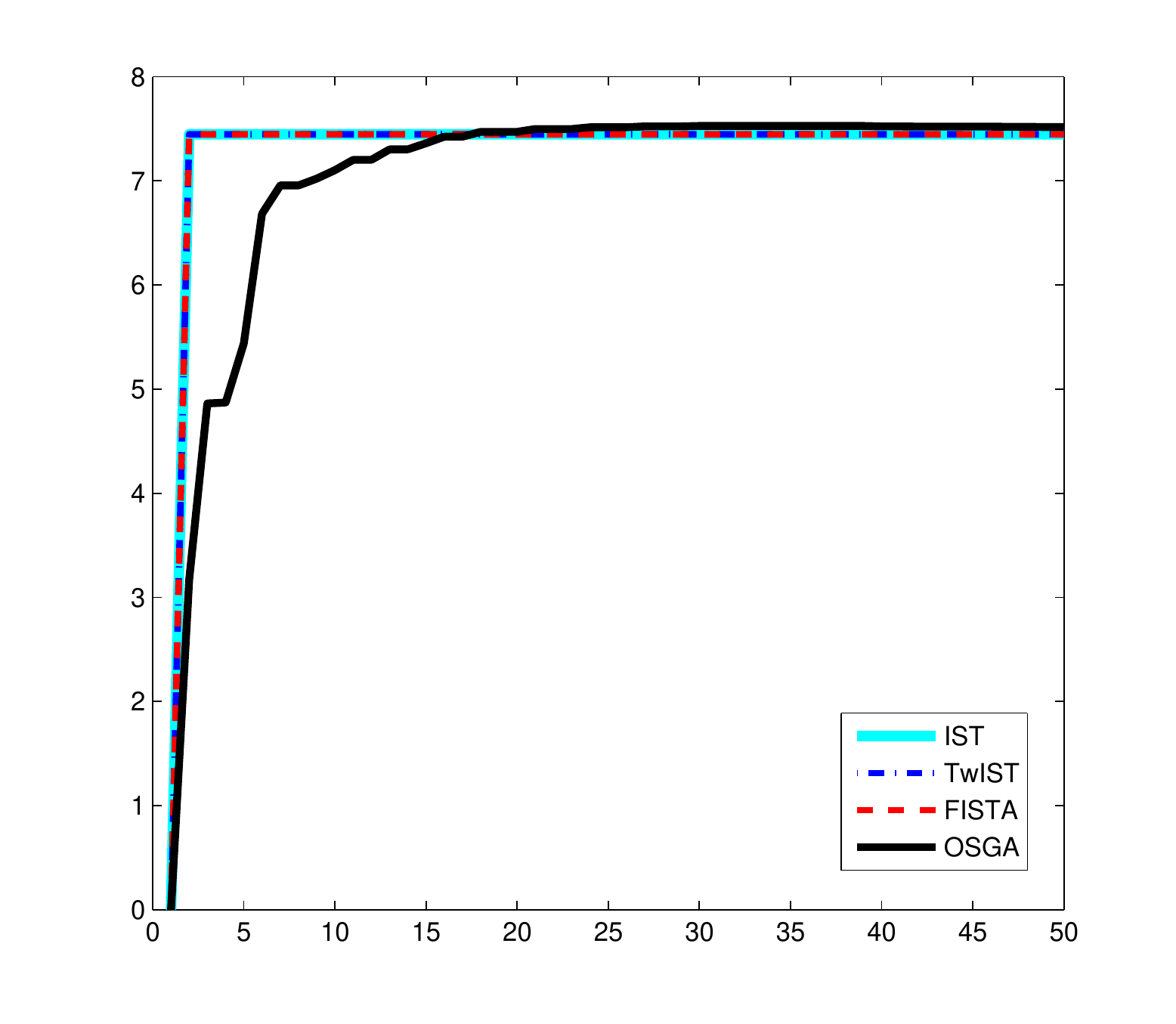}}
\qquad
\subfloat[][$ISNR~ vs.~ iterations~ (chit = 10)$]{\includegraphics[width=4.9cm]{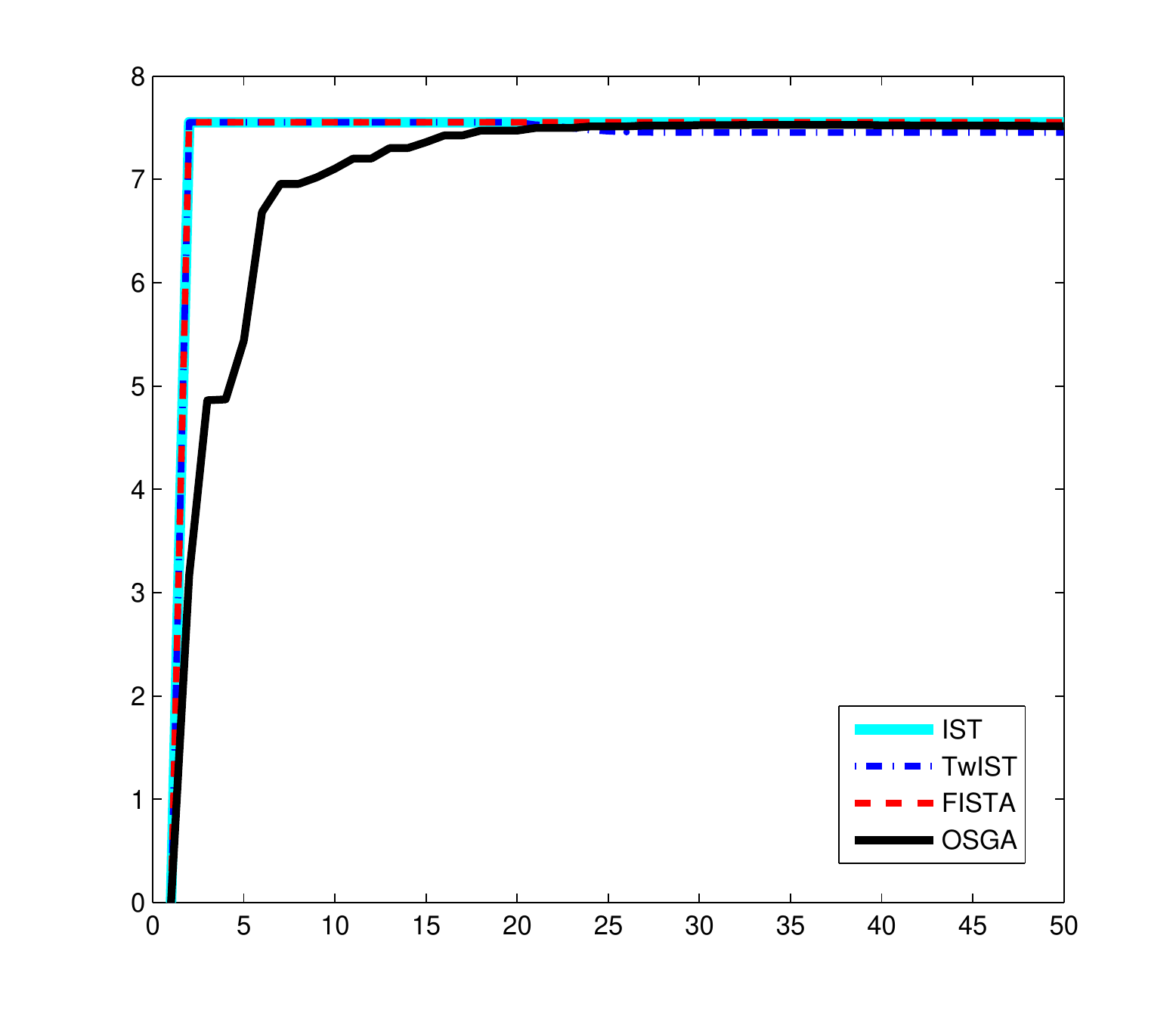}}%
\qquad
\subfloat[][$ISNR~ vs.~ iterations~ (chit = 20)$]{\includegraphics[width=4.9cm]{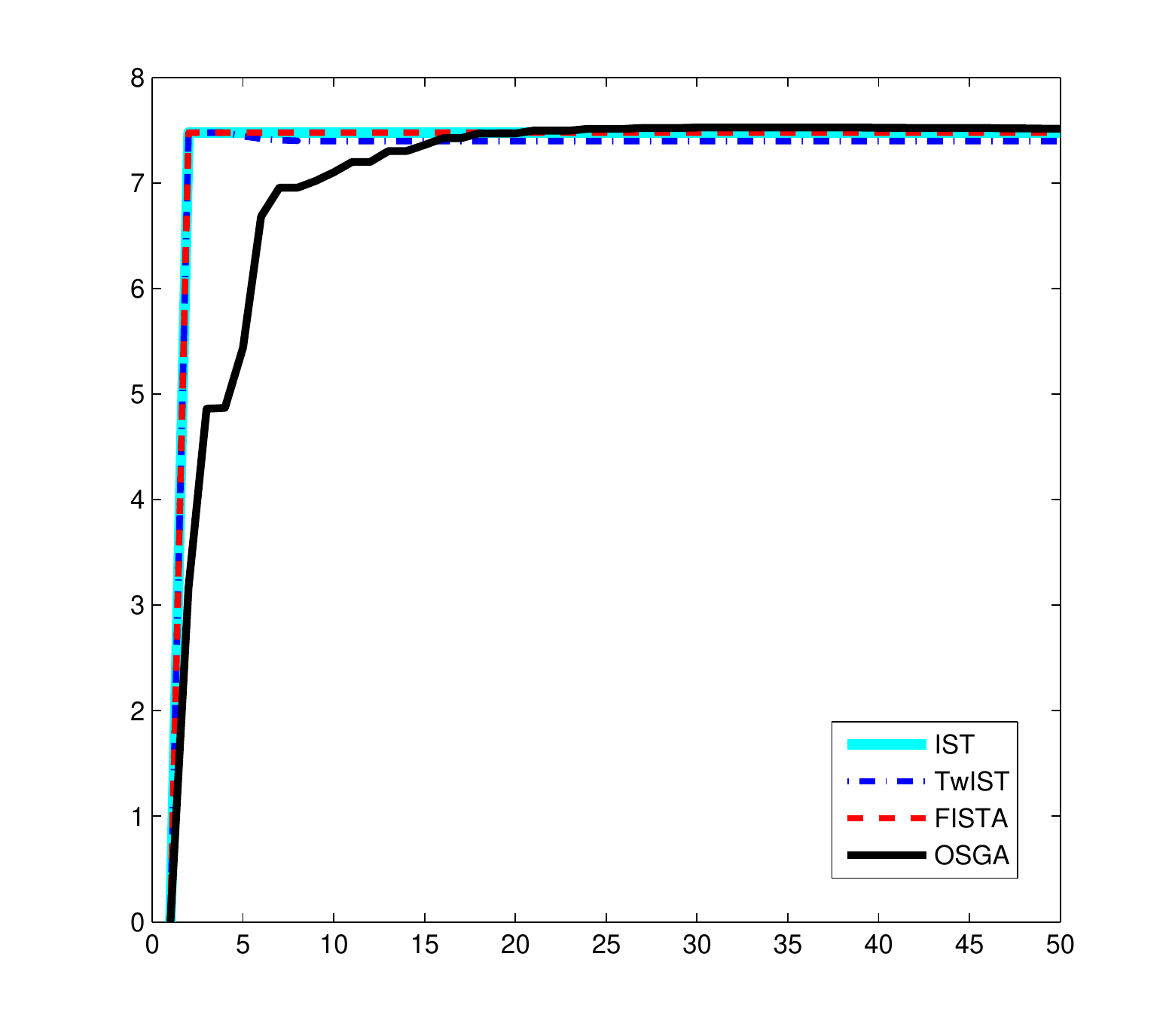}}
\caption{A comparison among IST, TwIST, FISTA and OSGA for denoisinging the $1024 \times 1024$ Pirate image when they stopped after 50 iterations. Each column stands for comparisons with a different number of Chambolle's iterations. While the first row denotes the relative error  $rel.~ 1 := \|x_k - x^*\|_2/\|x^*\|_2 $ of points
versus iterations, the second row shows the relative error $rel.~ 2 := (f_k - f^*)/(f_0 - f^*)$ of function values versus time. The third row illustrates $rel.~ 2$ versus iterations and fourth row stands for ISNR (\ref{e.isnr}) versus iterations.}%
\label{fig:cont}
\end{figure}

\begin{figure}[p]\label{f.den2}
\centering
\subfloat[][Original image]{\includegraphics[width=4.9cm]{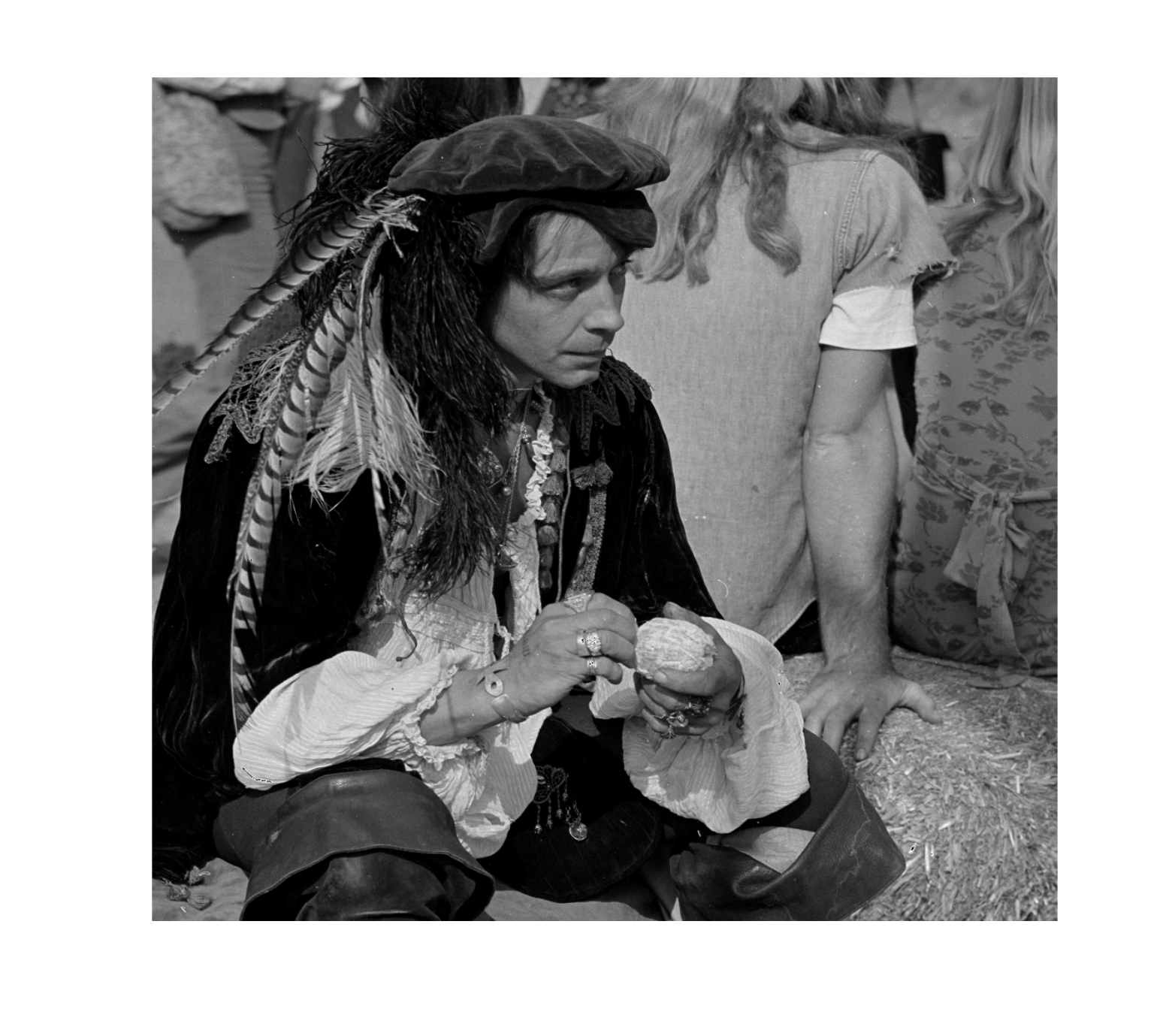}}%
\qquad
\subfloat[][Noisy image]{\includegraphics[width=4.9cm]{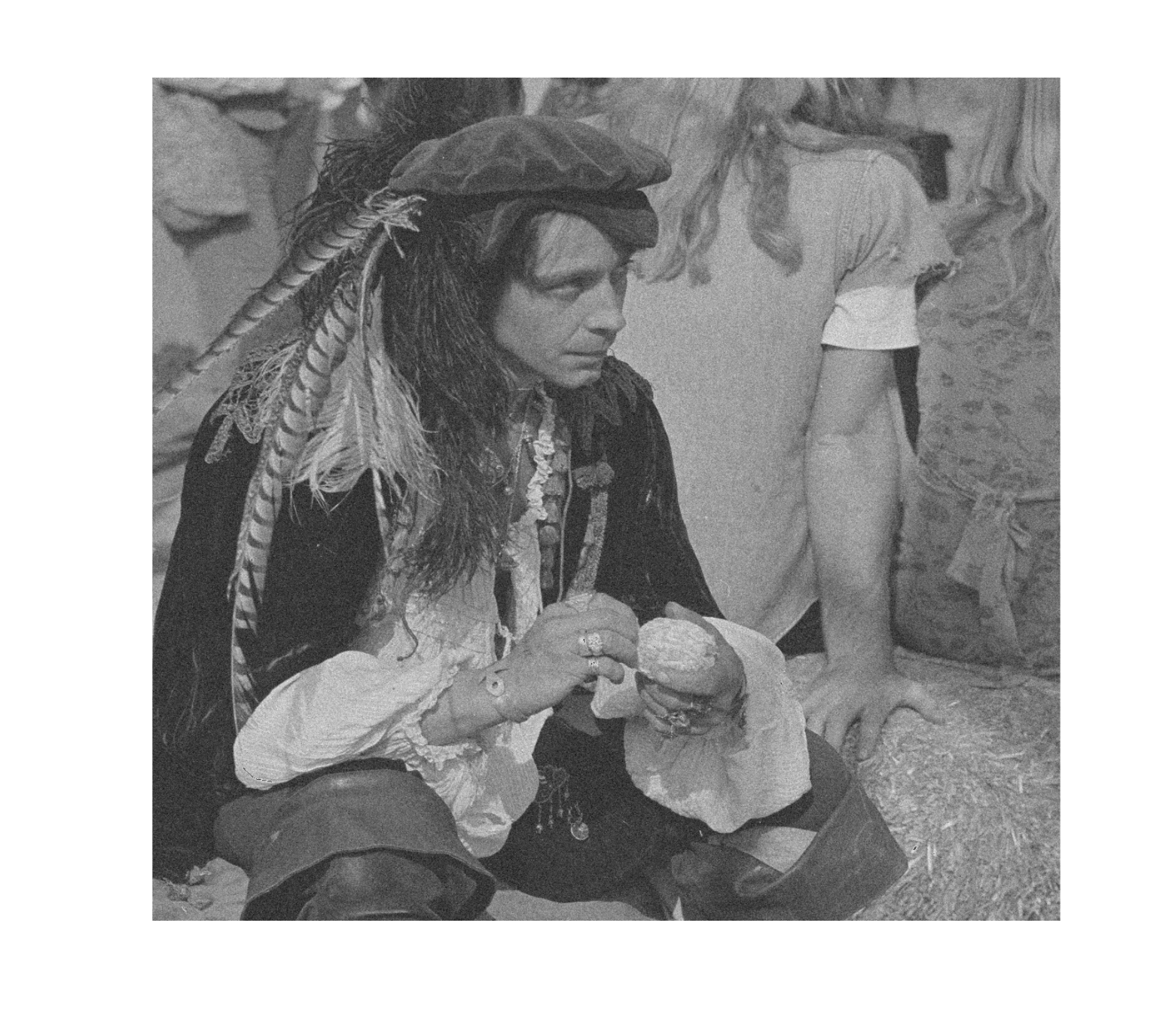}}
\qquad
\subfloat[][OSGA: $\mathrm{F} = 3721.52$, $\mathrm{PSNR} = 30.12$, $\mathrm{T} = 25.46$]{\includegraphics[width=4.9cm]{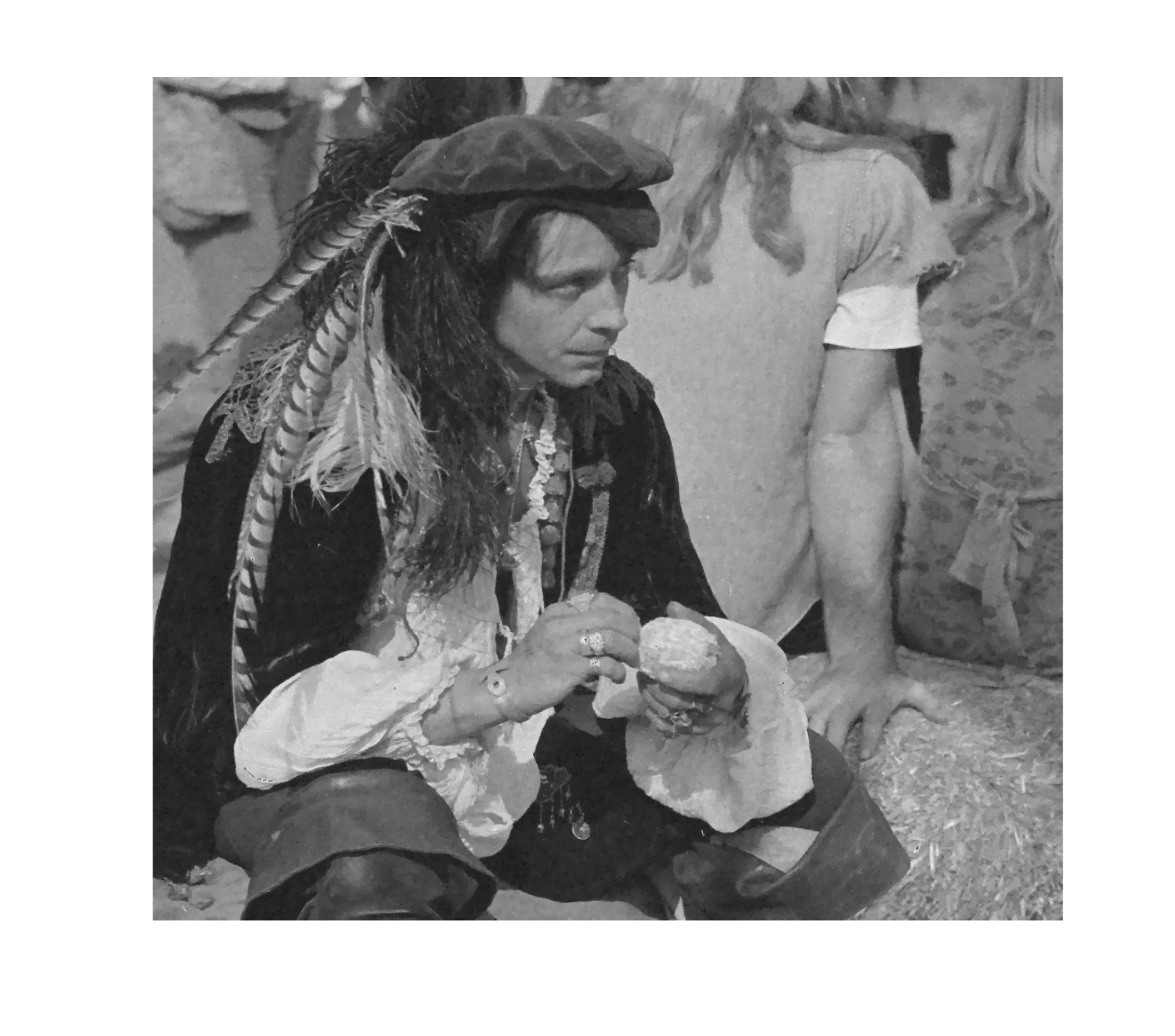}}%
\qquad
\subfloat[][IST~ ($chit = 5$): $\mathrm{F} = 3791.57$, $\mathrm{PSNR} = 30.05$, $\mathrm{T} = 22.76$]{\includegraphics[width=4.9cm]{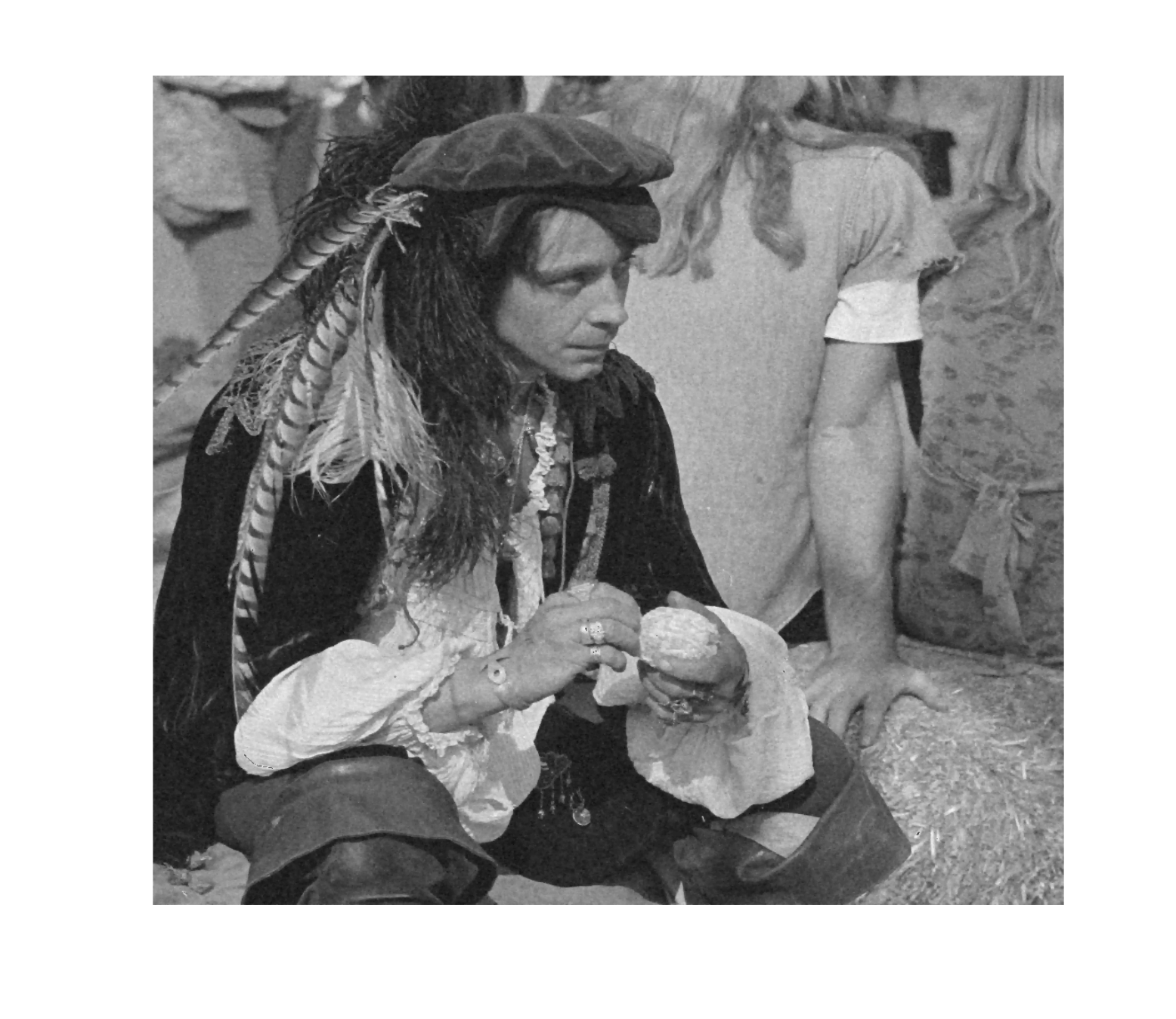}}
\qquad
\subfloat[][IST~ ($chit = 10$): $\mathrm{F} = 3750.28$, $\mathrm{PSNR} = 30.16$, $\mathrm{T} = 43.05$]{\includegraphics[width=4.9cm]{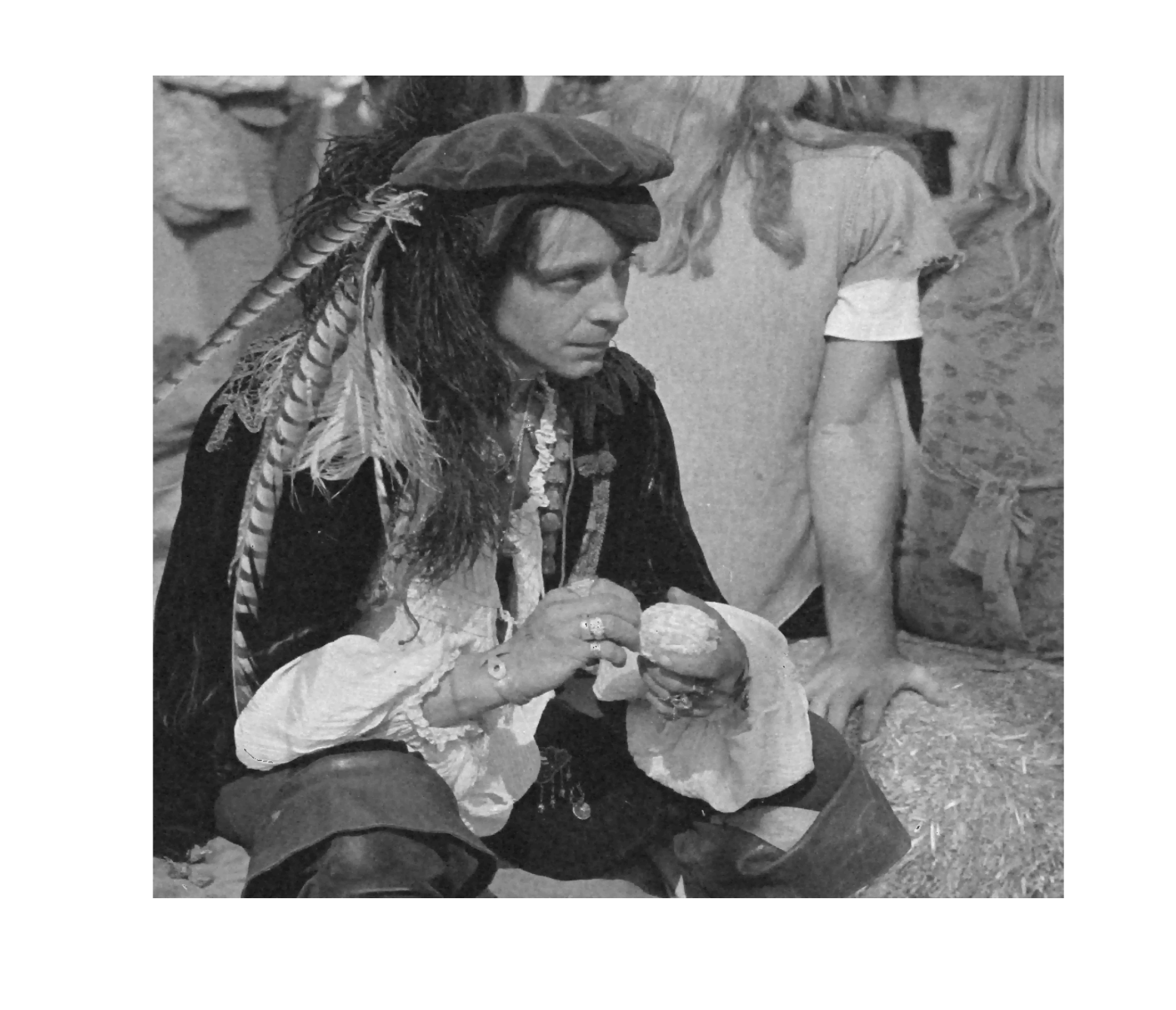}}%
\qquad
\subfloat[][IST~ ($chit = 20$): $\mathrm{F} = 3741.17$, $\mathrm{PSNR} = 30.08$, $\mathrm{T} = 71.42$]{\includegraphics[width=4.9cm]{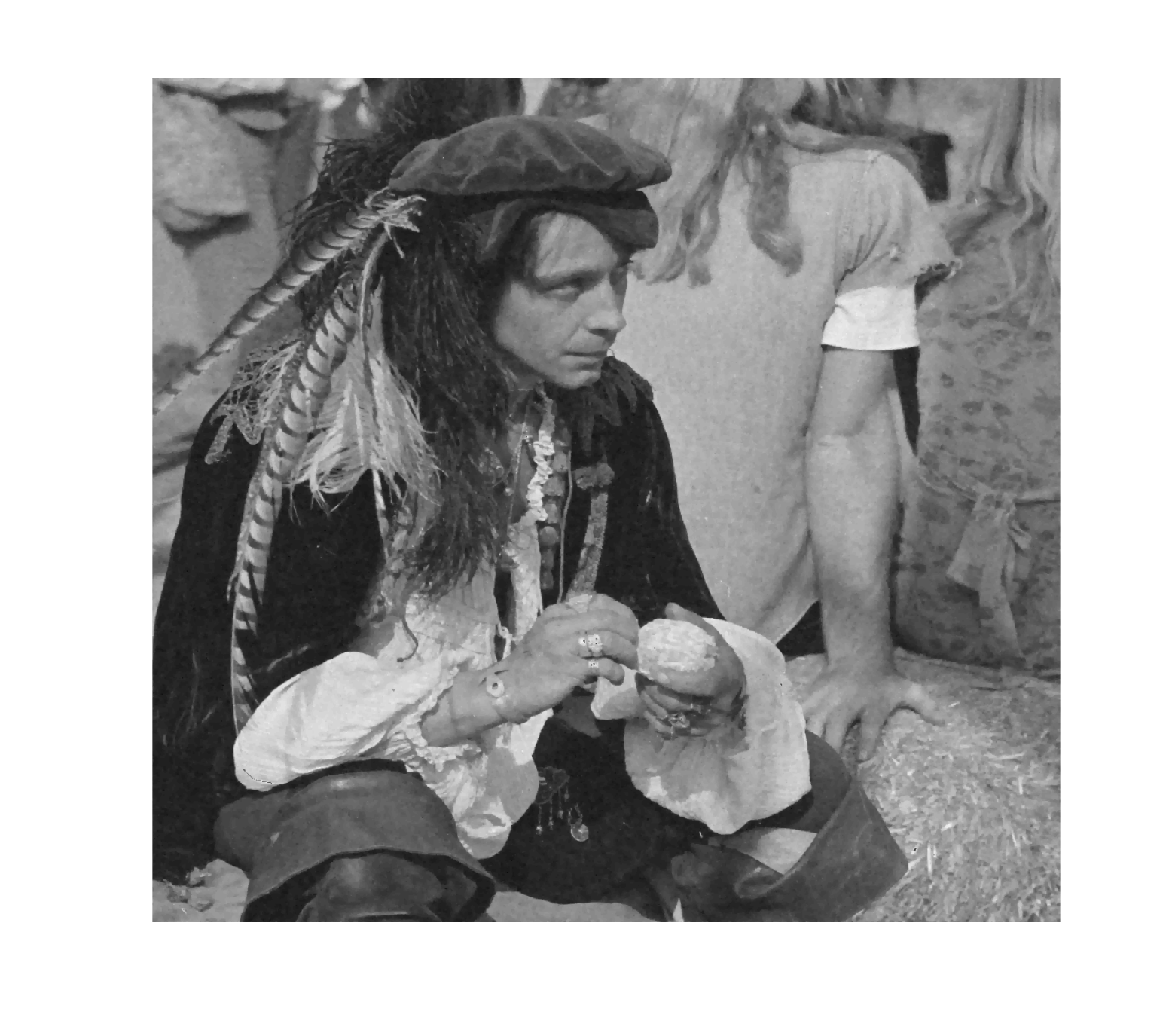}}
\qquad
\subfloat[][TwIST~ ($chit = 5$): $\mathrm{F} = 3791.57$, $\mathrm{PSNR} = 30.05$, $\mathrm{T} = 47.02$]{\includegraphics[width=4.9cm]{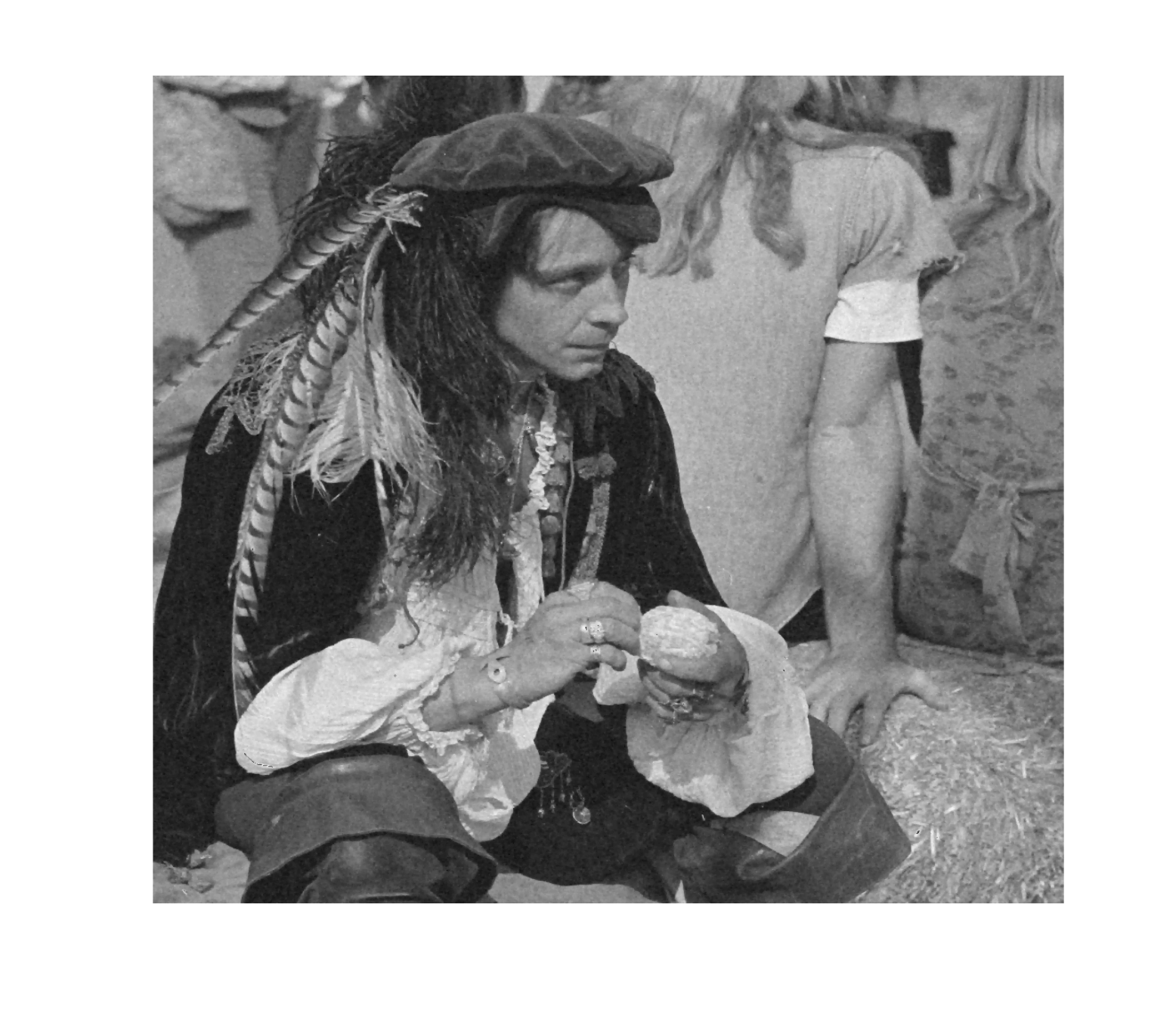}}
\qquad
\subfloat[][TwIST~ ($chit = 10$): $\mathrm{F} = 3729.11$, $\mathrm{PSNR} = 30.06$, $\mathrm{T} = 124.25$]{\includegraphics[width=4.9cm]{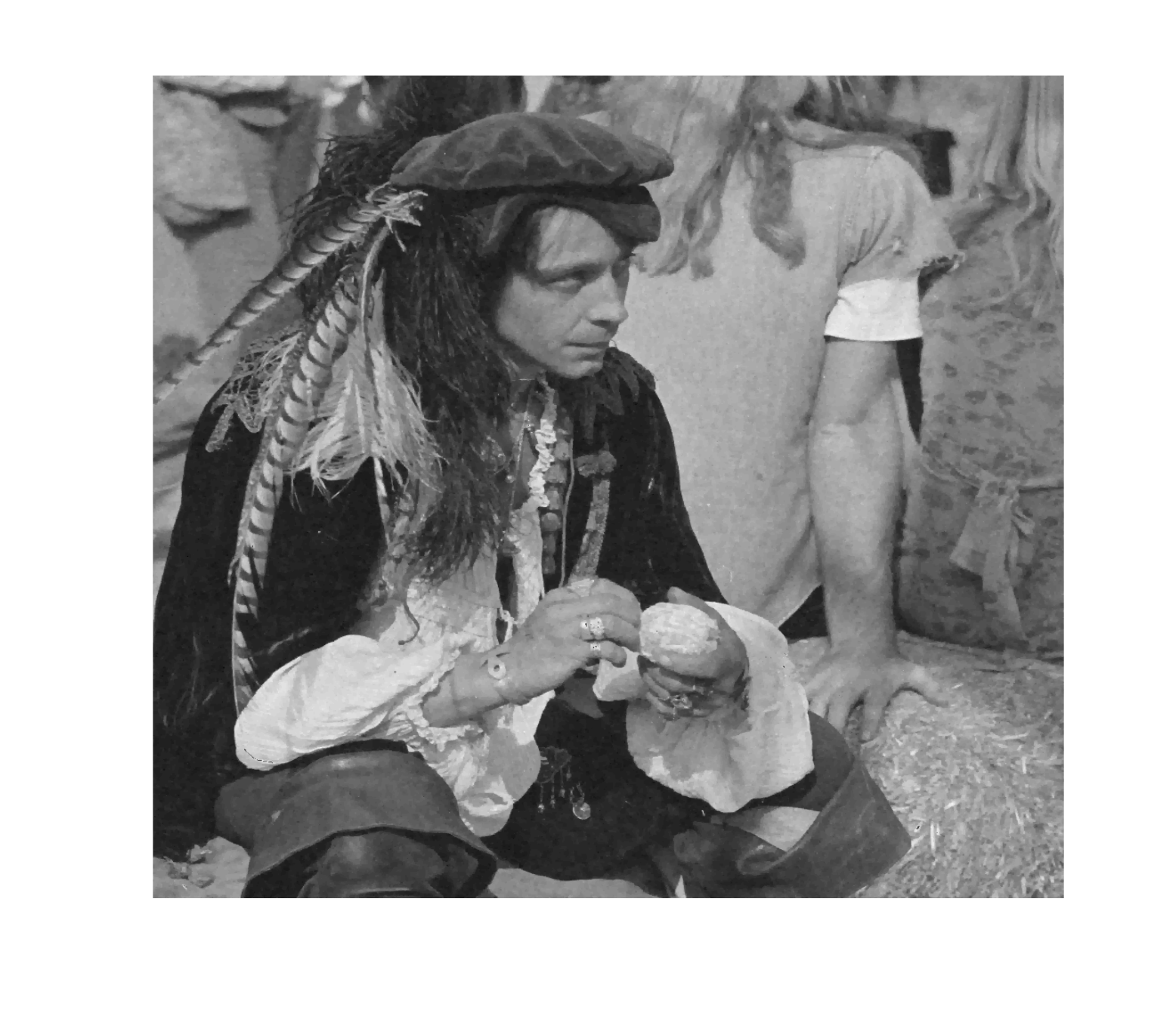}}%
\qquad
\subfloat[][TwIST~ ($chit = 20$): $\mathrm{F} = 3735.04$, $\mathrm{PSNR} = 30.00$, $\mathrm{T} = 224.44$]{\includegraphics[width=4.9cm]{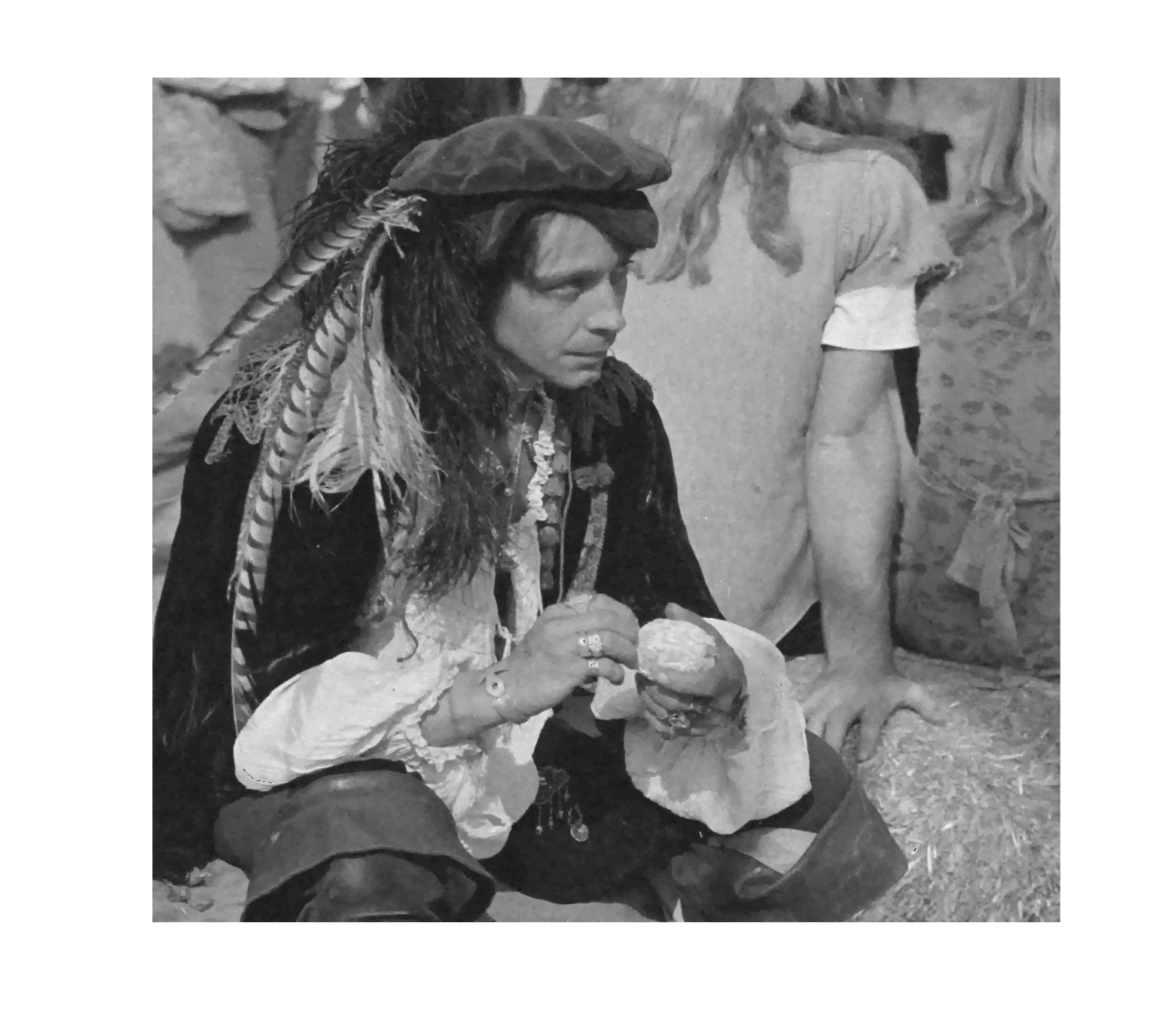}}
\qquad
\subfloat[][FISTA~ ($chit = 5$): $\mathrm{F} = 3791.57$, $\mathrm{PSNR} = 30.05$, $\mathrm{T} = 20.48$]{\includegraphics[width=4.9cm]{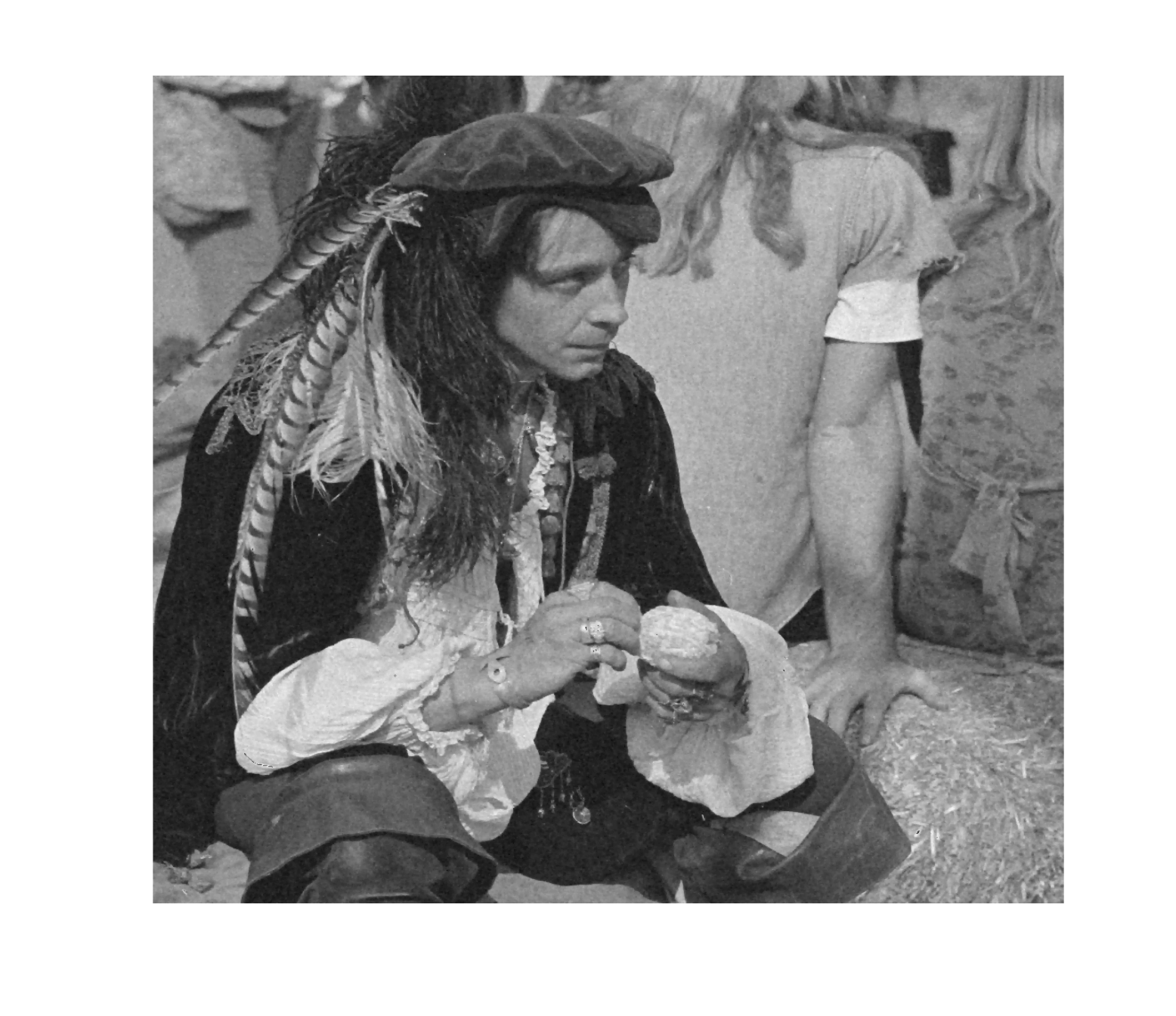}}
\qquad
\subfloat[][FISTA~ ($chit = 10$):  $\mathrm{F} = 3750.28$, $\mathrm{PSNR} = 30.16$, $\mathrm{T} = 42.04$]{\includegraphics[width=4.9cm]{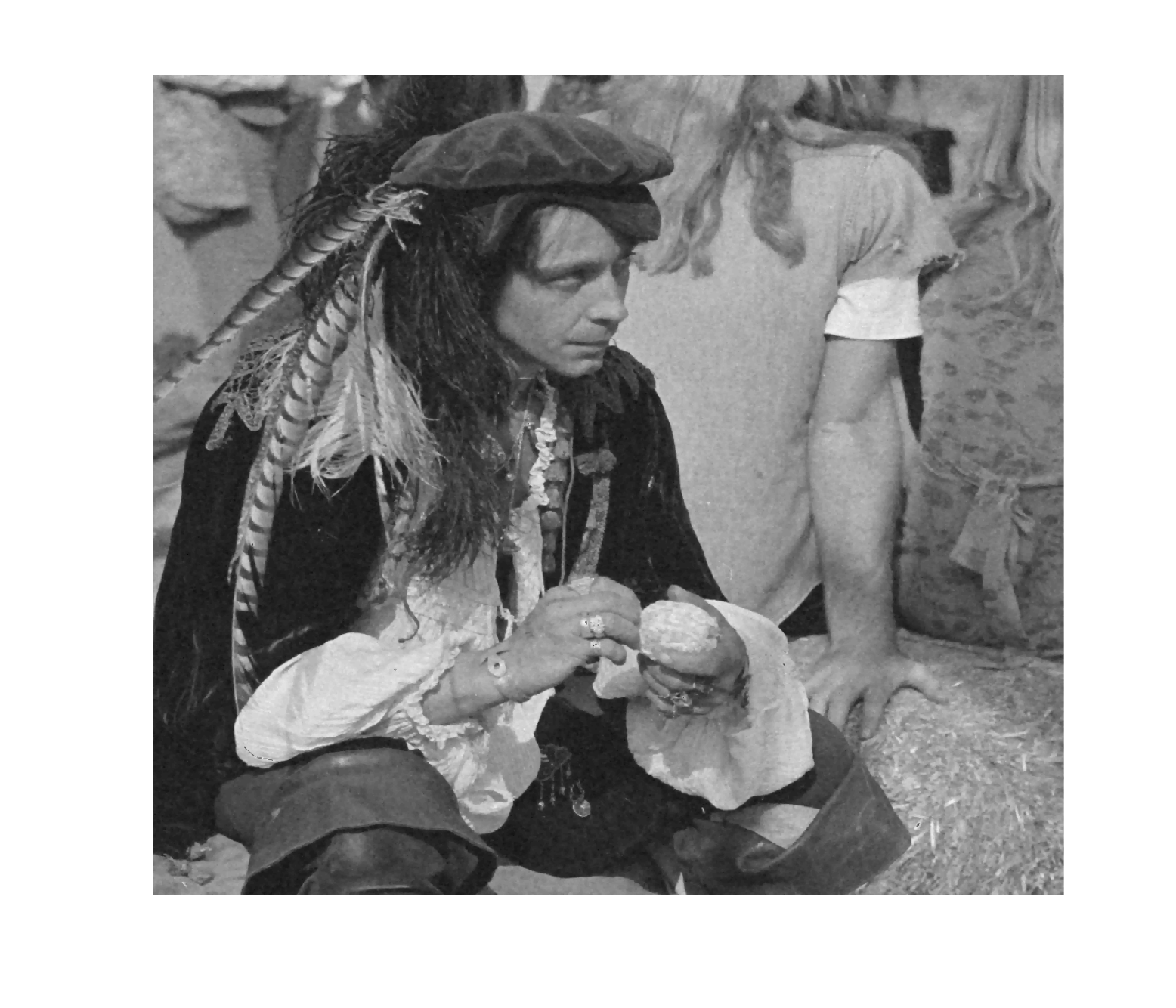}}%
\qquad
\subfloat[][FISTA~ ($chit = 20$): $\mathrm{F} = 3741.17$, $\mathrm{PSNR} = 30.08$, $\mathrm{T} = 65.88$;]{\includegraphics[width=4.9cm]{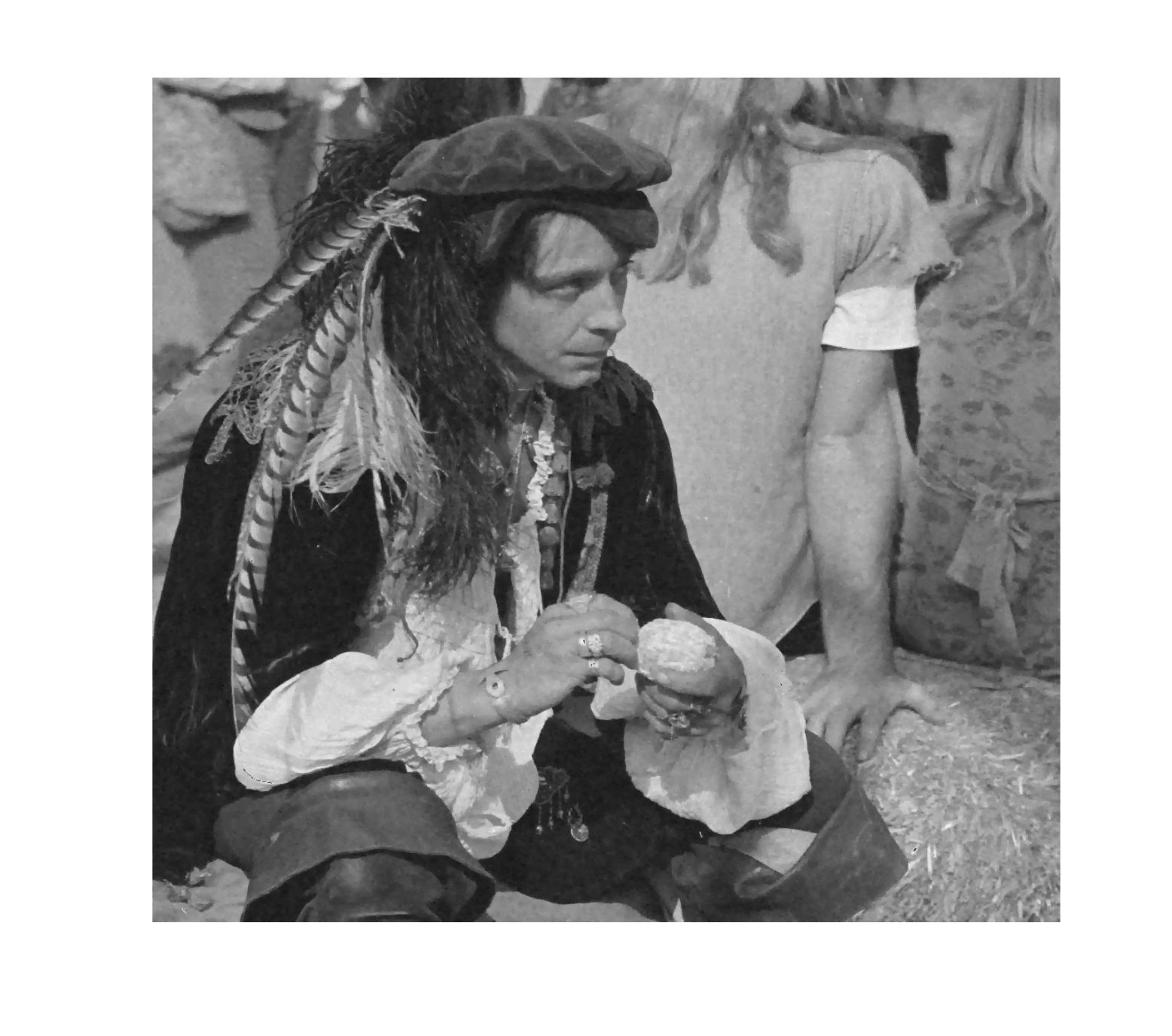}}

\caption{Denoising of the $1024 \times 1024$ Pirate image by IST, TwIST, FISTA and OSGA. Here, $\mathrm{F}$ denotes the best function value, PSNR is defined by (\ref{e.isnr}) and $\mathrm{T}$ stands for CPU time of the algorithms after 50 iterations.}%
\label{fig:cont}
\end{figure}

\subsubsection{Inpainting with isotropic total variation}\label{s.inp}  
Image inpainting is a set of techniques to fill lost or damaged parts of an image such that these modifications are undetectable for viewers who are not aware of the original image. The basic idea is originated from reconstructing of degraded or damaged paintings conventionally carried out by expert artists. This is without no doubt very time consuming and risky. Thus, the goal is to produce a modified image in which the inpainted region is unified into the image and looks normal for ordinary observers. The pioneering work on digital image inpainting was proposed by {\sc Bertalmi} et al. in \cite{BerSCB} based on nonlinear PDE. Since then lots of applications of inpainting in image interpolation, photo restoration, zooming, super-resolution images, image compression and image transmission have emerged. 

In 2006, {\sc Chan} in \cite{ChaSZ} proposed an efficient method to recover piecewise constant or smooth images by combining total variation regularizers and wavelet representations. Motivated by this work, we consider inpainting with isotropic total variation to reconstruct images with missing data. We consider the $512 \times 512$ Head CT image with 40\% missing sample. The same as previous section, IST, TwIST, FISTA and OSGA are employed to recover the image. The parameter $\lambda$ sets to $0.09$. We draw your attention to the different running time of these algorithms in each step, we stop them after 50 seconds. To have a fair comparison, three versions of each IST, TwIST and FISTA corresponding to 5, 10 and 20 iterations of Chambolle's algorithm for solving the proximal subproblems are considered. The results of these implementations are illustrated in Figures 3 and 4.

\begin{figure}[p]\label{f.inp1}
\centering
\subfloat[][$rel.~ 1~ vs.~ iterations~ (chit = 5)$]{\includegraphics[width=4.9cm]{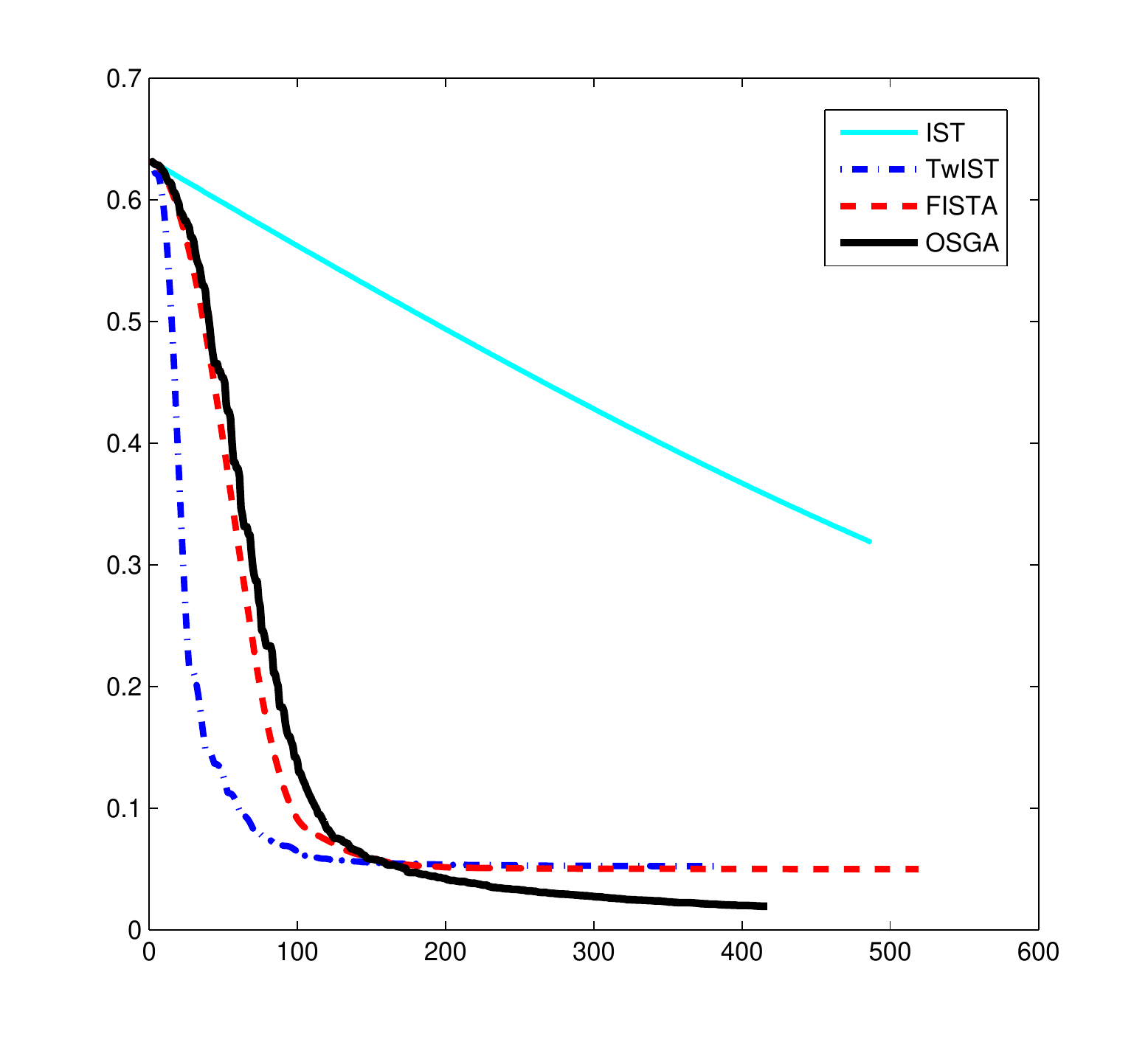}}%
\qquad
\subfloat[][$rel.~ 1~ vs.~ iterations~ (chit = 10)$]{\includegraphics[width=4.9cm]{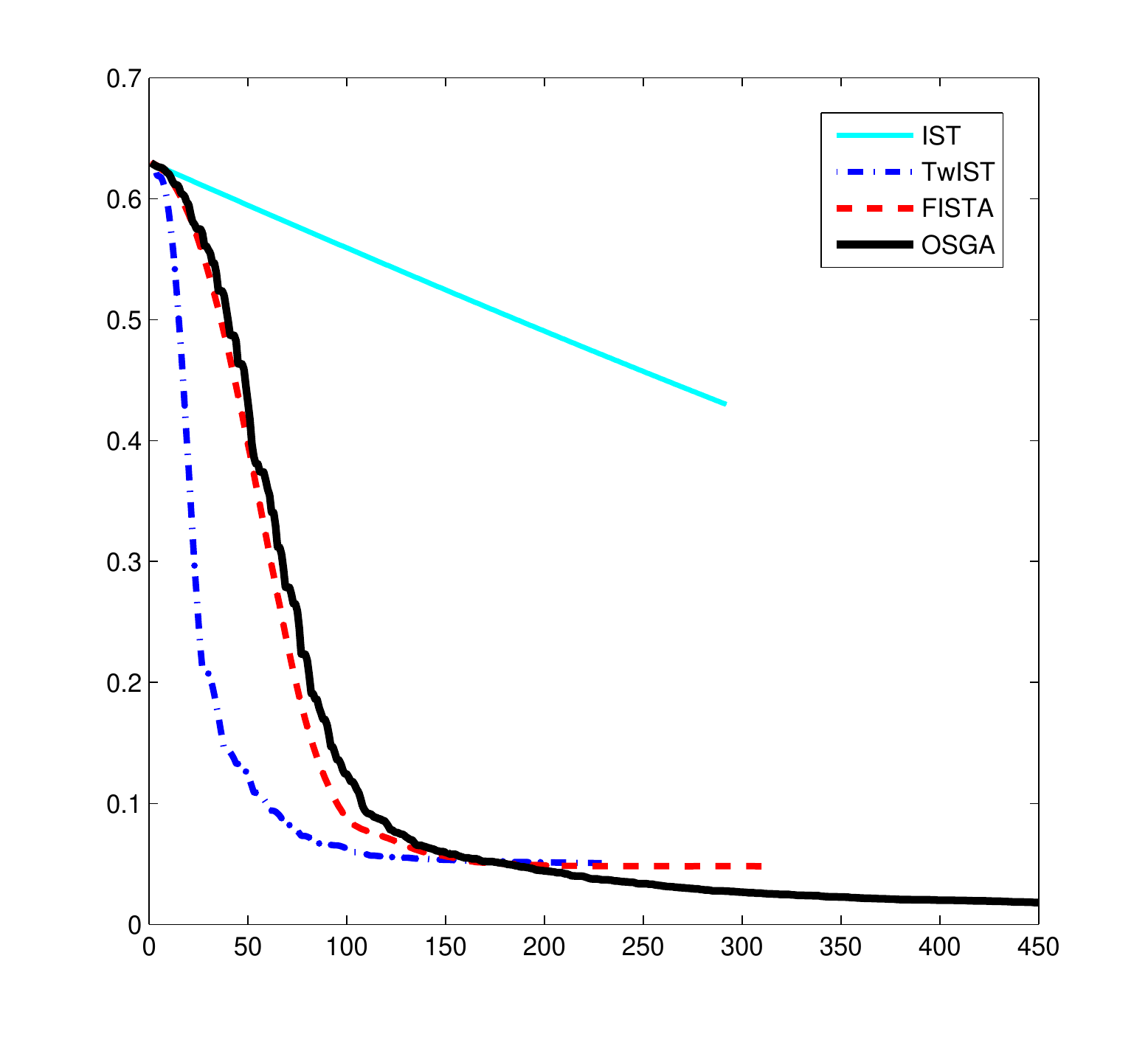}}
\qquad
\subfloat[][$rel.~ 1~ vs.~ iterations~ (chit = 20)$]{\includegraphics[width=4.9cm]{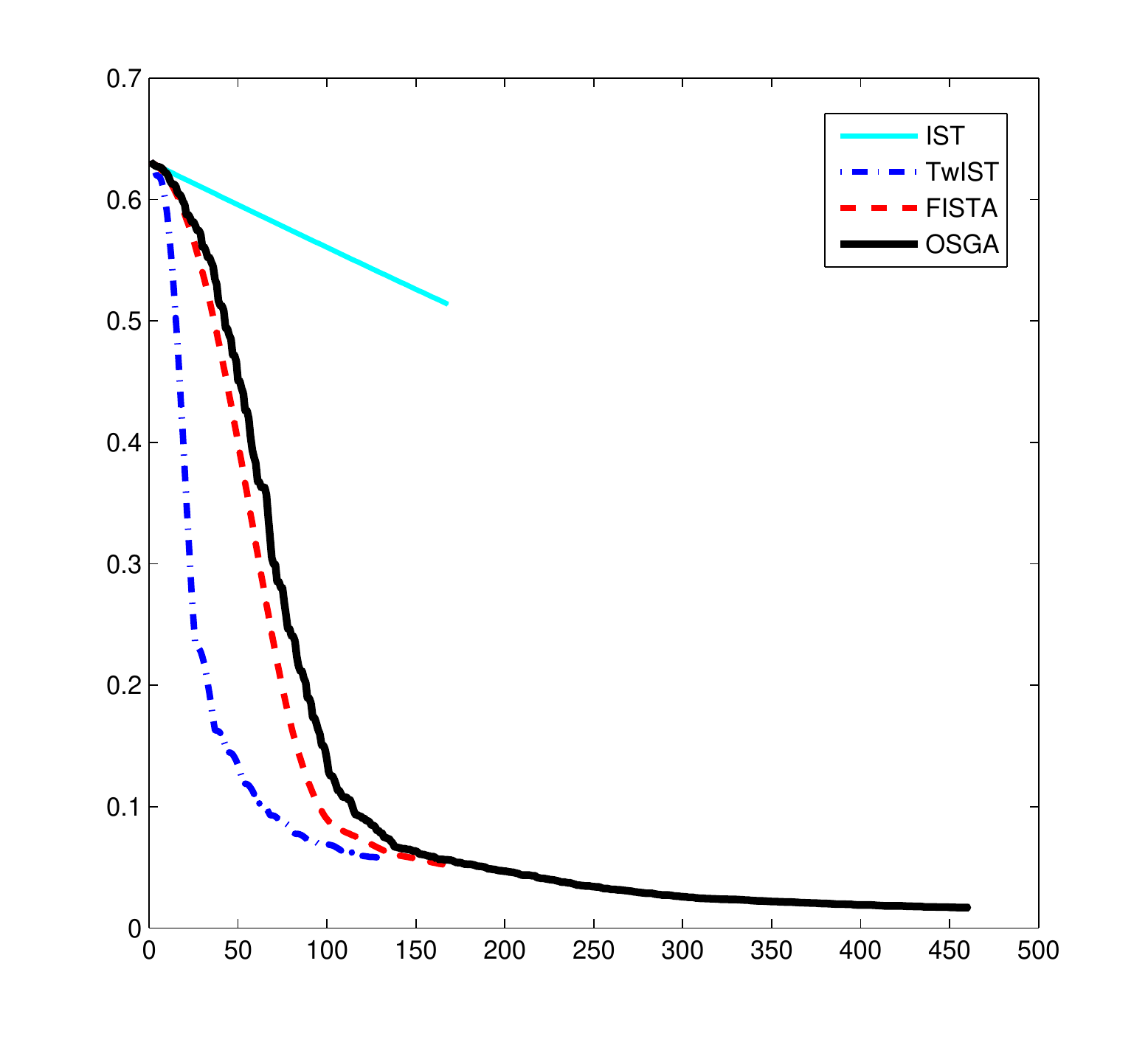}}%
\qquad
\subfloat[][$rel.~ 2~ vs.~ time~ (chit = 5)$]{\includegraphics[width=4.9cm]{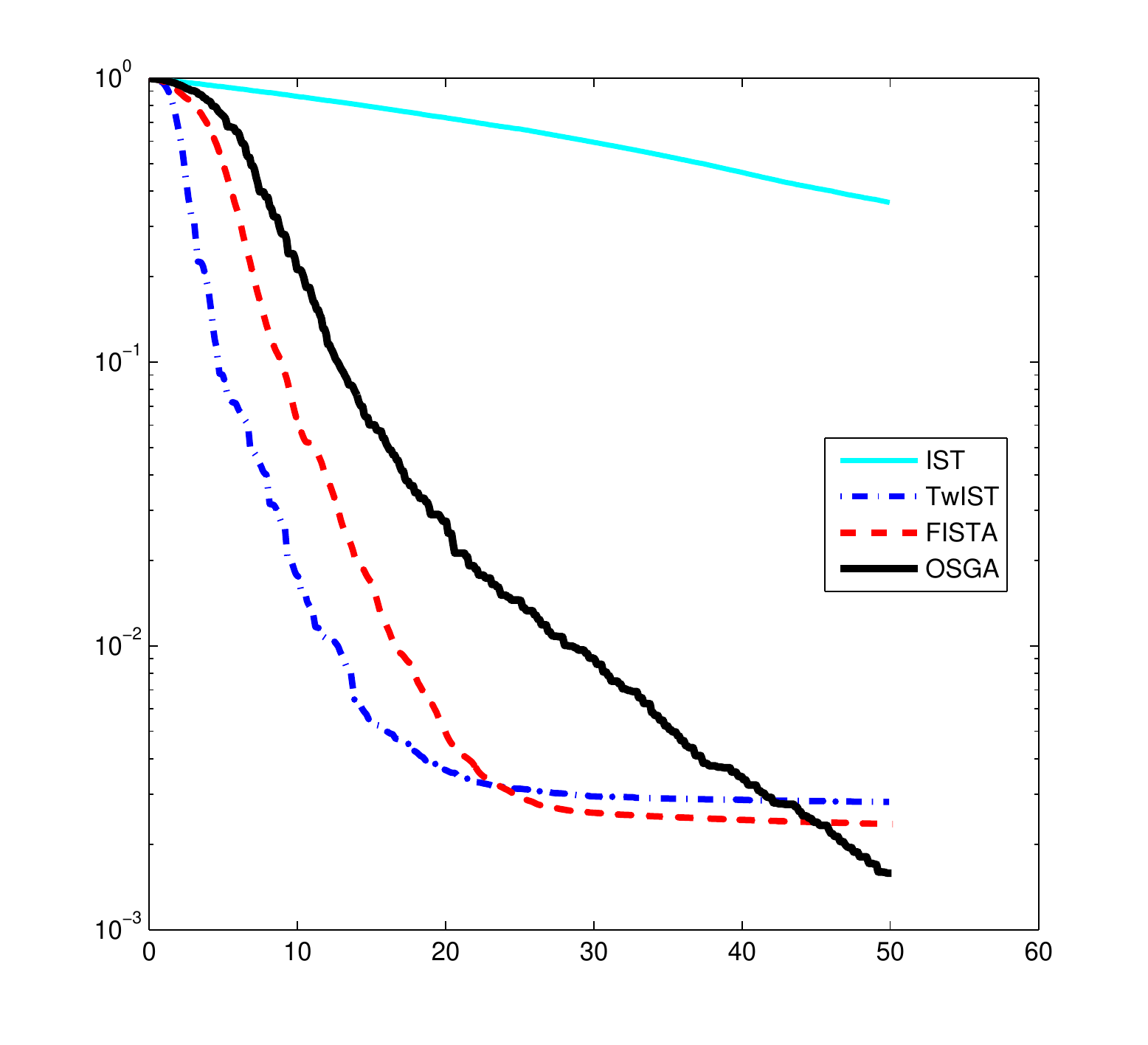}}
\qquad
\subfloat[][$rel.~ 2~ vs.~ time~ (chit = 10)$]{\includegraphics[width=4.9cm]{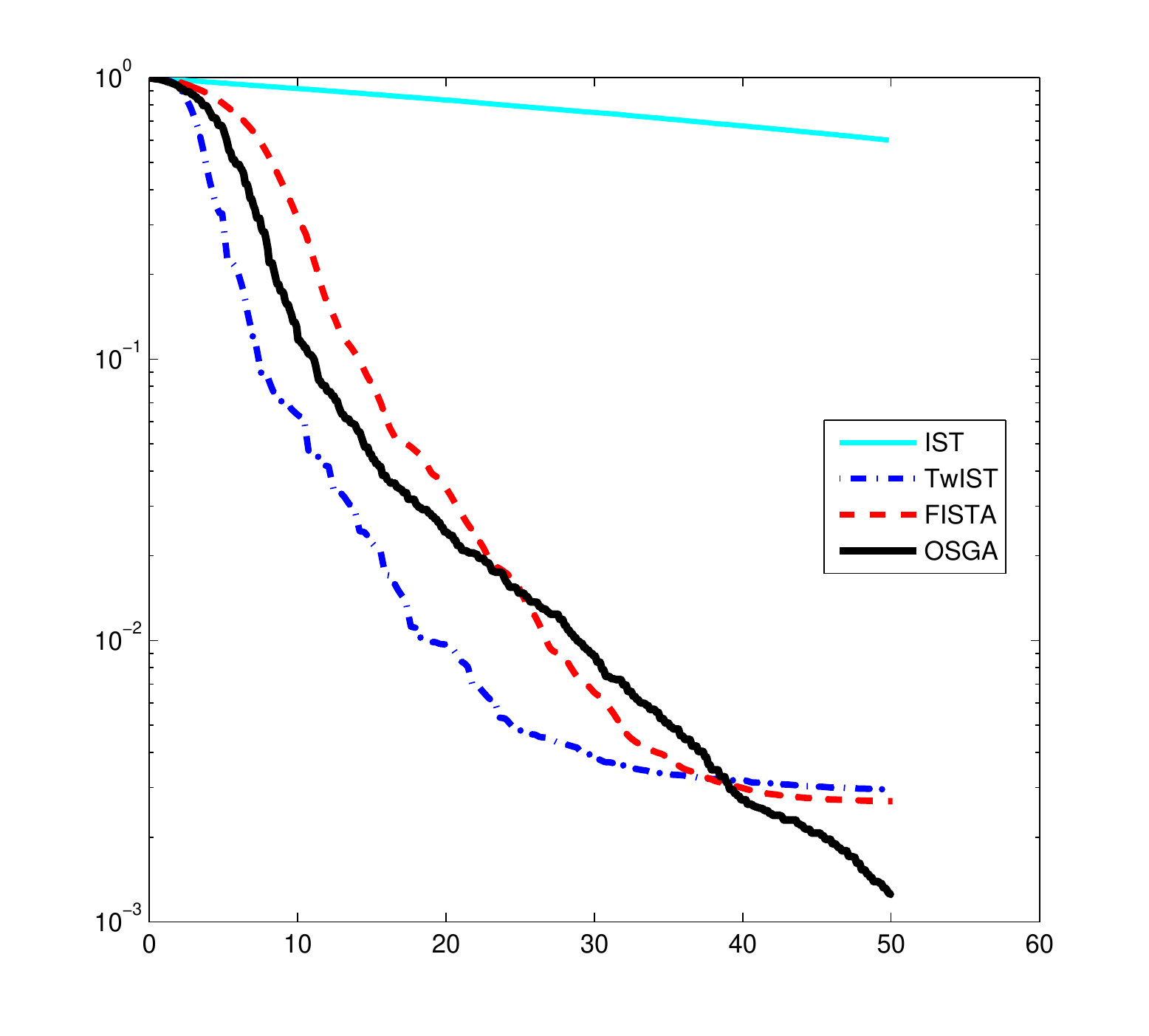}}%
\qquad
\subfloat[][$rel.~ 2~ vs.~ time~ (chit = 20)$]{\includegraphics[width=4.9cm]{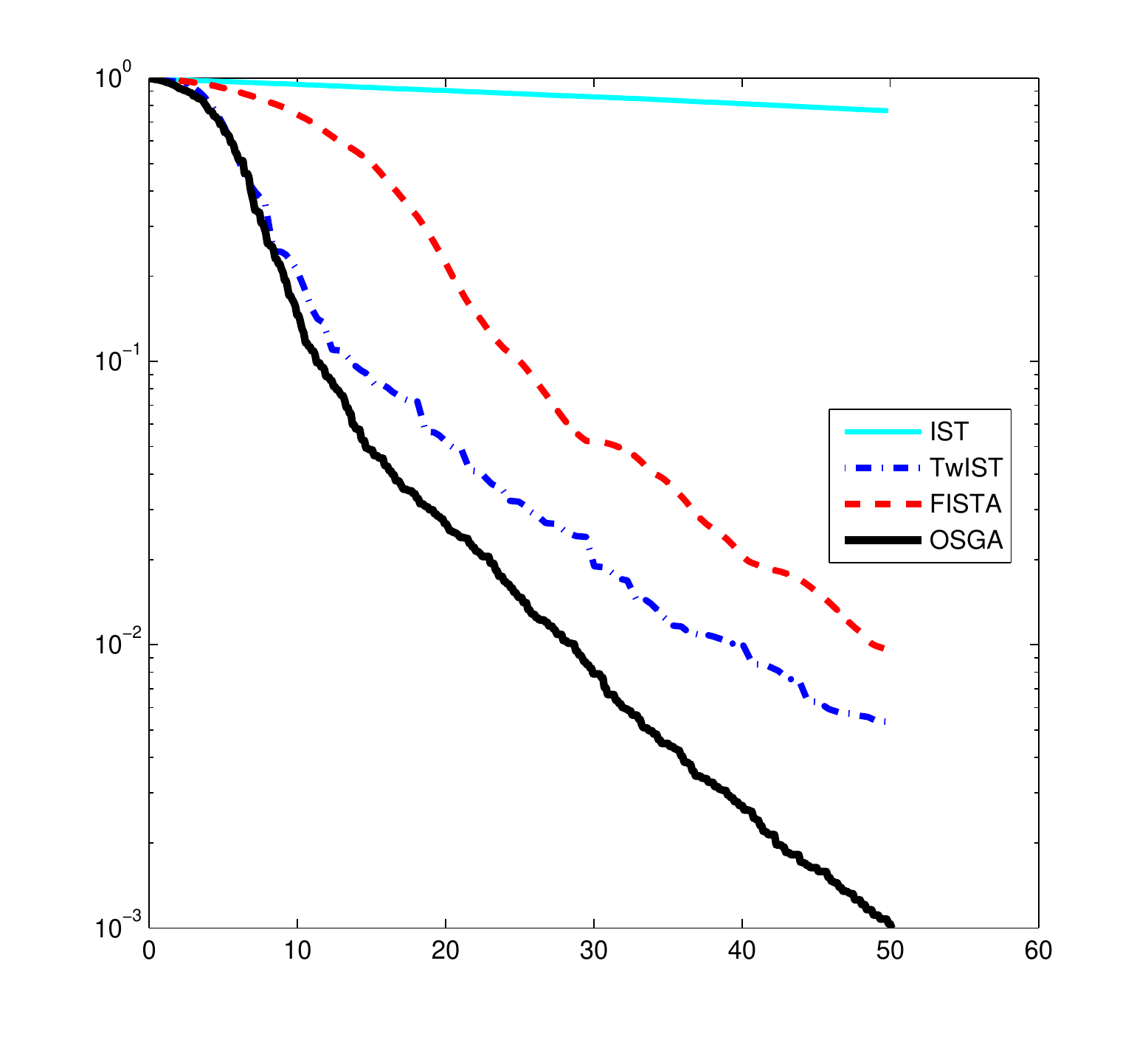}}
\qquad
\subfloat[][$rel.~ 2~ vs.~ iterations~ (chit = 5)$]{\includegraphics[width=4.9cm]{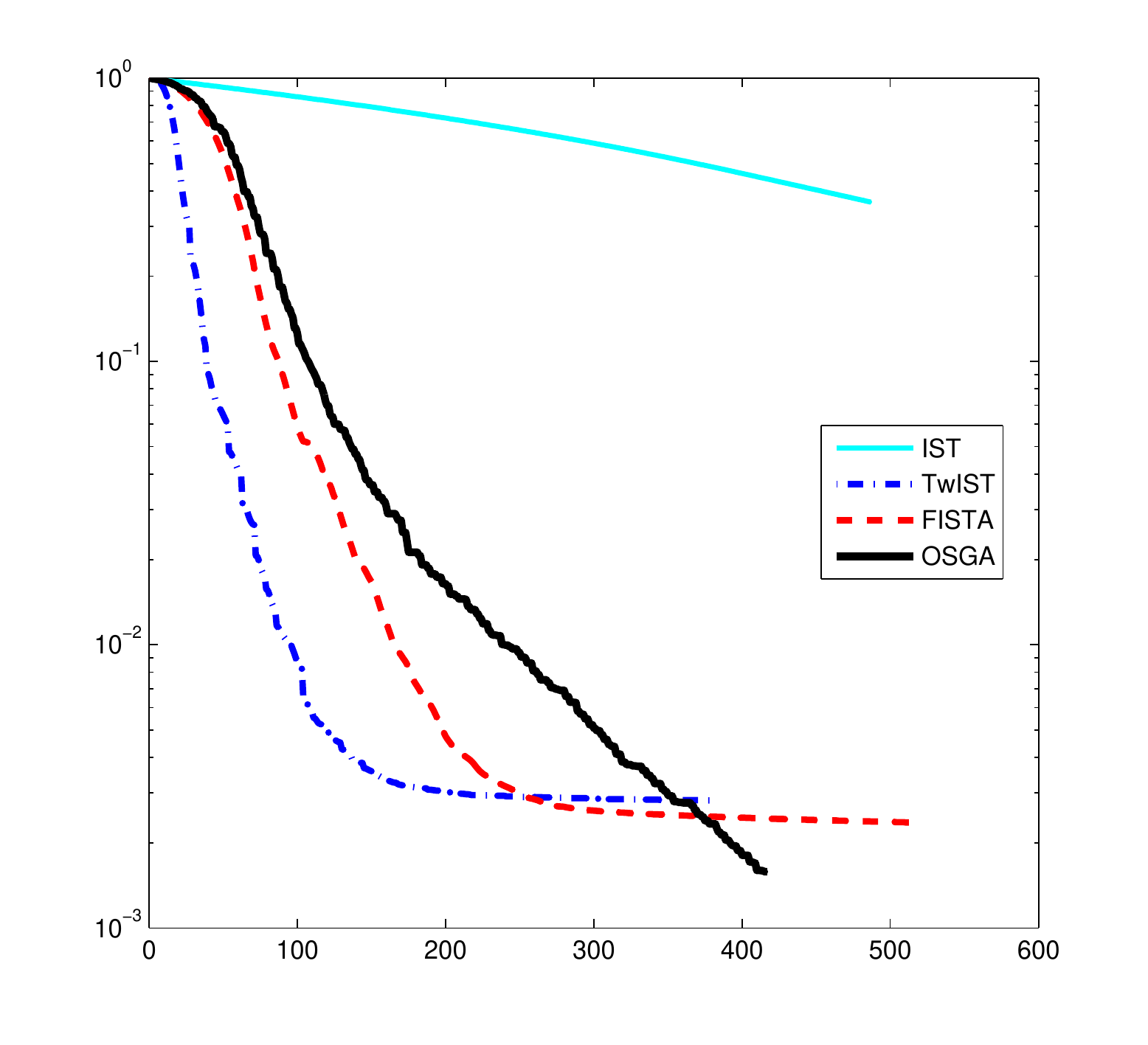}}
\qquad
\subfloat[][$rel.~ 2~ vs.~ iterations~ (chit = 10)$]{\includegraphics[width=4.9cm]{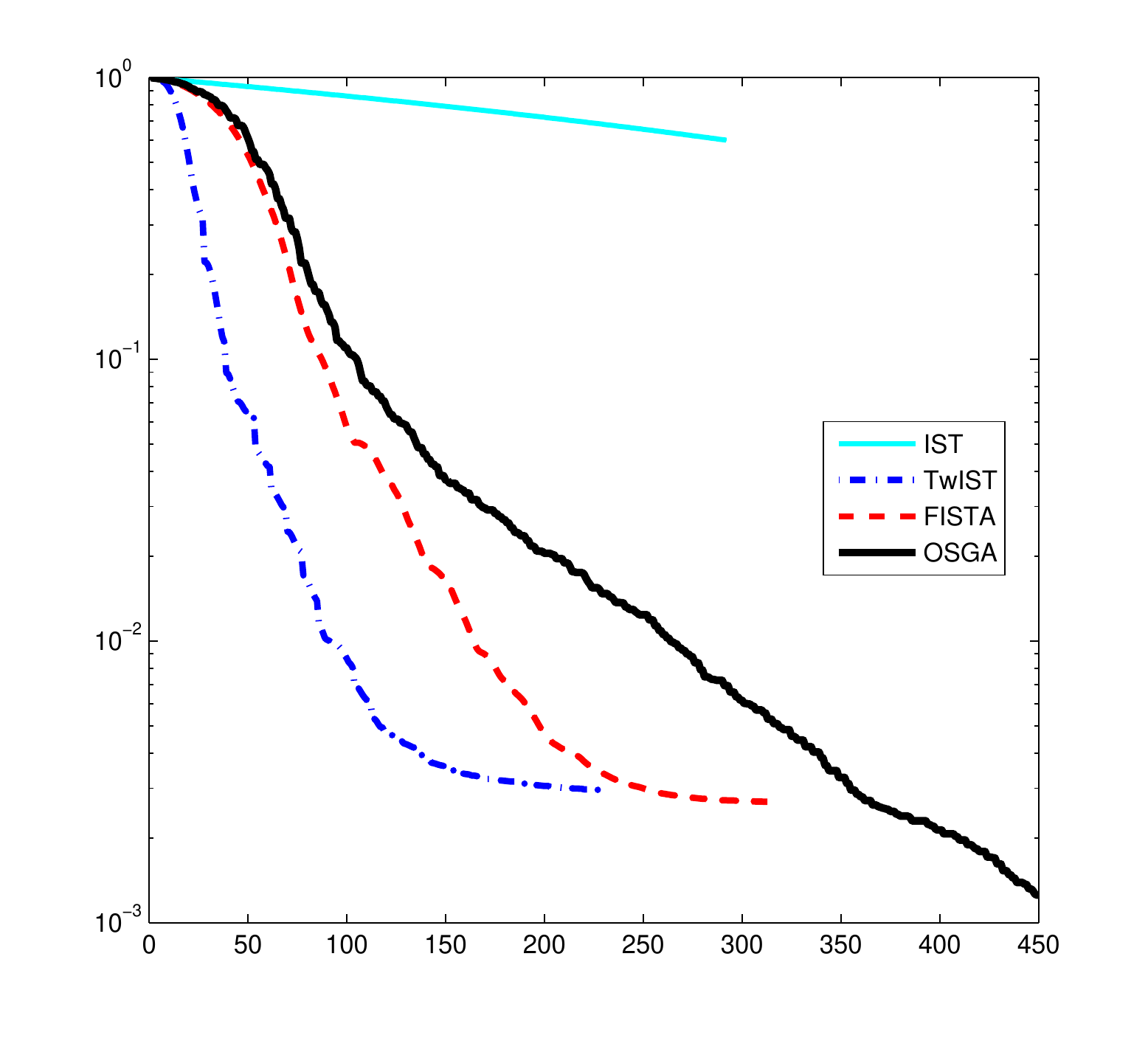}}%
\qquad
\subfloat[][$rel.~ 2~ vs.~ iterations~ (chit = 20)$]{\includegraphics[width=4.9cm]{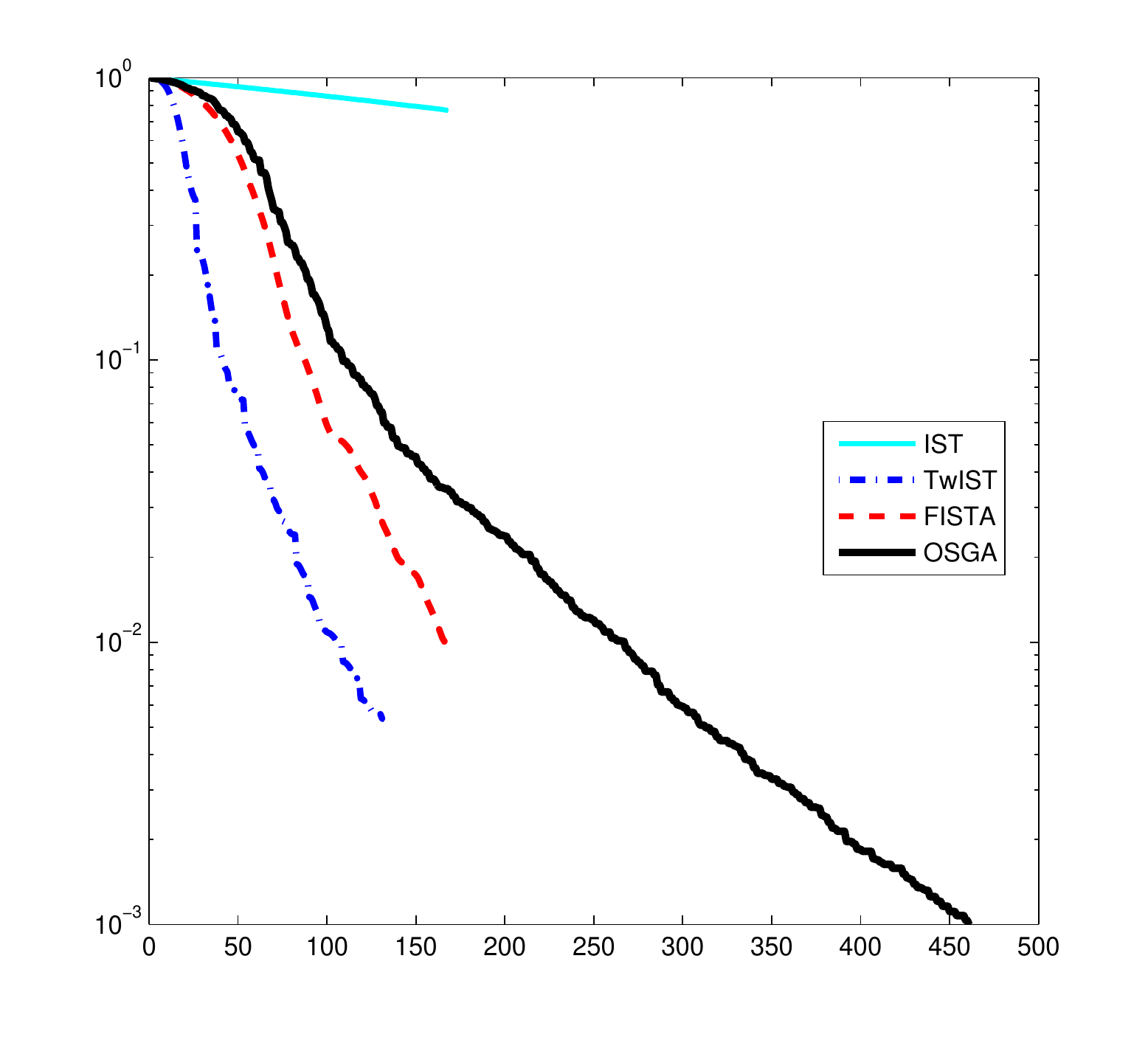}}
\qquad
\subfloat[][$ISNR~ vs.~ iterations~ (chit = 5)$]{\includegraphics[width=4.9cm]{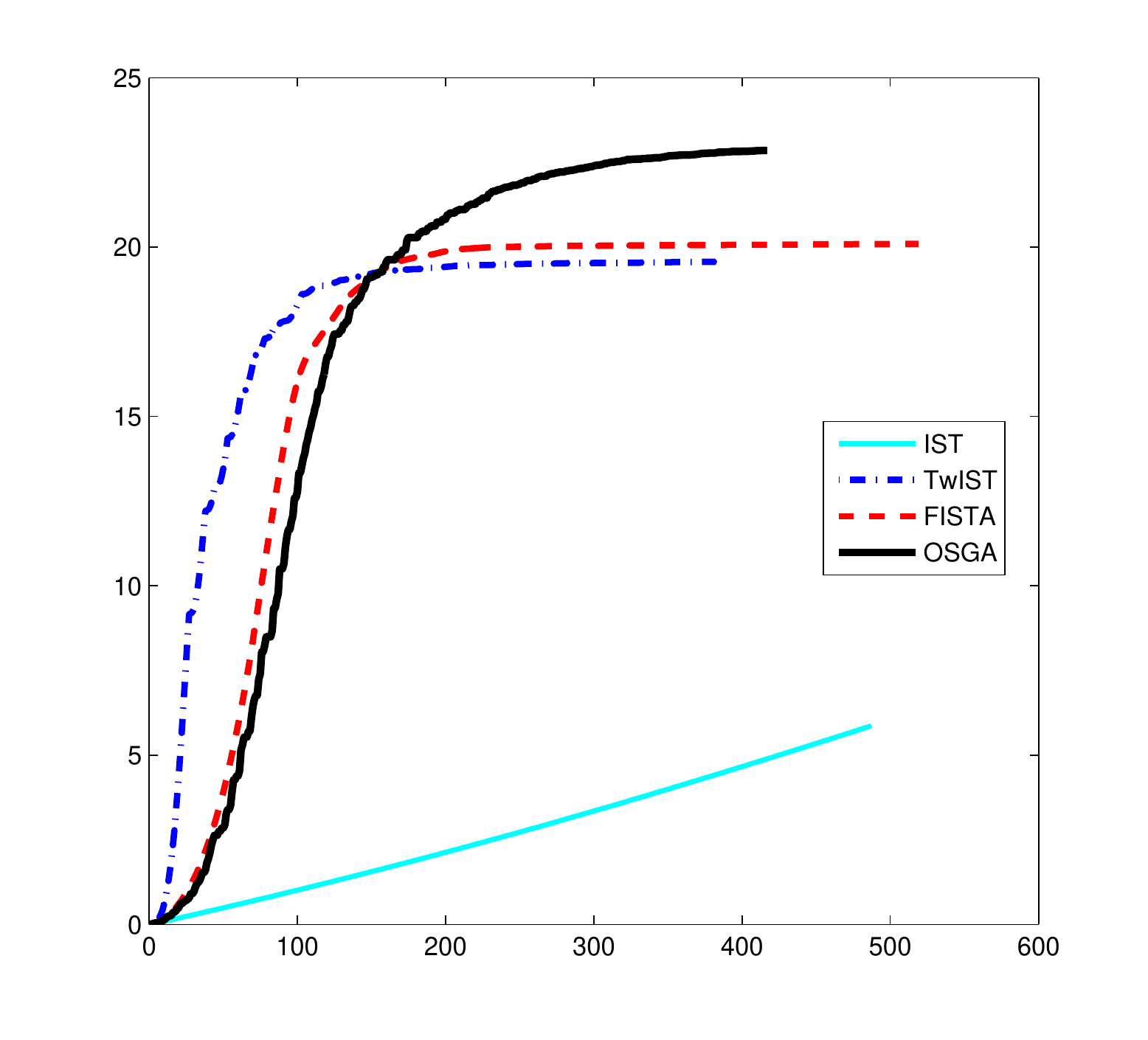}}
\qquad
\subfloat[][$ISNR~ vs.~ iterations~ (chit = 10)$]{\includegraphics[width=4.9cm]{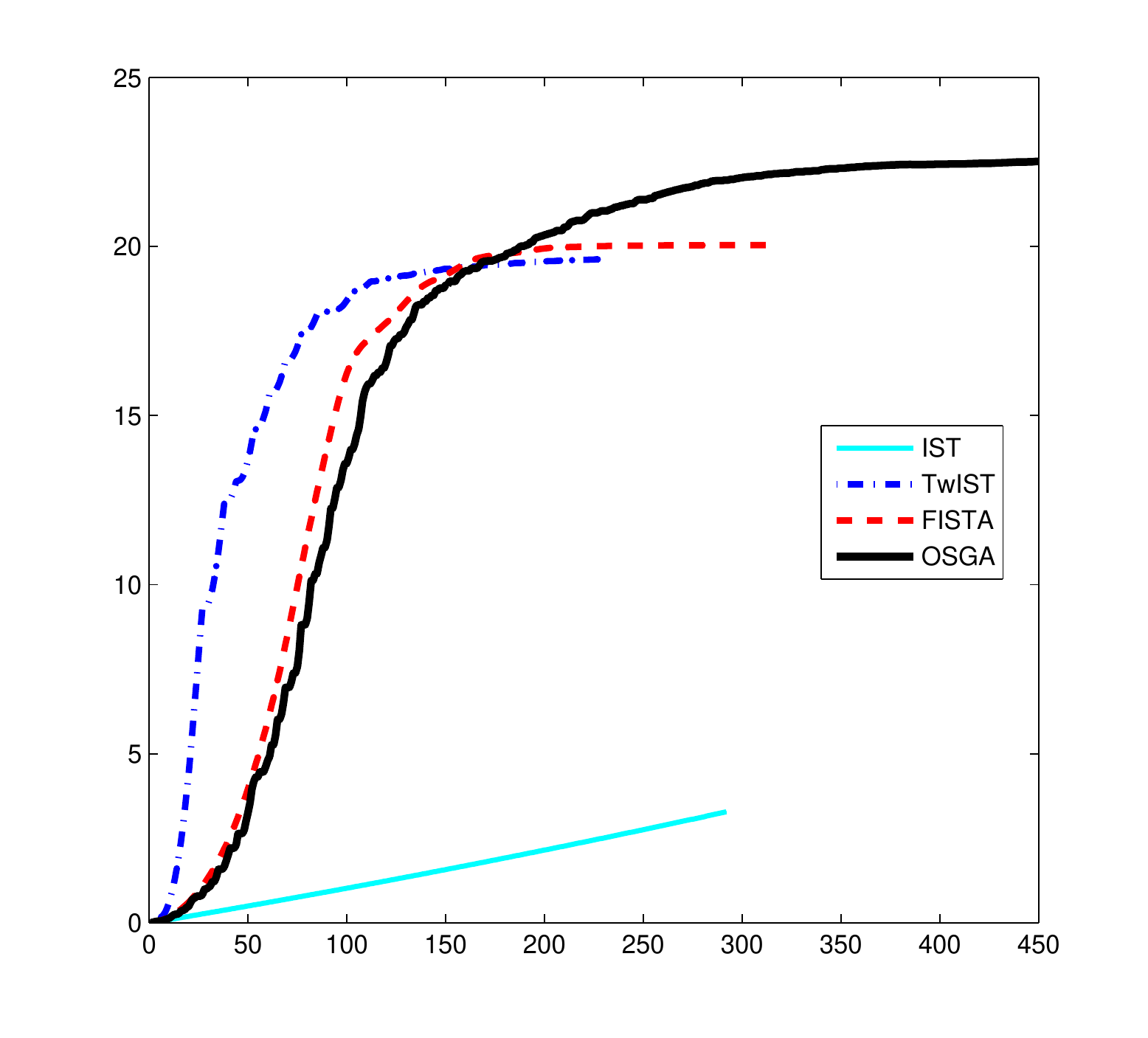}}%
\qquad
\subfloat[][$ISNR~ vs.~ iterations~ (chit = 20)$]{\includegraphics[width=4.9cm]{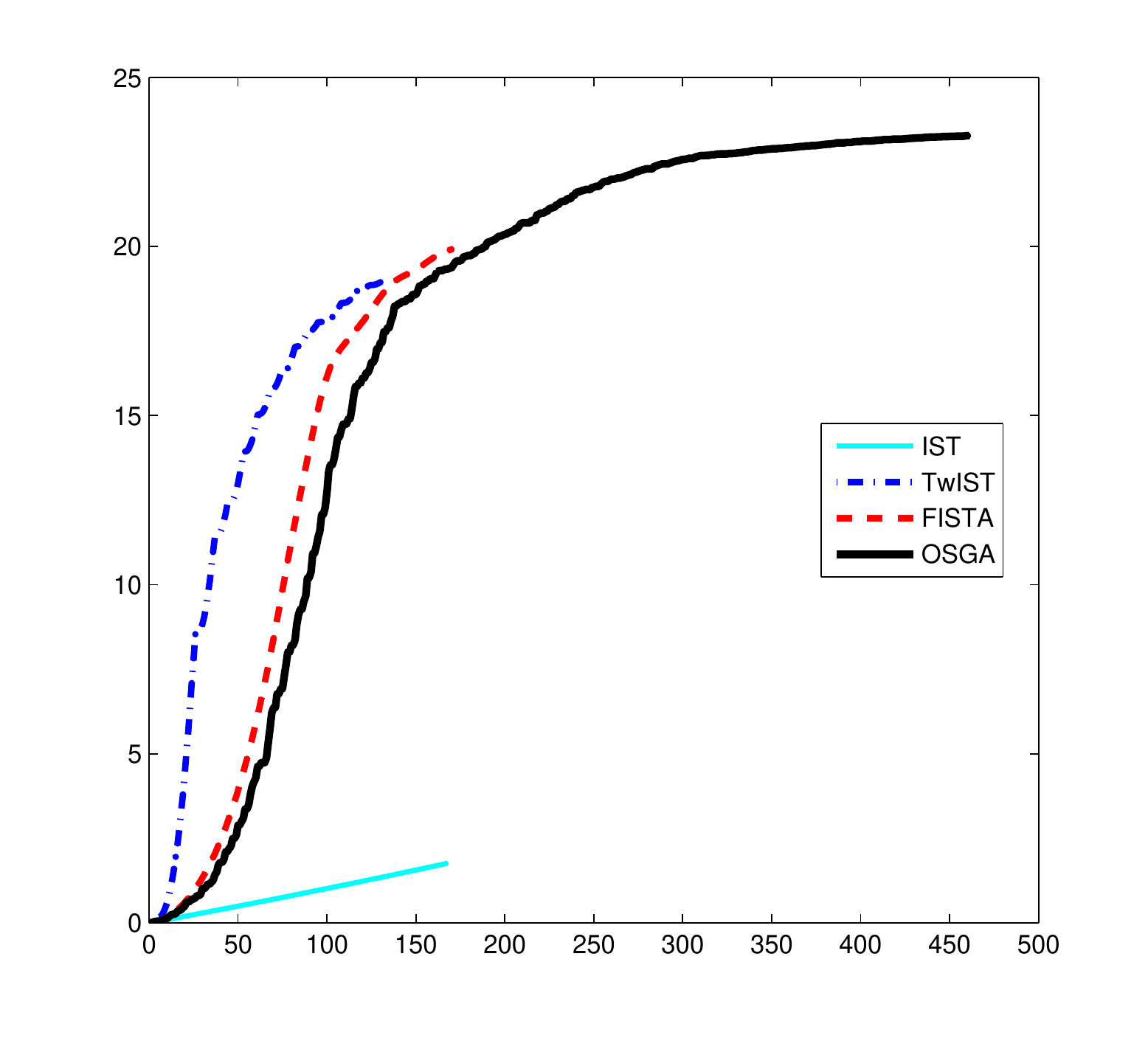}}
\caption{A comparison among IST, TwIST, FISTA and OSGA for inpainting the $512 \times 512$ Head CT image when 40\% of its data is missing. The algorithms are stopped after 50 seconds of the running time. Each column stands for comparisons with a different number of Chambolle's iterations. The first row denotes the relative error $rel.~ 1 = \|x_k - x^*\|_2/\|x^*\|_2$ of points versus iterations, the second row shows the relative error $rel.~ 2 = (f_k - f^*)/(f_0 - f^*)$ of function values versus time, the third row illustrates $rel. 2$ versus iterations, and fourth row stands for ISNR (\ref{e.isnr}) versus iterations.}%
\label{fig:cont}
\end{figure}

\begin{figure}[p]\label{f.inp2}
\centering
\subfloat[][Original image]{\includegraphics[width=4.9cm]{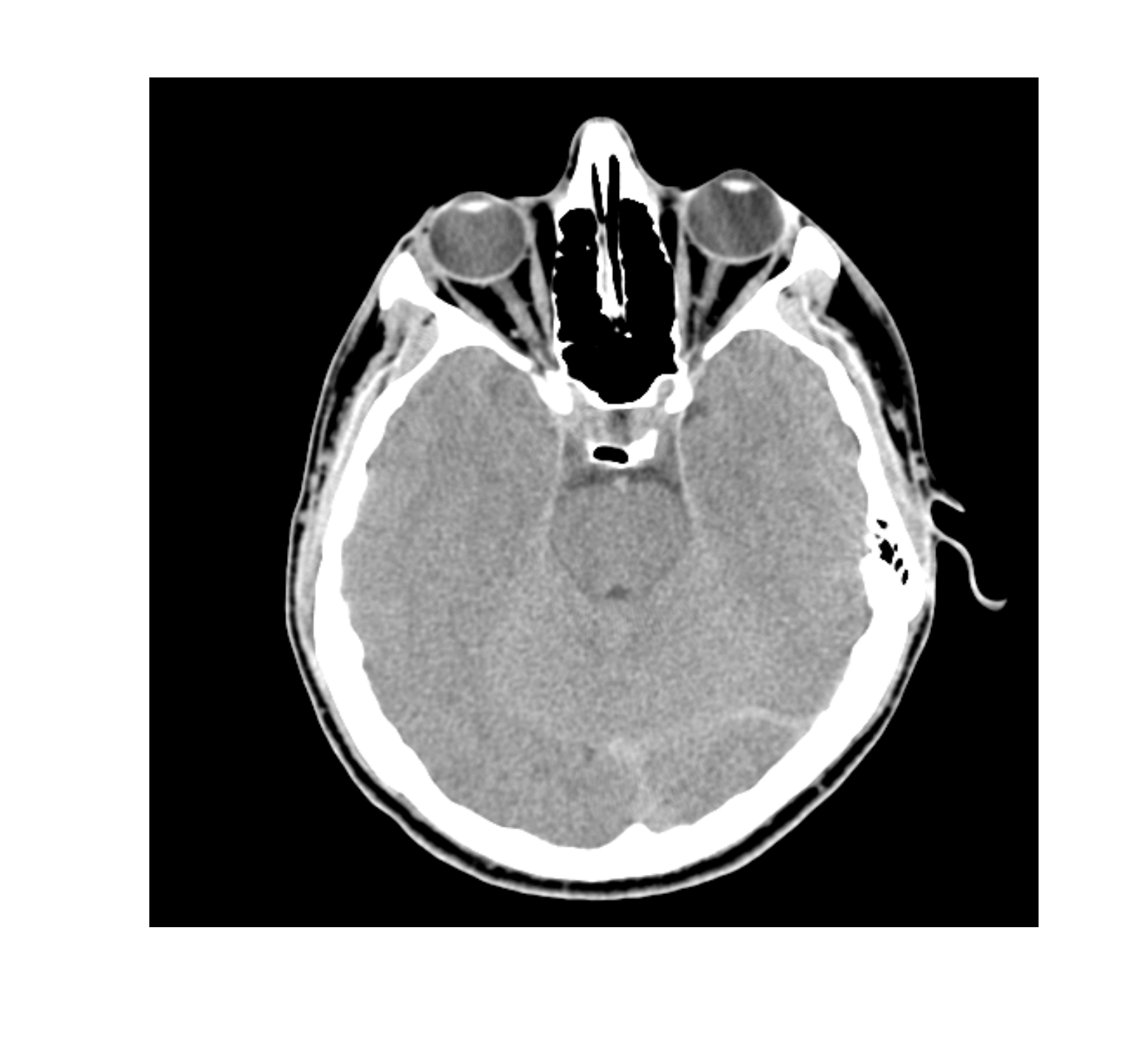}}%
\qquad
\subfloat[][Blurred/noisy image]{\includegraphics[width=4.9cm]{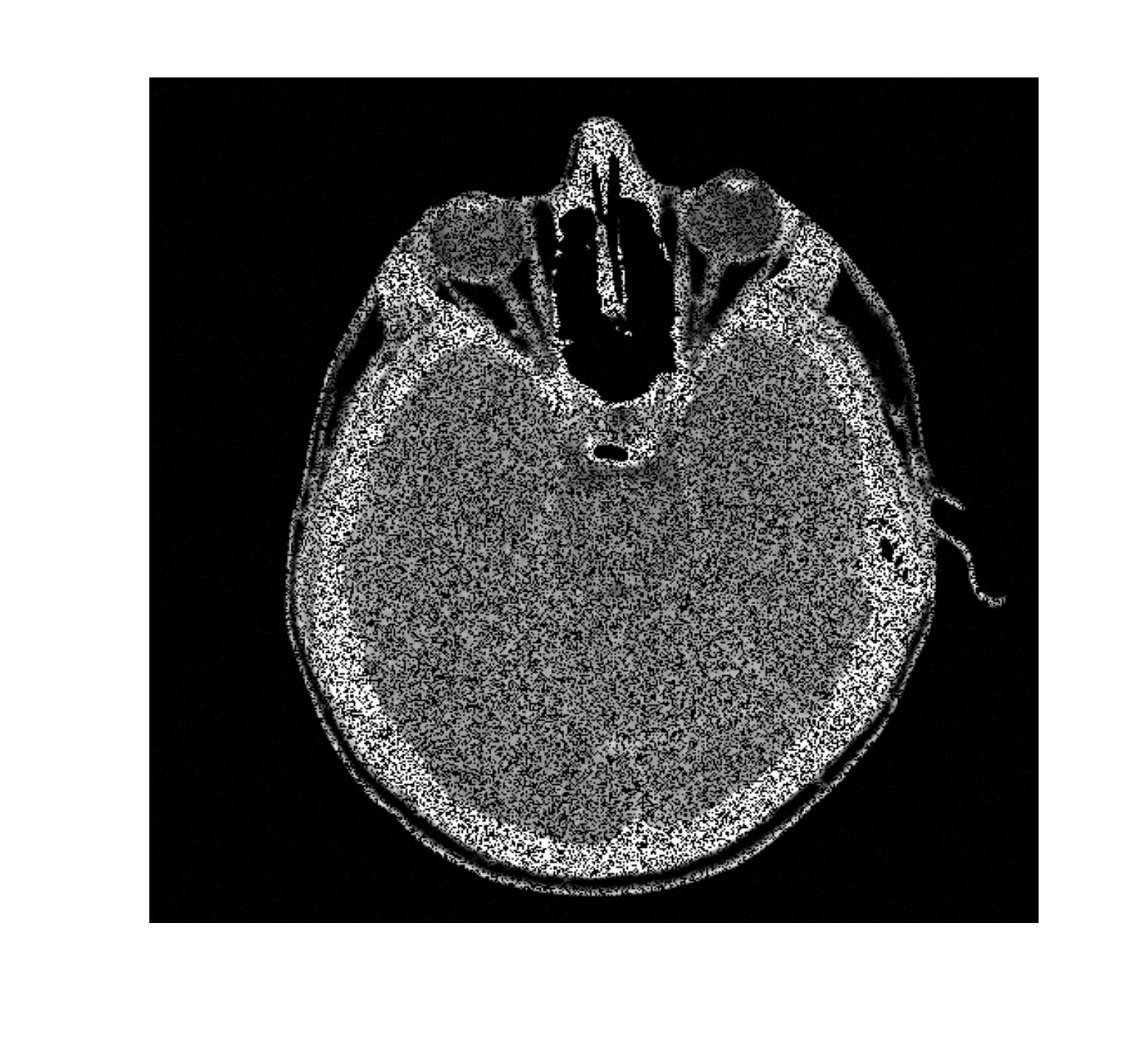}}
\qquad
\subfloat[][OSGA: $\mathrm{F} = 187564.02$, $\mathrm{PSNR} = 33.13$]{\includegraphics[width=4.9cm]{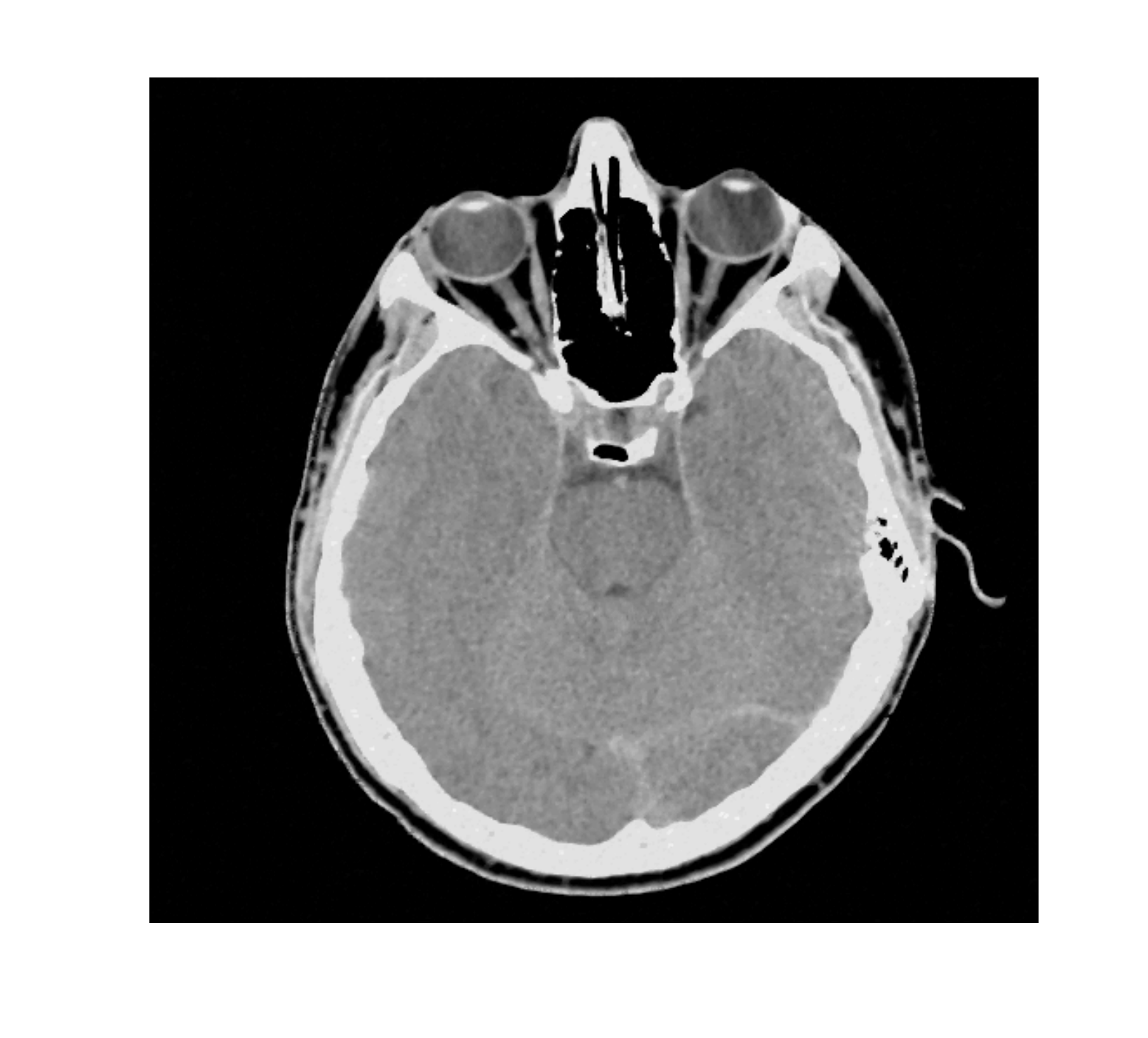}}%
\qquad
\subfloat[][IST ($chit = 5$): $\mathrm{F} = 710603.30$, $\mathrm{PSNR} = 16.15$]{\includegraphics[width=4.9cm]{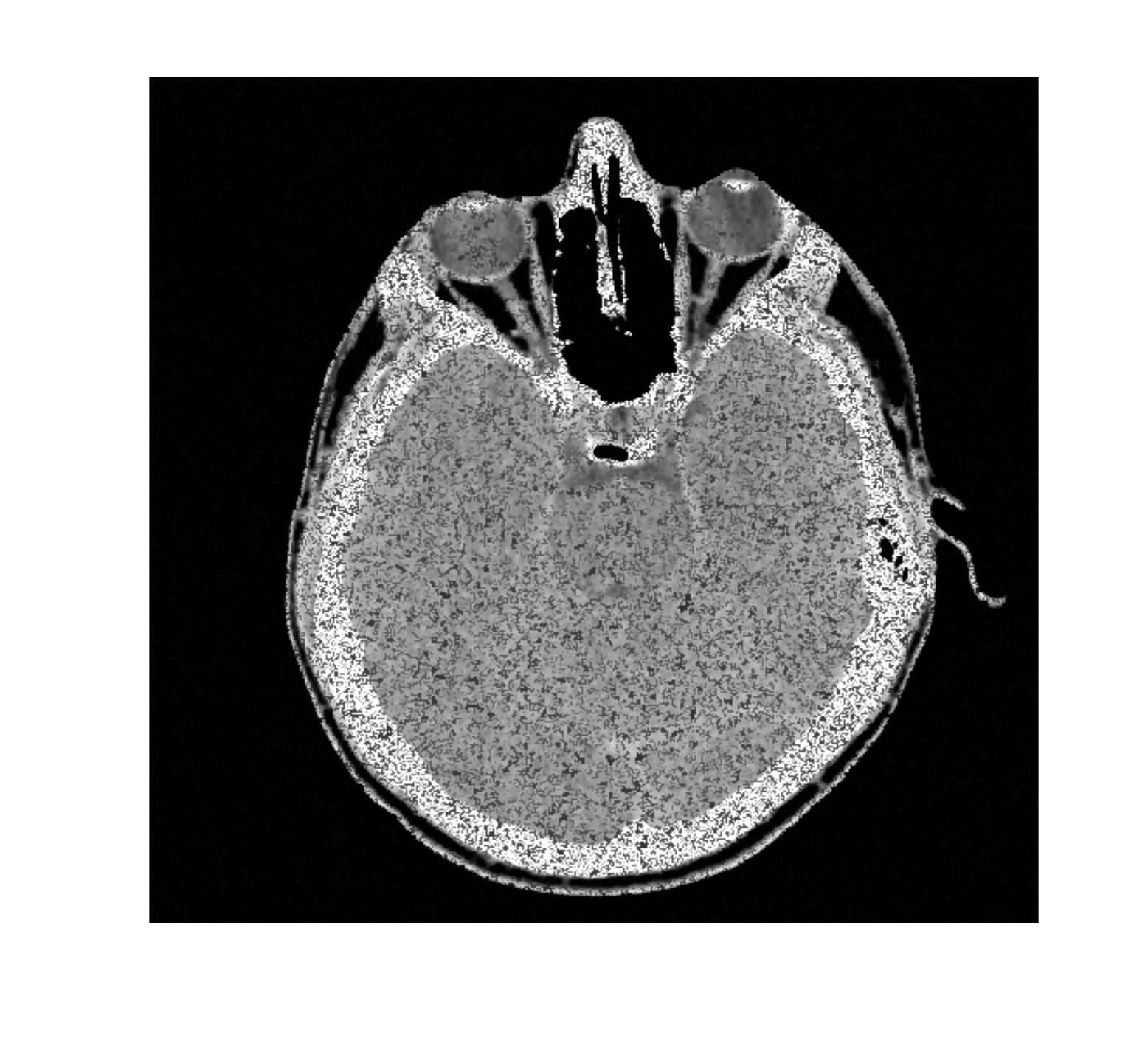}}
\qquad
\subfloat[][IST ($chit = 10$): $\mathrm{F} = 1046712.18$, $\mathrm{PSNR} = 13.61$]{\includegraphics[width=4.9cm]{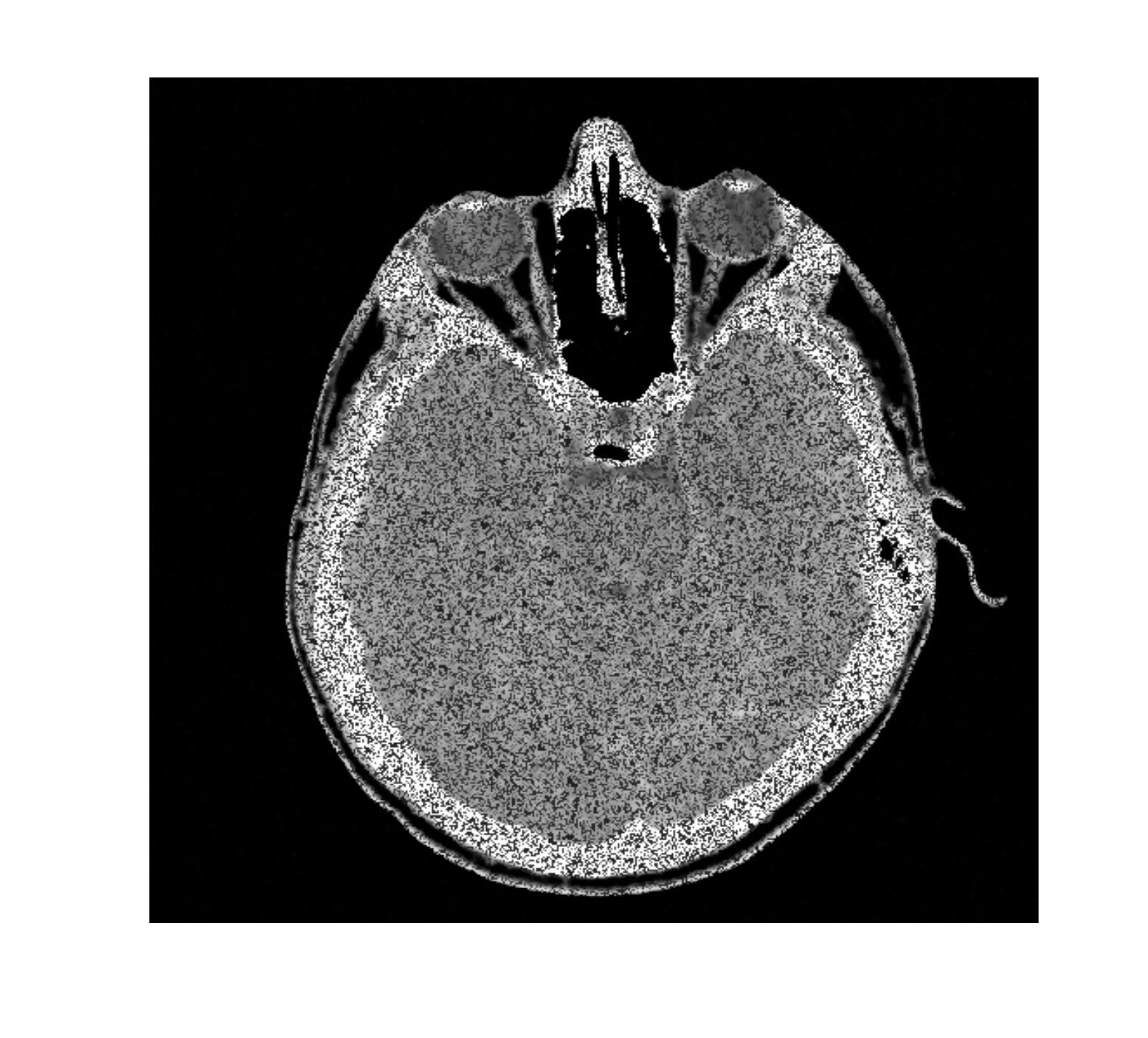}}%
\qquad
\subfloat[][IST ($chit = 20$): $\mathrm{F} = 1284434.37$, $\mathrm{PSNR} = 12.08$]{\includegraphics[width=4.9cm]{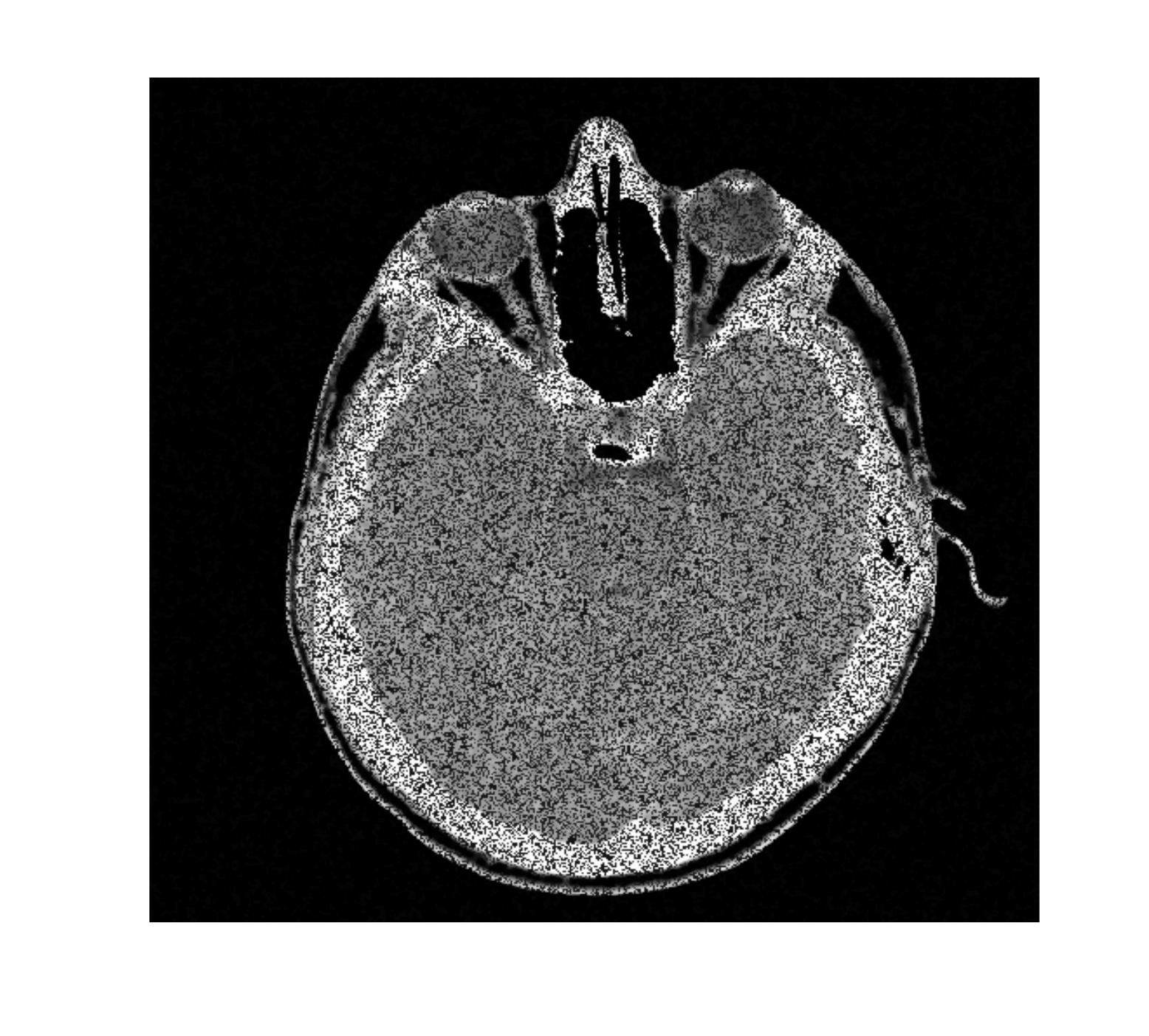}}
\qquad
\subfloat[][TwIST ($chit = 5$): $\mathrm{F} = 189347.03$, $\mathrm{PSNR} = 29.84$]{\includegraphics[width=4.9cm]{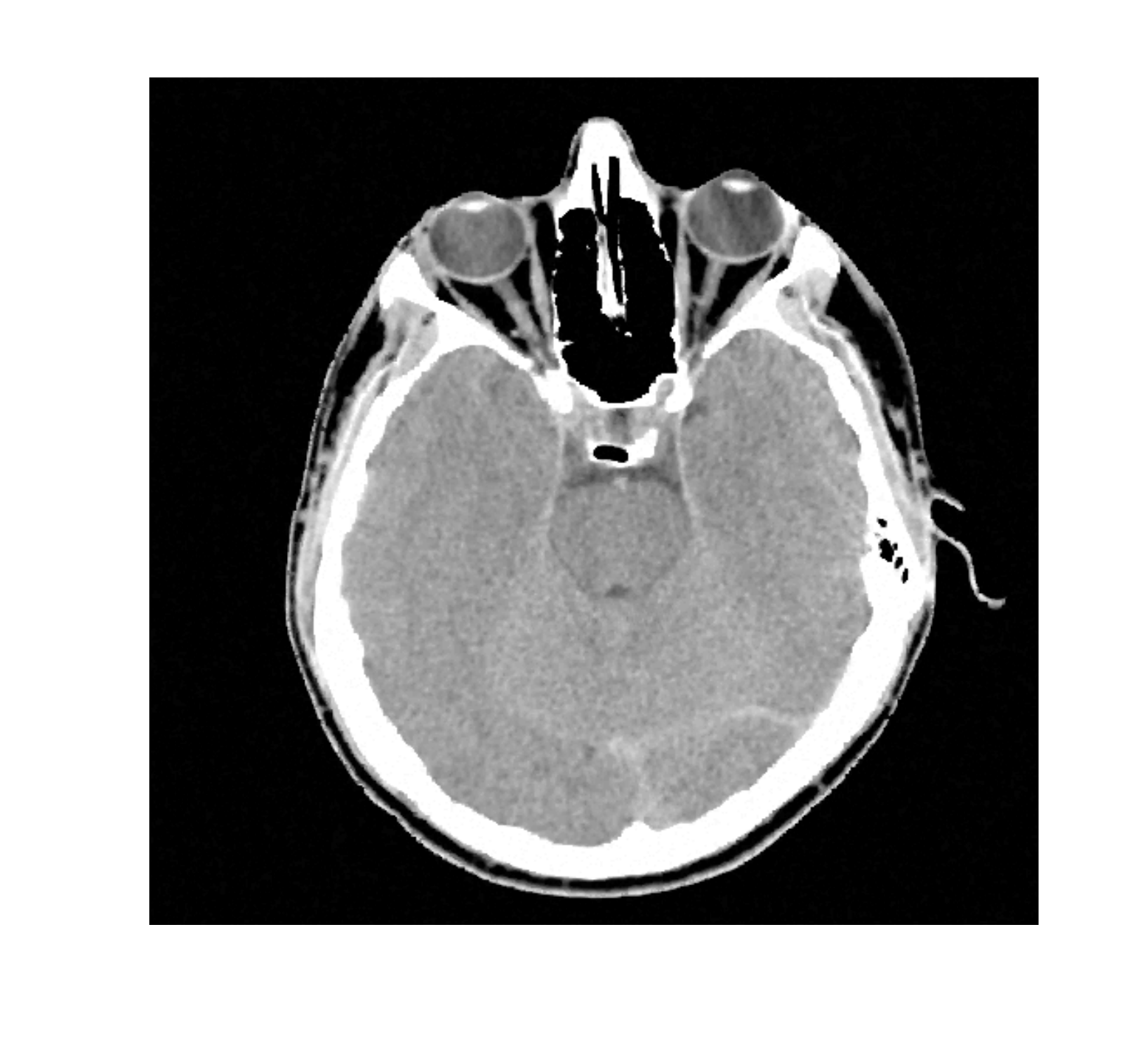}}
\qquad
\subfloat[][TwIST ($chit = 10$): $\mathrm{F} = 189533.03$, $\mathrm{PSNR} = 29.94$]{\includegraphics[width=4.9cm]{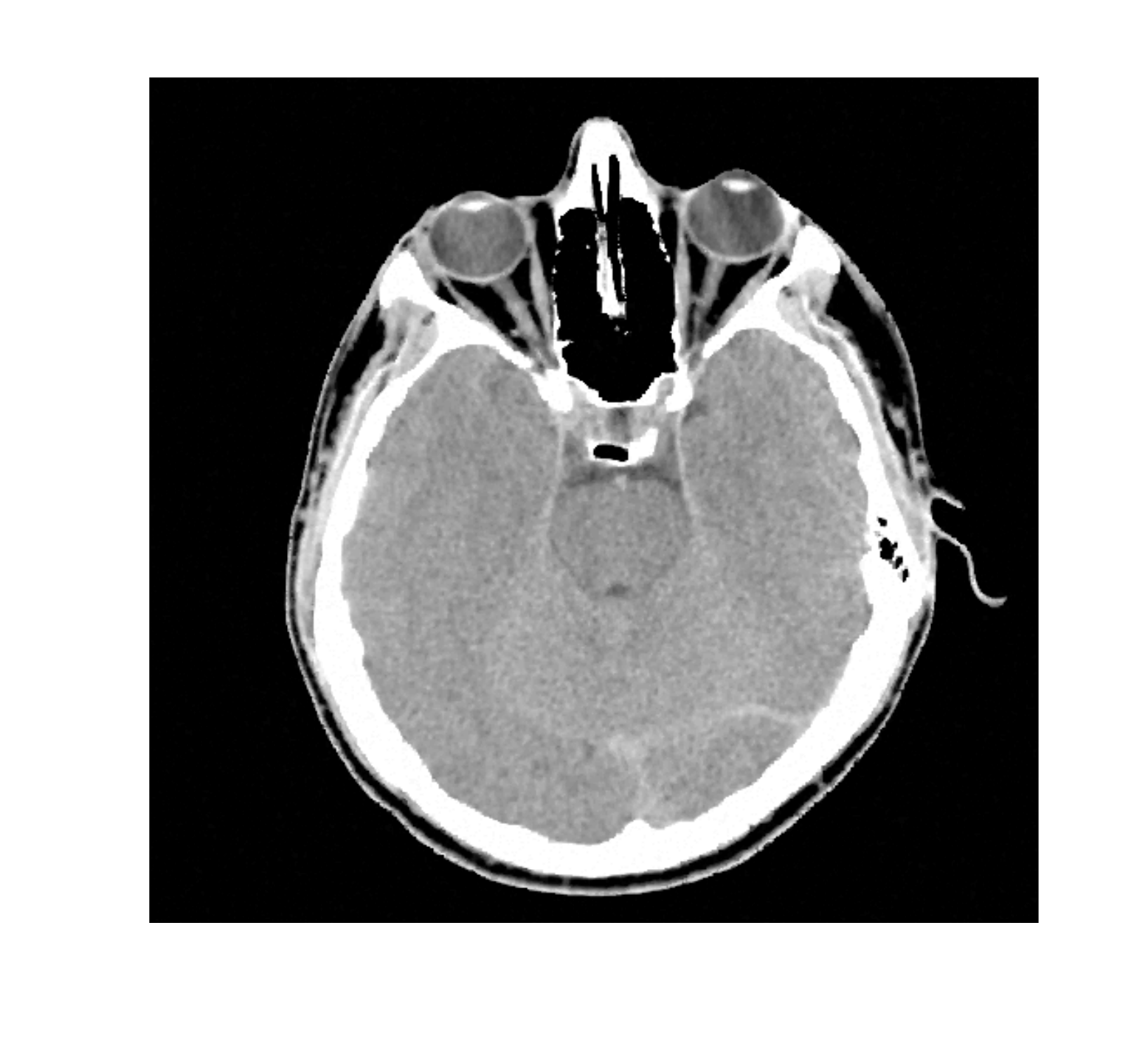}}%
\qquad
\subfloat[][TwIST ($chit = 20$): $\mathrm{F} = 192946.47$, $\mathrm{PSNR} = 29.26$]{\includegraphics[width=4.9cm]{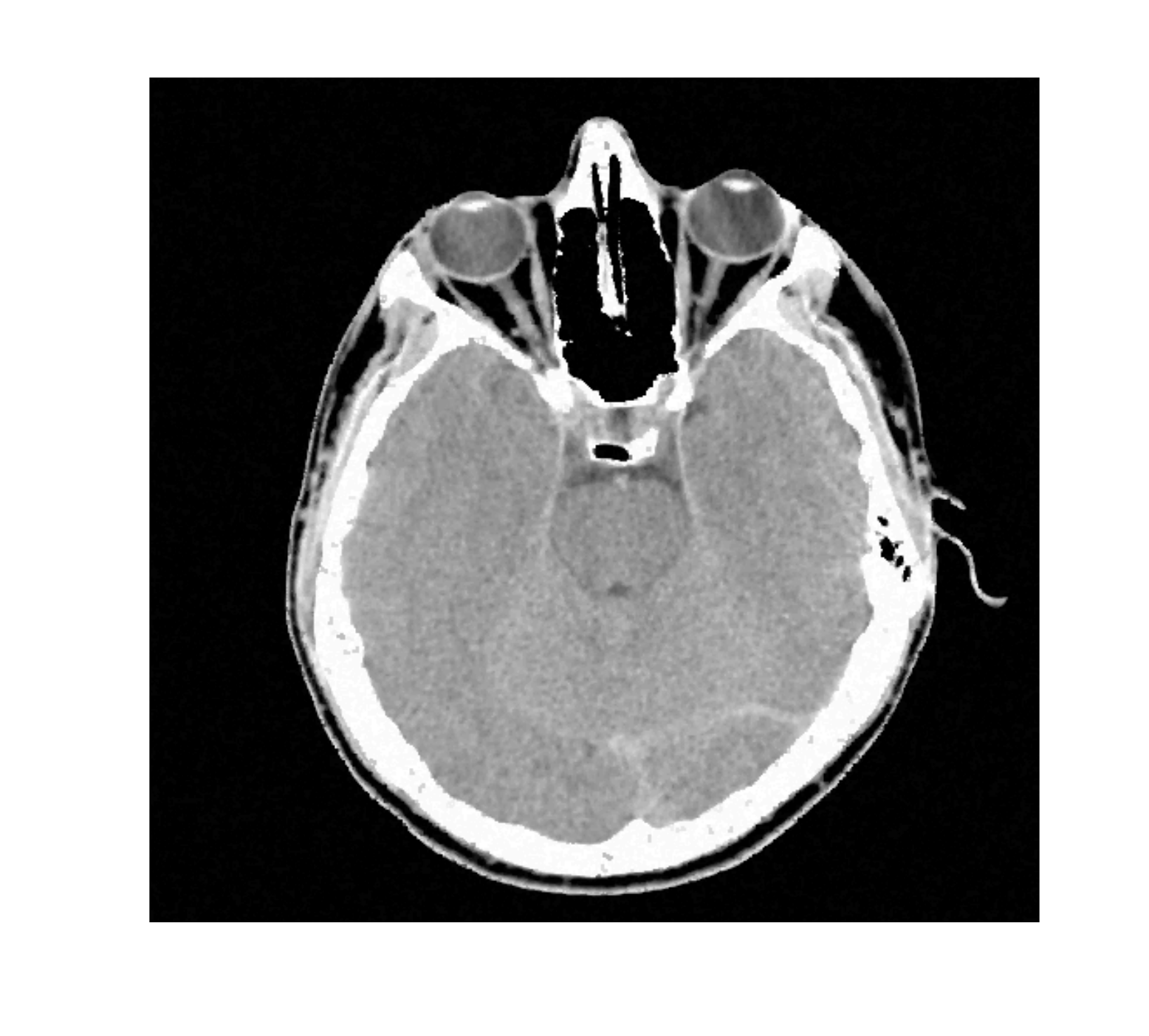}}
\qquad
\subfloat[][FISTA ($chit = 5$): $\mathrm{F} = 188676.25$, $\mathrm{PSNR} = 30.37$]{\includegraphics[width=4.9cm]{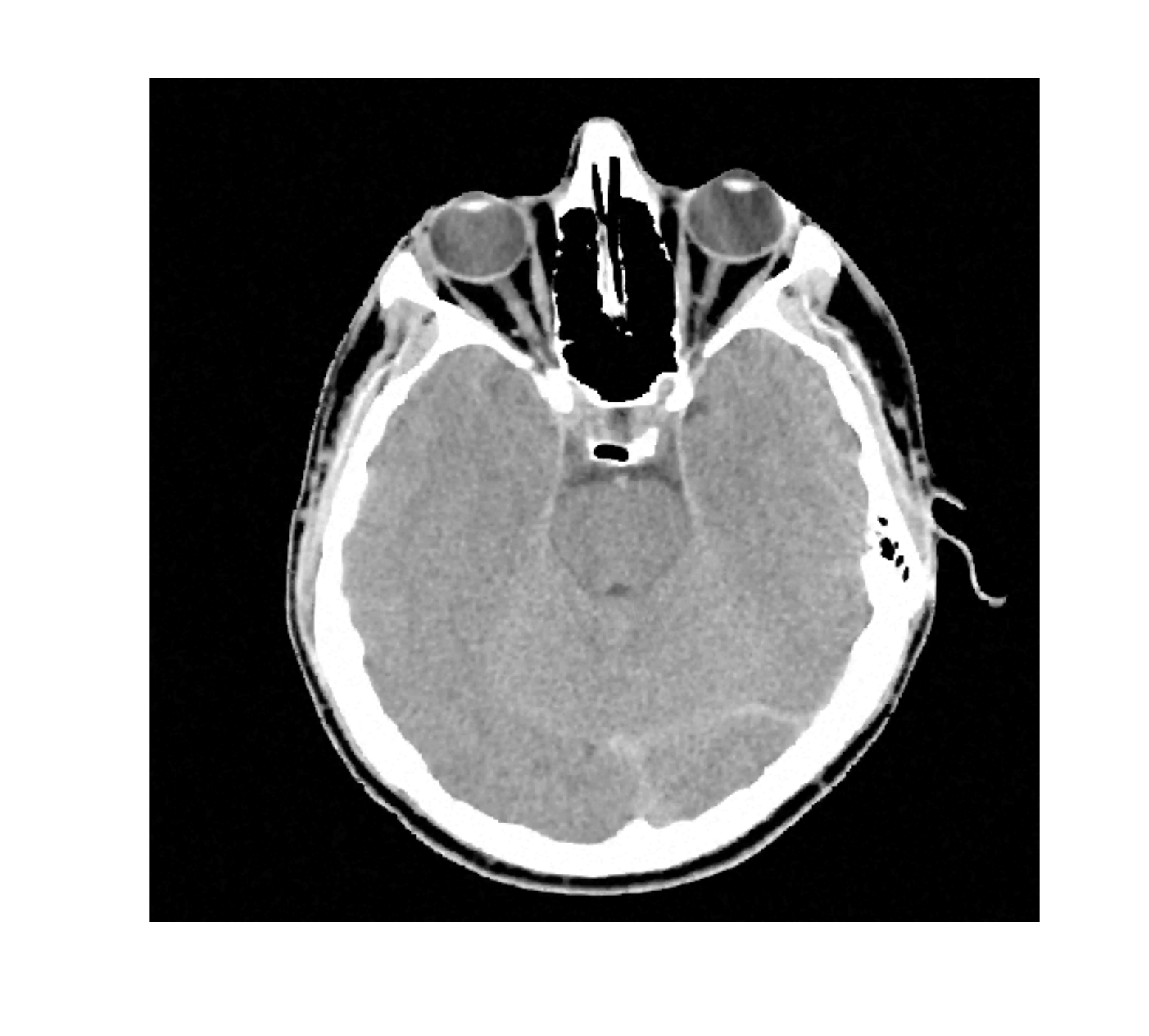}}
\qquad
\subfloat[][FISTA ($chit = 10$): $\mathrm{F} = 189145.34$, $\mathrm{PSNR} = 30.36$]{\includegraphics[width=4.9cm]{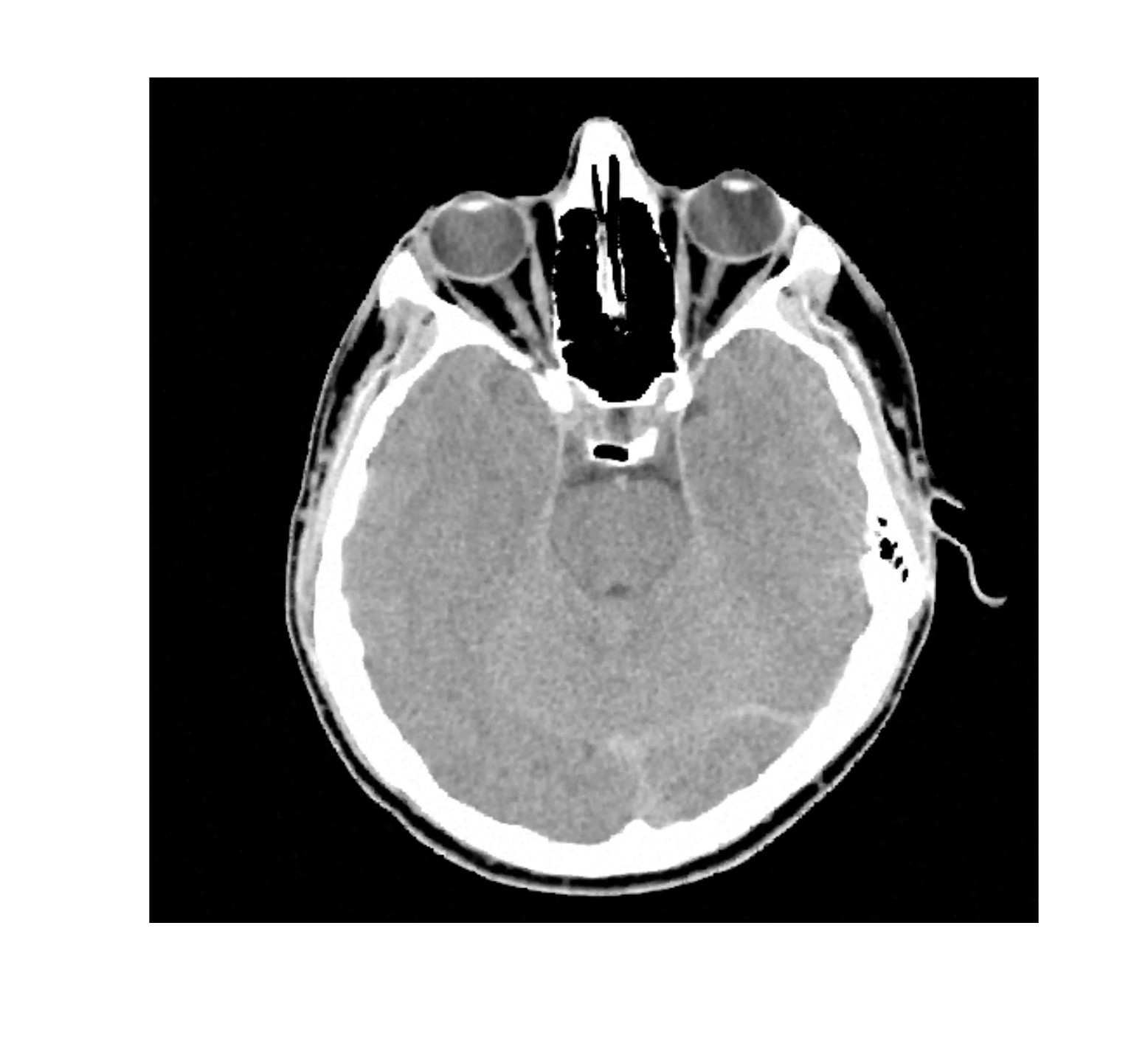}}%
\qquad
\subfloat[][FISTA ($chit = 20$): $\mathrm{F} = 198692.14$ \& $\mathrm{PSNR} = 30.24$]{\includegraphics[width=4.9cm]{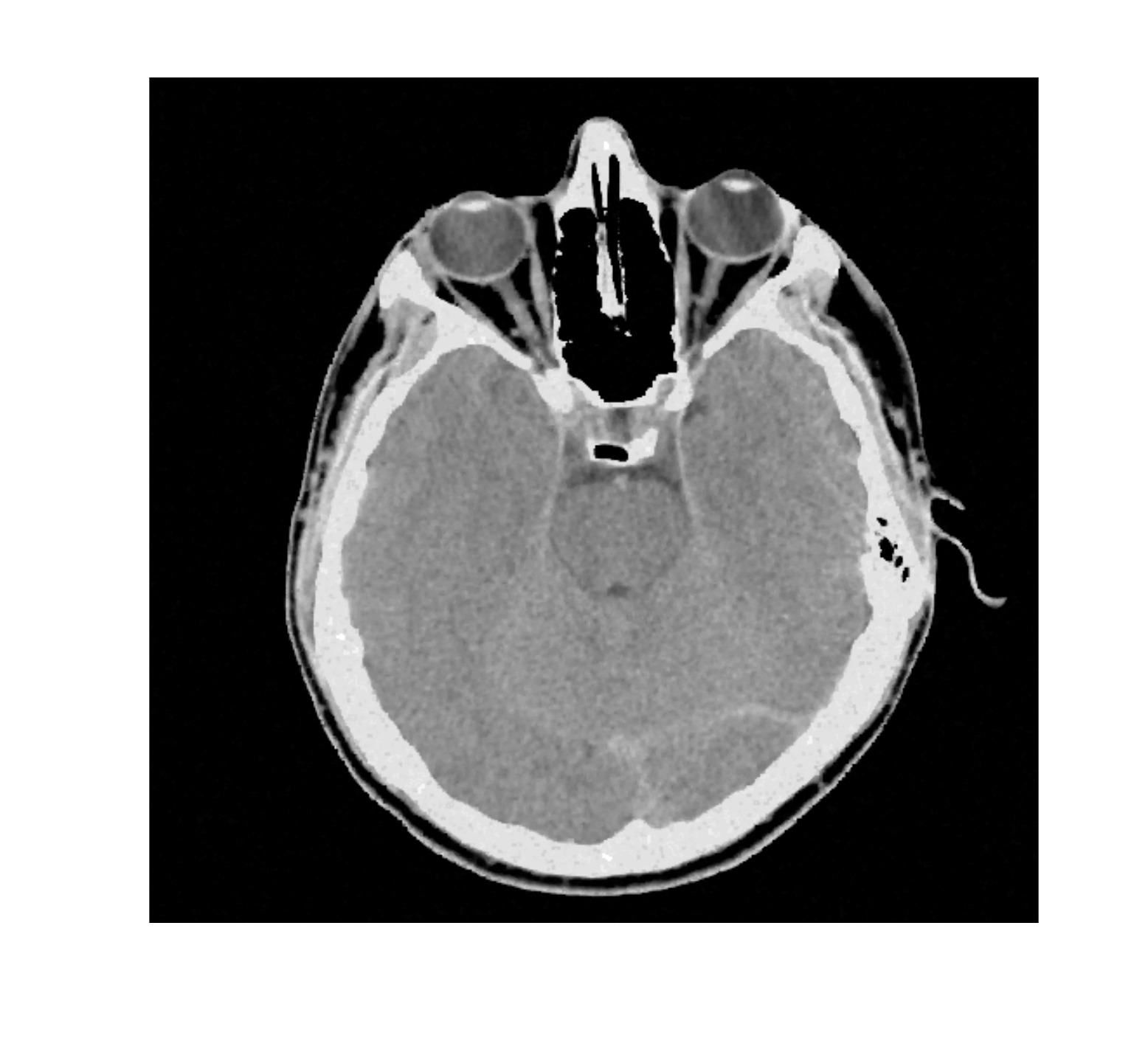}}

\caption{Inpainting of the $512 \times 512$ Head CT image by IST, TwIST, FISTA, OSGA when they stopped after 50 seconds of the running time.}%
\label{fig:cont}
\end{figure}

Similar to denoising, Figure 3 reports comparisons with some relative errors for step points and function values and also ISNR. These results display that IST is not comparable with others, and also OSGA is superior to IST, TwIST and FISTA regarding all considered measures. Since the data is not scaled in the implementation, PSNR is computed by 
\[
\mathrm{PSNR} = 20 \log_{10} \left( \frac{255\sqrt{mn}}{\|X - X_0\|_F} \right),
\]
where $X_0$ is the clean image. The definition clearly proposes that a bigger PSNR value means a smaller $\|X - X_0\|_F$, suggesting a better quality for the restored image. Conversely, a smaller PSNR value indicates a bigger $\|X - X_0\|_F$ implying the worse image quality. Figure 5 shows that IST cannot recover the missing data effectively compared with other algorithms. Moreover, the results suggest that OSGA achieves the best function value 187564.02 and the best PSNR 33.13 among the solvers considered.

\subsubsection{Deblurring with isotropic total variation}\label{s.deb}
Image blur is one of the most common problem which happens in the photography and can often ruin the photograph. In digital photography, the motion blur is caused by camera shakes, which is difficult to avoid in many situations. Deblurring has been a fundamental research problem in the digital imaging and can be considered as an inverse problem of the form (\ref{e.inv2}). The same as other inverse problems, this is an inherently ill-conditioned problem leading to the optimization problem of the form (\ref{e.genf}) or (\ref{e.genf1}). Deblurring is generally much harder than denoising because most of algorithms that have been developed for solving this problem need to solve a subproblem of the form (\ref{e.pro}) in each iteration. 

Here, the deblurring problem of the form (\ref{e.TV}) with the isotropic total variation regularizer is considered. As discussed about denoising and inpainting, IST cannot recover images efficiently and is not comparable with the other algorithms considered. Hence, we instead take advantage of the more sophisticated package SpaRSA \cite{WriNF} in our implementations, so the comparison for deblurring involves TwIST, SpaRSA, FISTA and OSGA. Here, the regularization parameter set to $\lambda = 0.05$. 

We now consider deblurring of the $512 \times 512$ blurred/noisy Elaine image. Let $Y$ be a blurred/noisy version of this image generated by $9 \times 9$ uniform blur and adding Gaussian noise with $\mathrm{SNR} = 40$ dB. We recover it using (\ref{e.TV}) by the algorithms in 25 seconds of the running time. The results are demonstrated in Figures 5 and 6. In Figure 5, the algorithms are compared regarding relative errors for step points and function values and also ISNR. In Figures 5, subfigures (a)-(b) indicate that OSGA obtains a better level of accuracy regarding both of these measures, while Figure 5 (d) suggests that OSGA is superior to the others regarding ISNR. To see the constructed results visually, Figure 6 is considered, where it reports last function value and PSNR of the algorithms. It can be seen that OSGA obtains the best function value 79252.60 and the best PSNR 32.23 among the considered algorithms. 

\begin{figure}[h] \label{f.deb1}
\centering
\subfloat[][$rel.~ 1~ vs.~ iterations,~ chit = 5$]{\includegraphics[width=7.7cm]{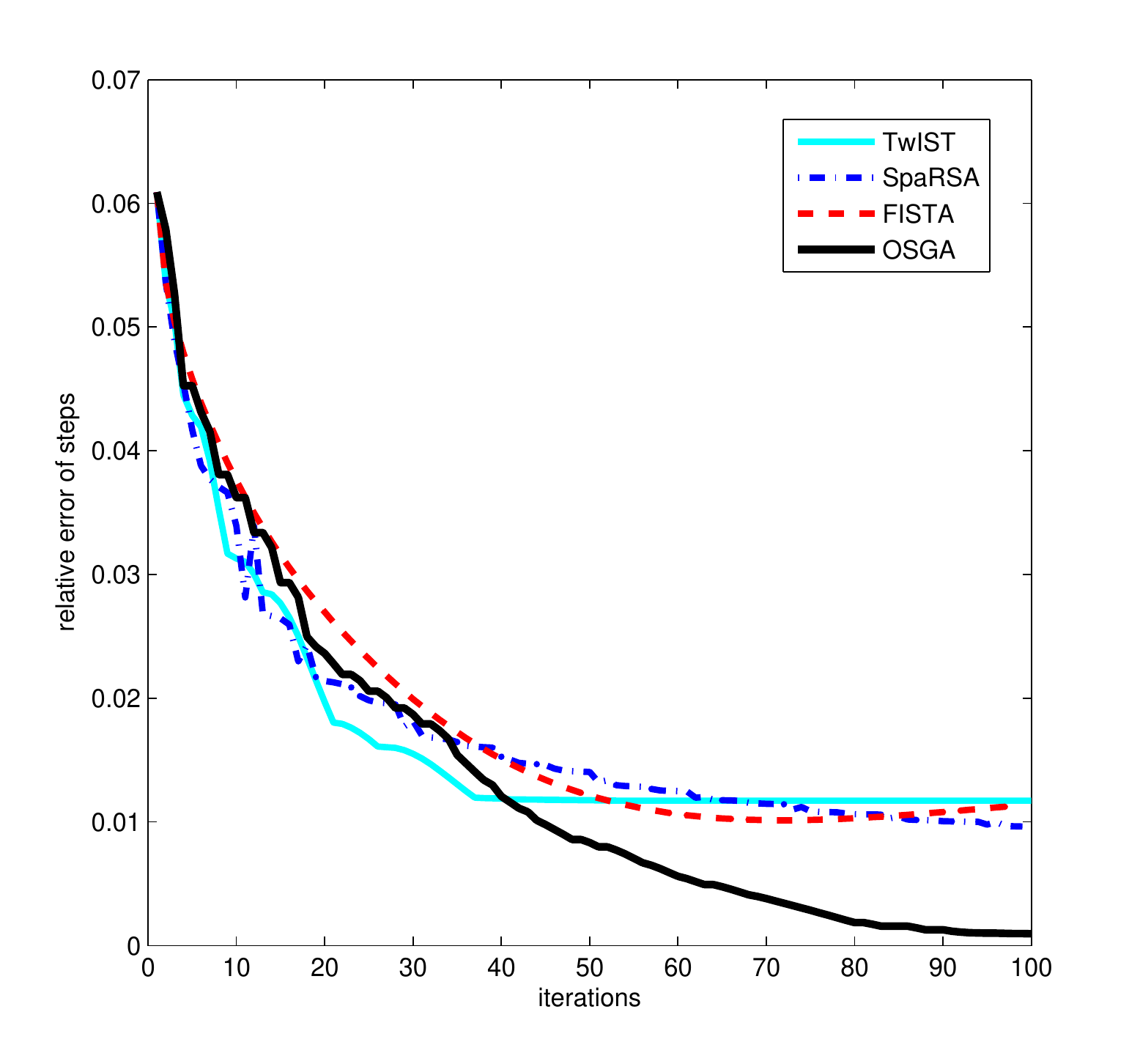}}%
\qquad
\subfloat[][$rel.~ 2~ vs.~ time,~ chit = 5$]{\includegraphics[width=7.7cm]{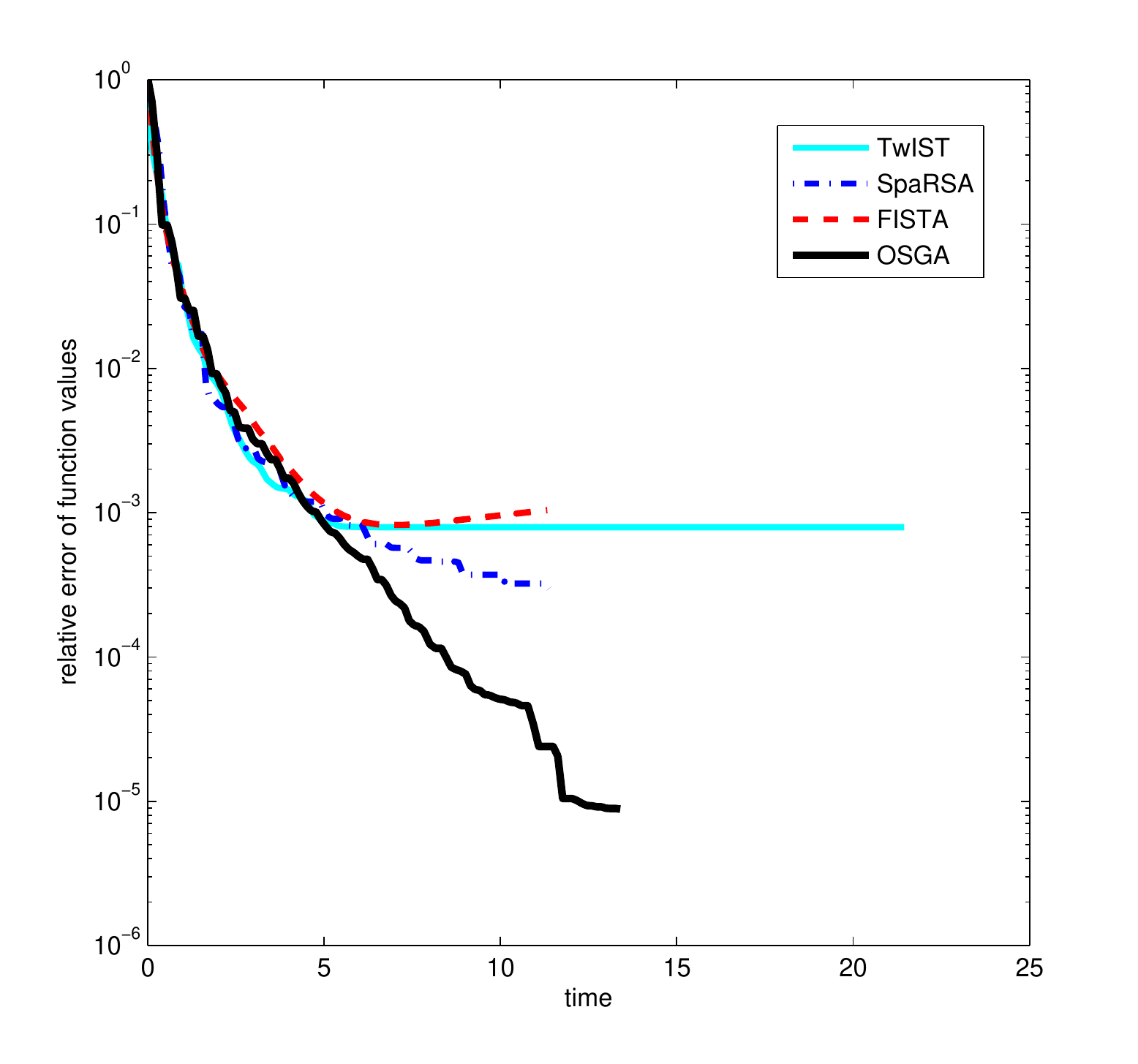}}
\qquad
\subfloat[][$rel.~ 2~ vs.~ iterations,~ chit = 5$]{\includegraphics[width=7.7cm]{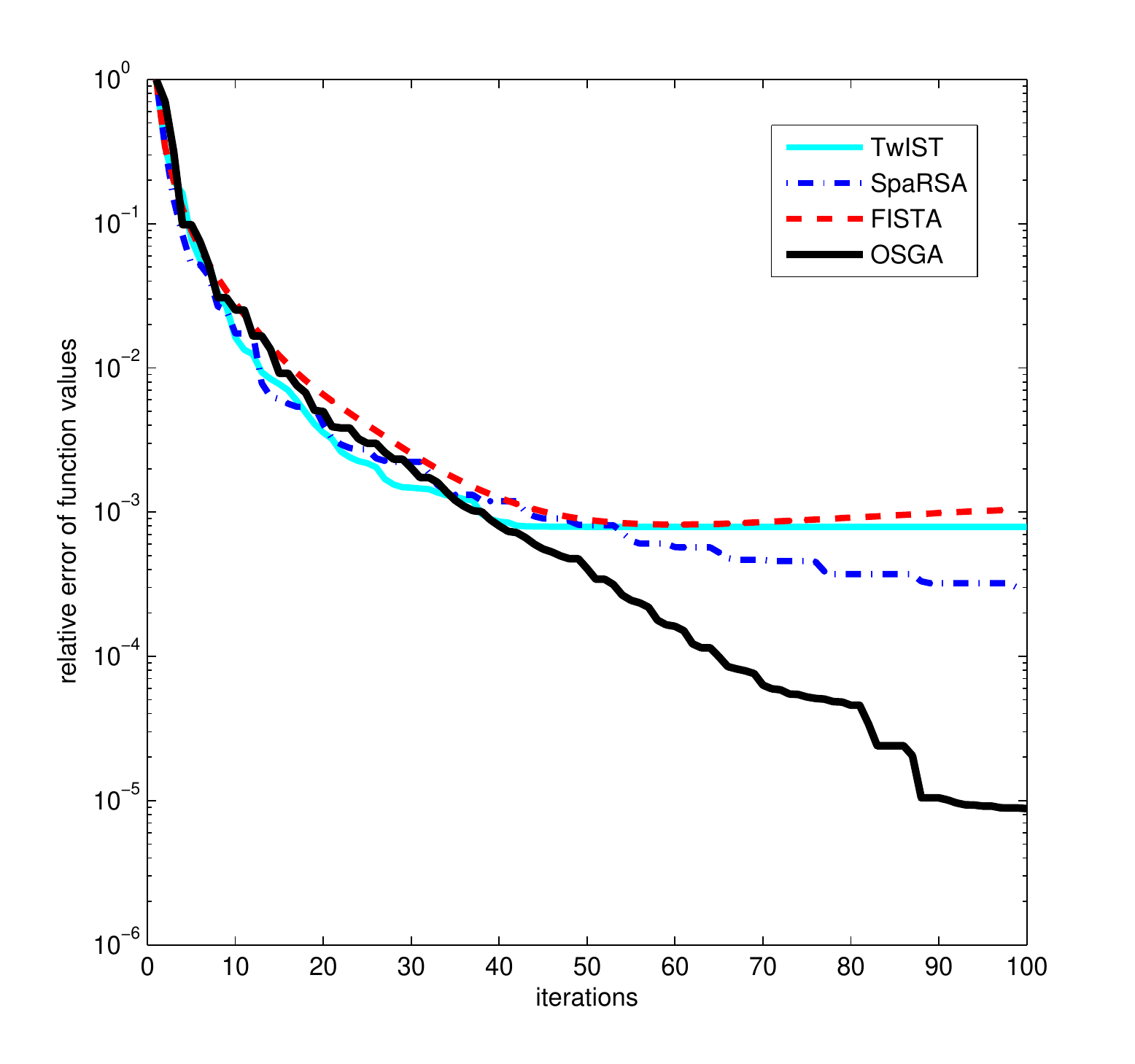}}%
\qquad
\subfloat[][$ISNR~ vs.~ iterations,~ chit = 5$]{\includegraphics[width=7.7cm]{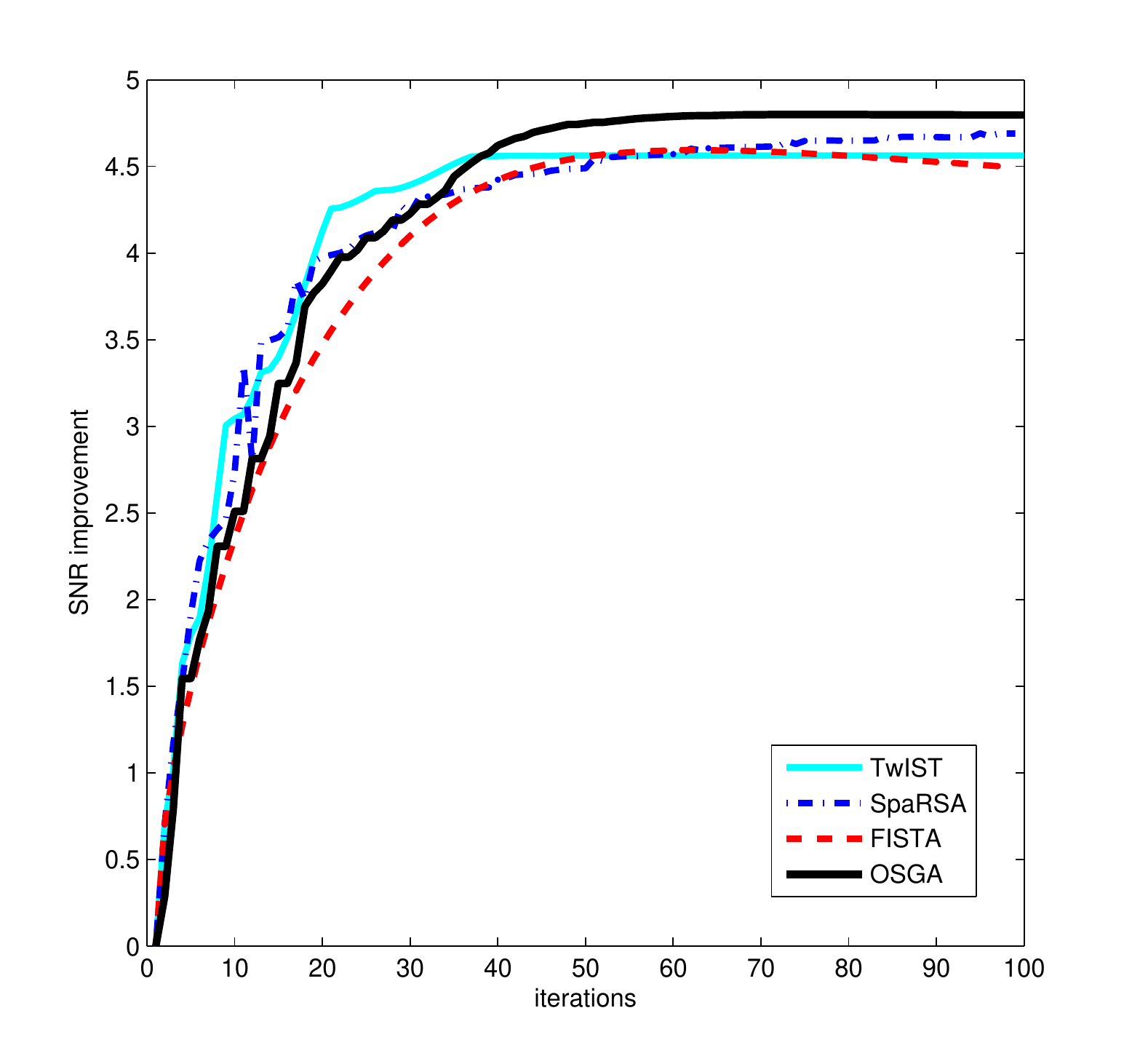}}

\caption{A comparison among TwIST, SpaRSA, FISTA and OSGA for deblurring the $512 \times 512$ Elaine image with $9 \times 9$ uniform blur and  Gaussian noise with $\mathrm{SNR} = 40$ dB. TwIST, SpaRSA and FISTA exploit $chit = 5$, and all algorithms are stopped after 100 iterations. In the figure: (a) stands for the relative error $rel.~ 1 = \|x_k - x^*\|_2/\|x^*\|_2$ of points versus iterations, (b) shows the relative error  $rel.~ 2 = (f_k - f^*)/(f_0 - f^*)$ of function values versus time, (c) demonstrates $rel. 2$ versus iterations, and (d) depicts ISNR (\ref{e.isnr}) versus iterations.}%
\label{fig:cont}%
\end{figure}

\begin{figure}[H] \label{f.deb2}
\centering
\subfloat[][Original image]{\includegraphics[width=7.5cm]{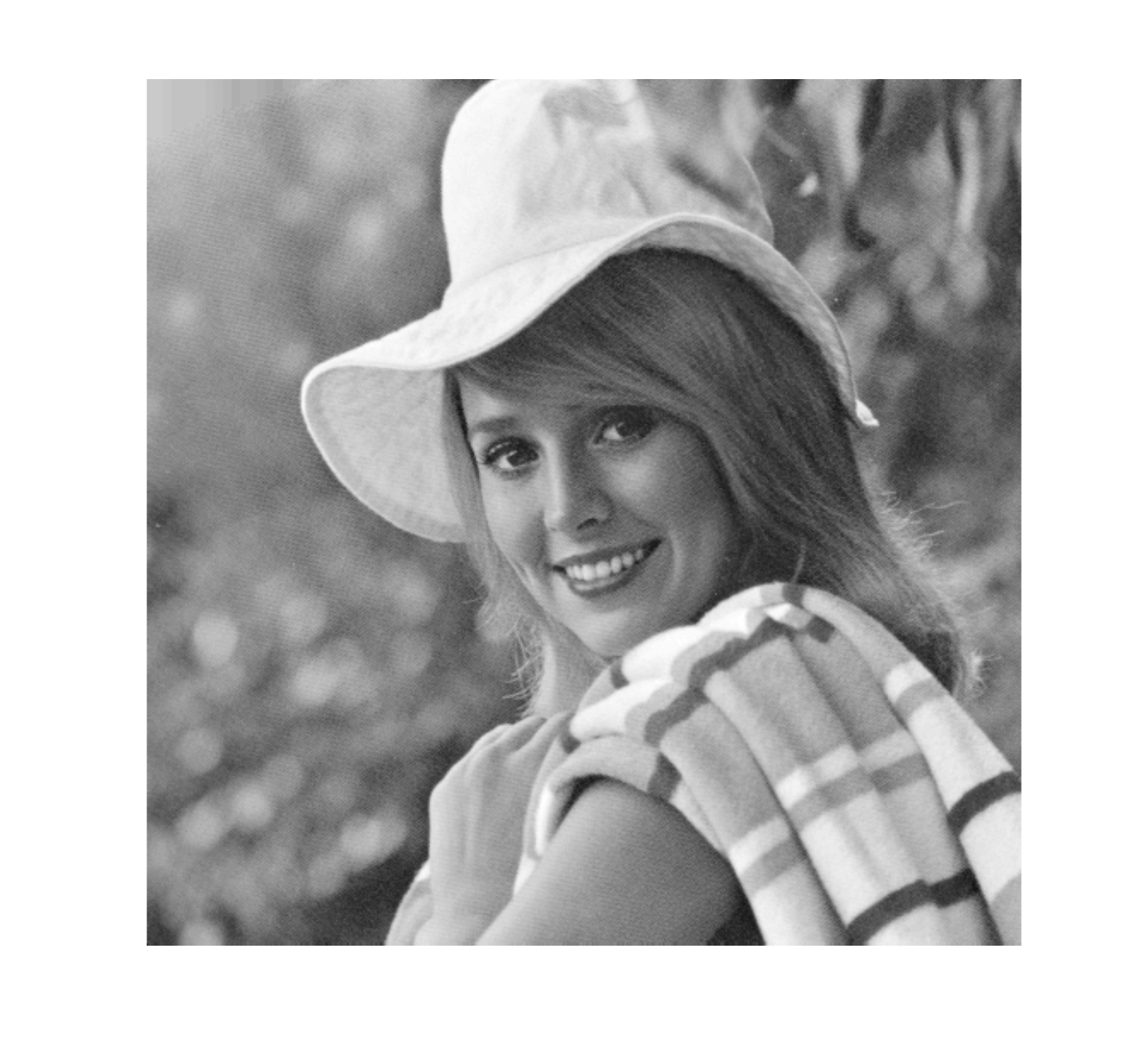}}%
\qquad
\subfloat[][Blurred/noisy image]{\includegraphics[width=7.5cm]{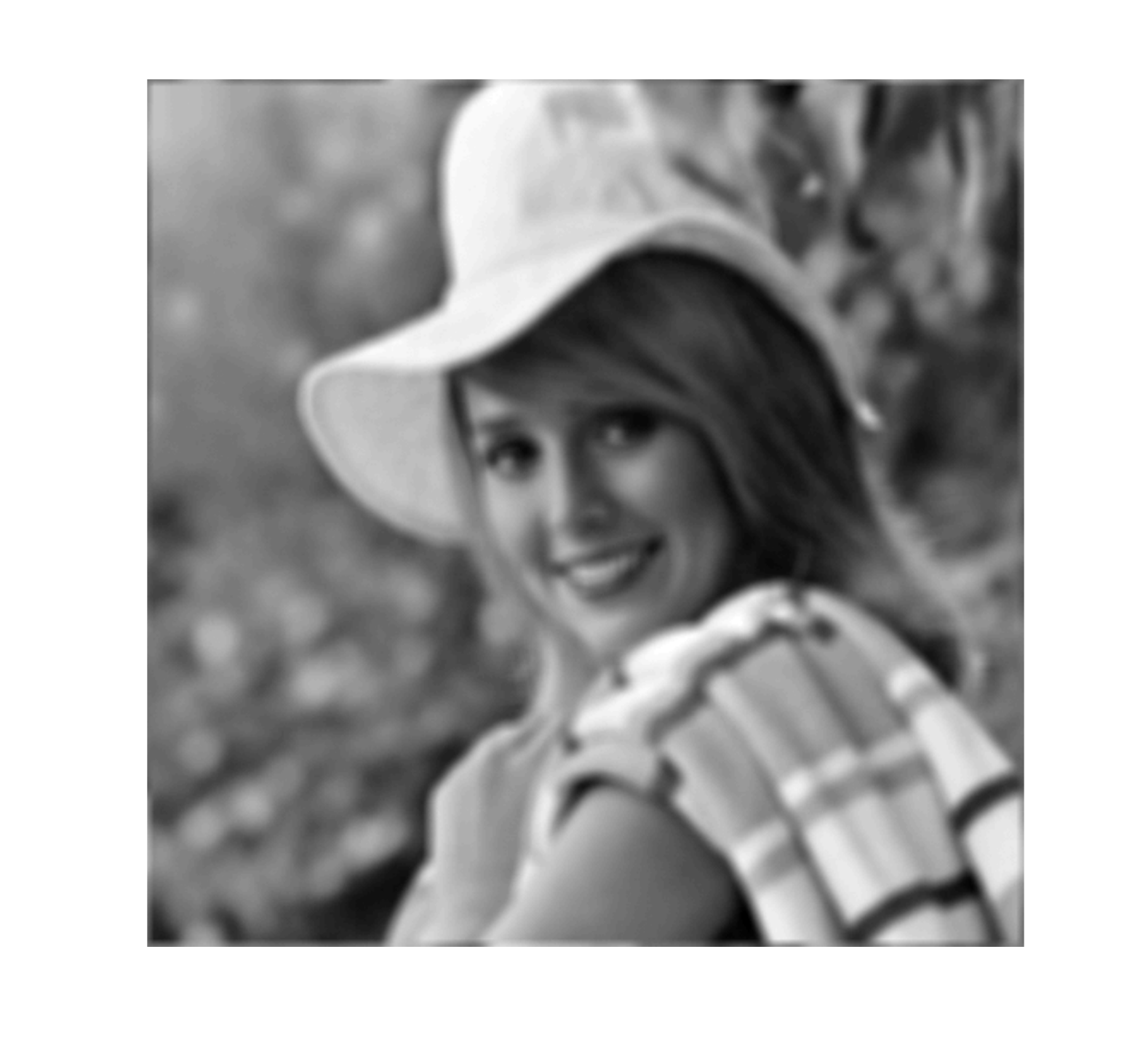}}
\qquad
\subfloat[][TwIST: $f = 80733.24,~ \mathrm{PSNR} = 31.99,~ \mathrm{T} = 20.07$]{\includegraphics[width=7.5cm]{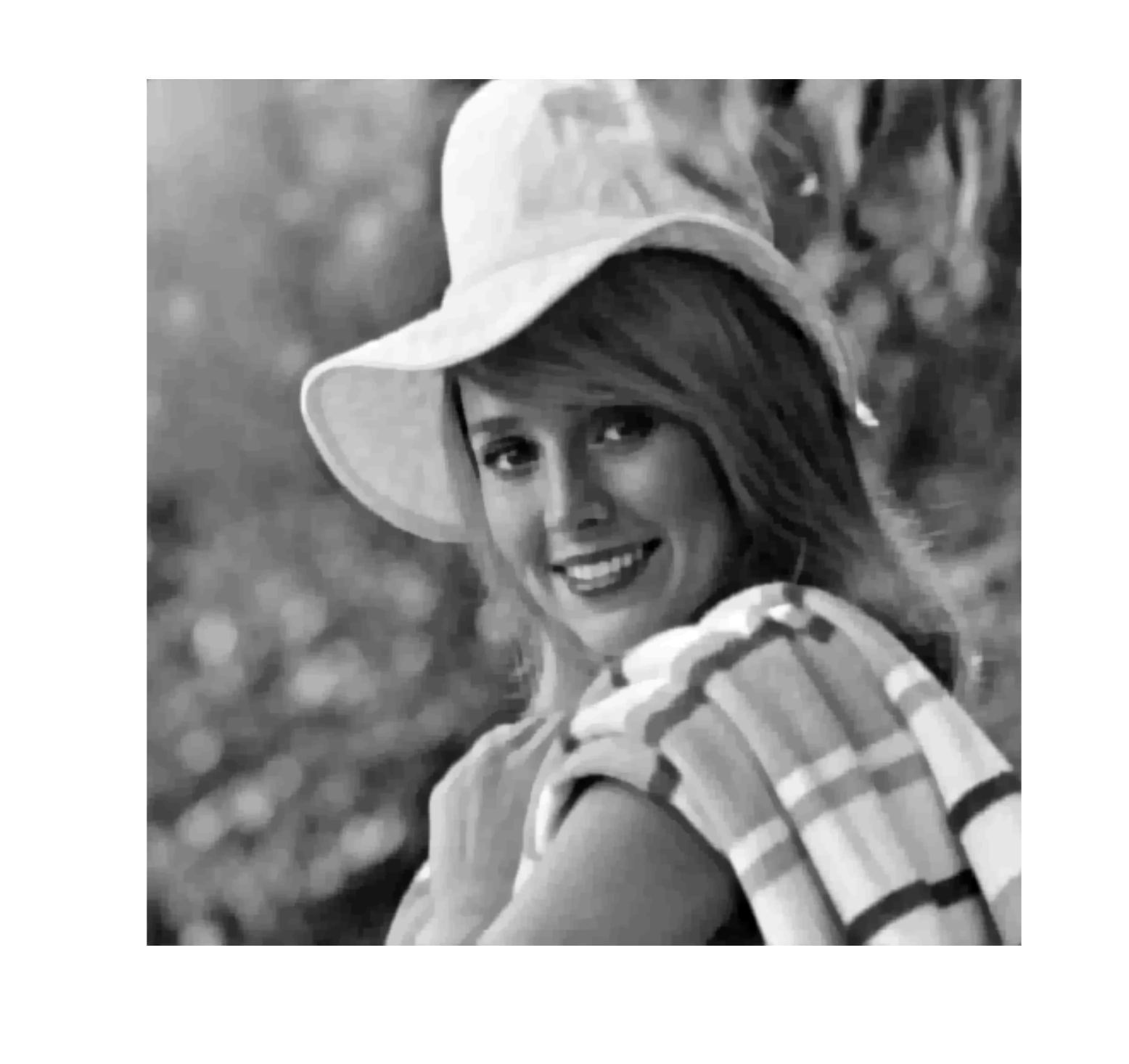}}%
\qquad
\subfloat[][SpaRSA: $f = 79698.36,~ \mathrm{PSNR} = 32.12,~ \mathrm{T} = 12.16$]{\includegraphics[width=7.5cm]{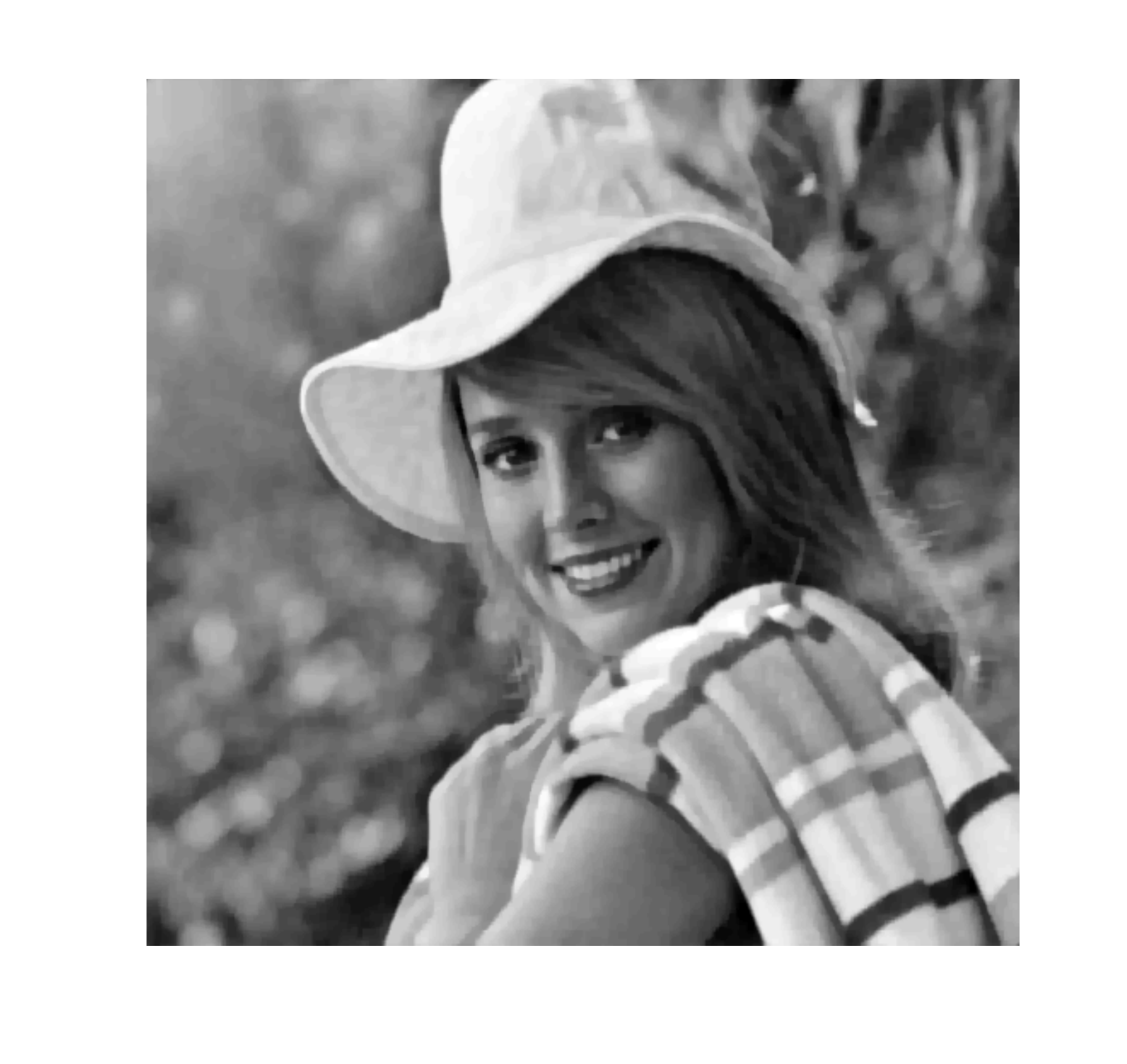}}
\qquad
\subfloat[][FISTA: $f = 80795.59,~ \mathrm{PSNR} = 31.92,~ \mathrm{T} = 11.84$]{\includegraphics[width=7.5cm]{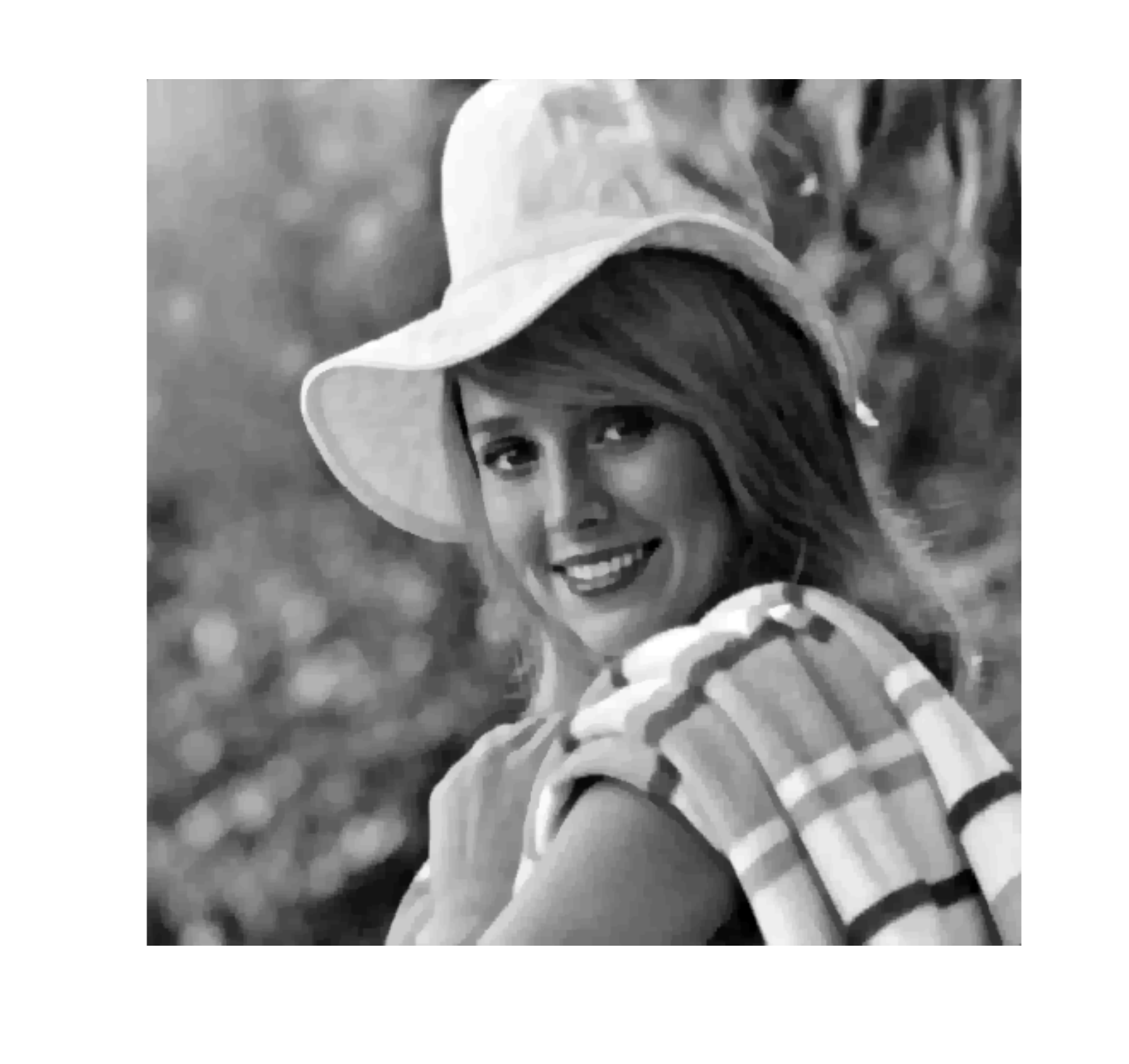}}%
\qquad
\subfloat[][OSGA: $f = 79252.60,~ \mathrm{PSNR} = 32.23,~ \mathrm{T} = 14.82$]{\includegraphics[width=7.5cm]{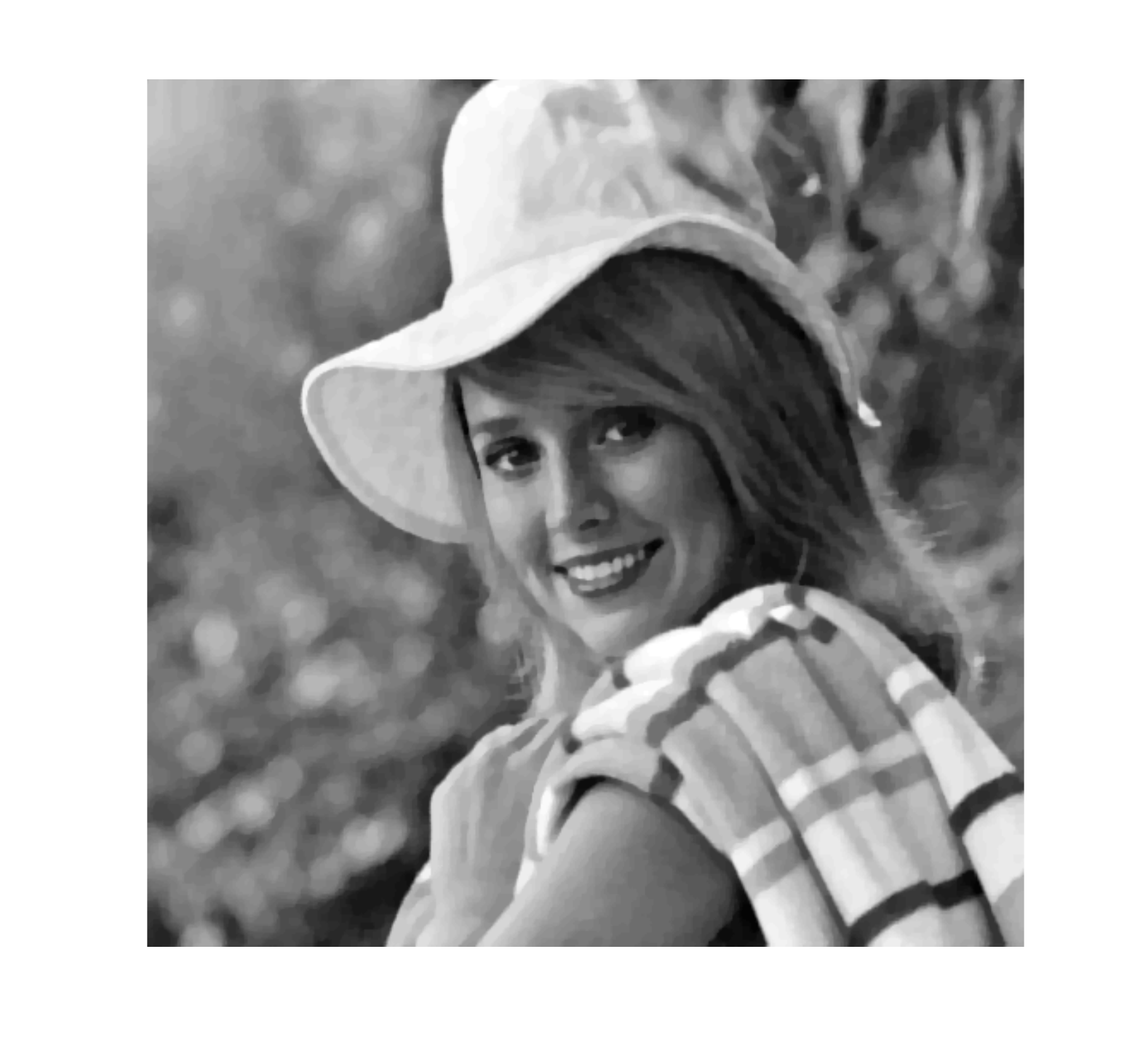}}
\caption{Deblurring of the $512 \times 512$ Elaine image with $9 \times 9$ uniform blur and  Gaussian noise with $\mathrm{SNR} = 40$ dB by TwIST, SpaRSA, FISTA and OSGA. The algorithms stopped after 100 iterations, and the number of Chambolle iterations for TwIST, SpaRSA and FISTA is set to $chit = 5$.}%
\label{fig:cont}
\end{figure}

We finally consider 72 widely used test images to do more comparisons, where blurred/noisy images are generated by $9 \times 9$ uniform blur and adding Gaussian noise with $\mathrm{SNR} = 40$ dB. The images recovered by TwIST, SpaRSA, FISTA and OSGA, where the algorithm are stopped after 100 iterations. The results of reconstructions are summarized in Table 1, where  we report the best function value (F), PSNR and CPU time (T) suggesting OSGA outperforms the others. The results are illustrated in Figure 7, where the comparisons are based on the performance profile of Dolan and Mor\'{e} in \cite{DolM}. Figure 7 shows that OSGA wins with 84\% and 93\% score for function values and PSNR among all others. It also indicates that all algorithms recover images successfully.

\begin{figure}[H] \label{f.deb3}
\centering
\subfloat[][performance profile for function values]{\includegraphics[width=12cm]{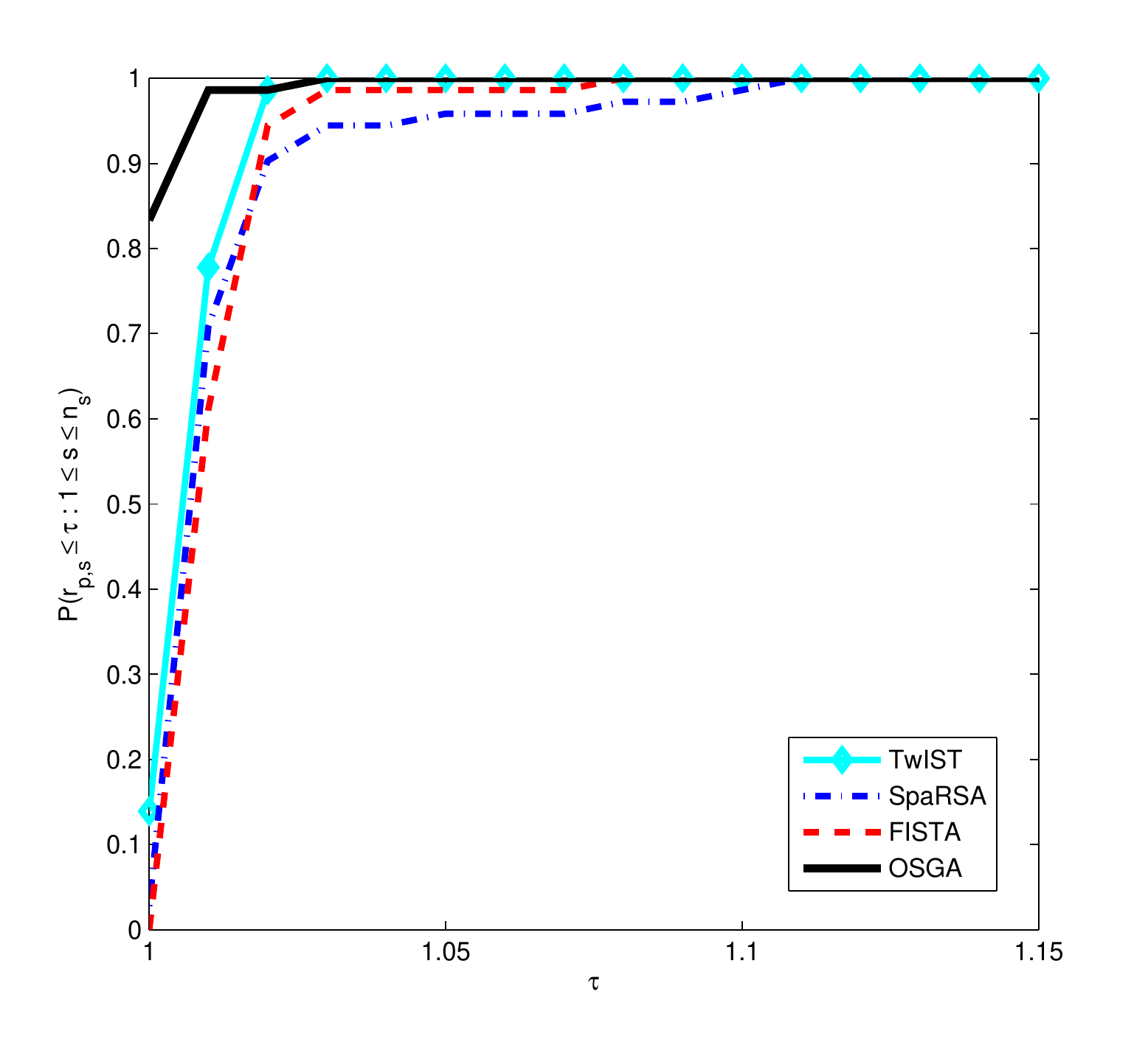}}%
\qquad
\subfloat[][performance profile for PSNR]{\includegraphics[width=12cm]{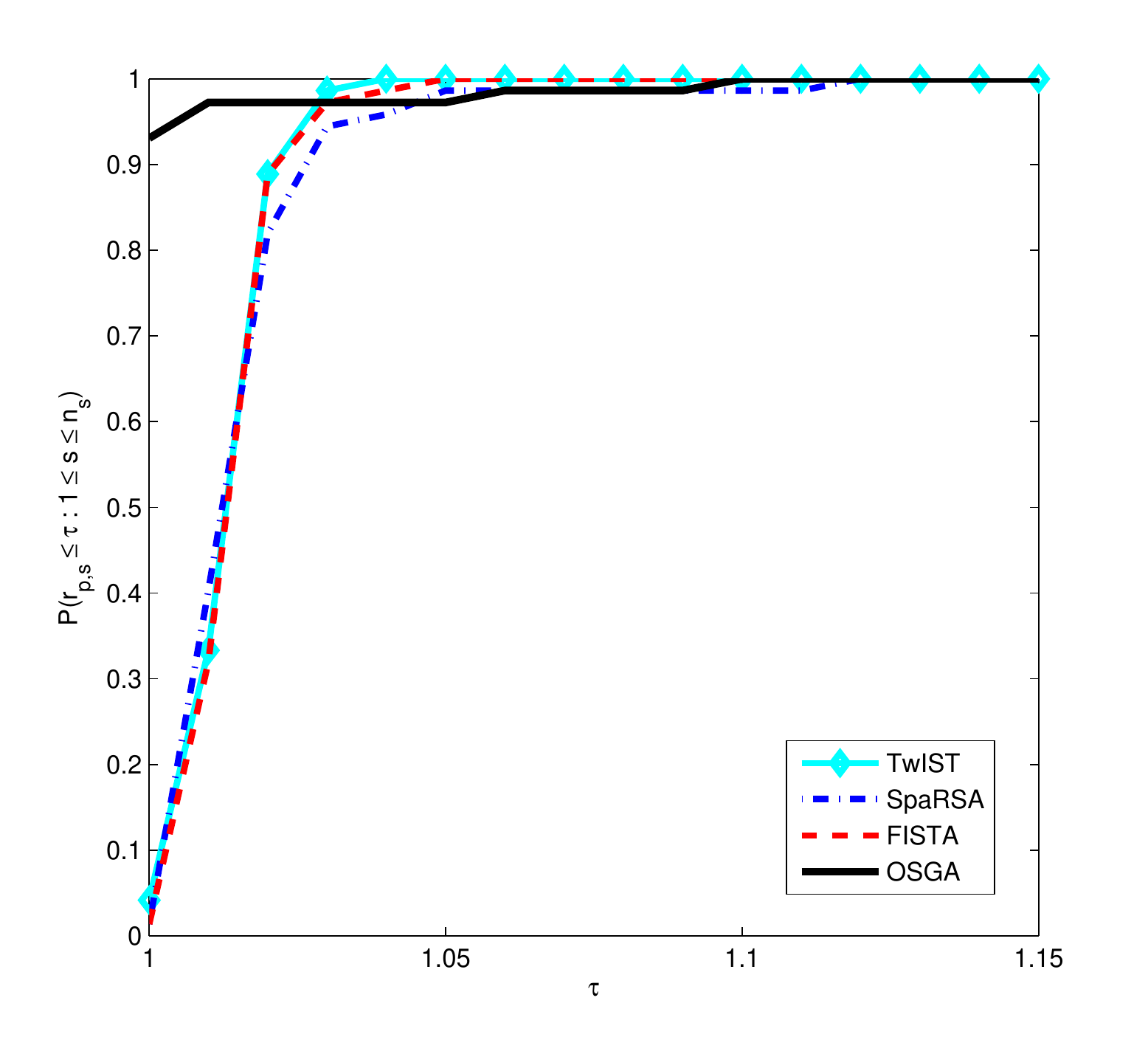}}

\caption{Deblurring of 72 test images with isotropic total variation using TwIST, SpaRSA FISTA, OSGA in 25 seconds of the running time.}%
\label{fig:cont}%
\end{figure}

\renewcommand{\arraystretch}{1.2}
\begin{landscape}
\small
\begin{longtable}{llllllllllllllll}
\multicolumn{5}{l}
{Table 1. Deblurring with isotropic TV ($\mathrm{F}$, $\mathrm{T}$ \& $\mathrm{PSNR}$)}\\
\cmidrule{1-14} 

\multicolumn{1}{l}{}        &\multicolumn{1}{l}{} 
&\multicolumn{1}{l}{TwIST}  &\multicolumn{1}{l}{} &\multicolumn{1}{l}{}\hspace{10mm}
&\multicolumn{1}{l}{SpaRSA} &\multicolumn{1}{l}{} &\multicolumn{1}{l}{}\hspace{10mm}
&\multicolumn{1}{l}{FISTA}  &\multicolumn{1}{l}{} &\multicolumn{1}{l}{}\hspace{10mm}
&\multicolumn{1}{l}{OSGA}   &\multicolumn{1}{l}{} &\multicolumn{1}{l}{}\\
\cmidrule(lr){3-5} \cmidrule(lr){6-8} \cmidrule(lr){9-11} \cmidrule(lr){12-14}

\multicolumn{1}{l}{Image name}      &\multicolumn{1}{l}{Dimension}

&\multicolumn{1}{l}{$\mathrm{F}$}   &\multicolumn{1}{l}{$\mathrm{PSNR}$} &\multicolumn{1}{l}{$\mathrm{T}$}
&\multicolumn{1}{l}{$\mathrm{F}$}   &\multicolumn{1}{l}{$\mathrm{PSNR}$} &\multicolumn{1}{l}{$\mathrm{T}$}
&\multicolumn{1}{l}{$\mathrm{F}$}   &\multicolumn{1}{l}{$\mathrm{PSNR}$} &\multicolumn{1}{l}{$\mathrm{T}$}
&\multicolumn{1}{l}{$\mathrm{F}$}   &\multicolumn{1}{l}{$\mathrm{PSNR}$} &\multicolumn{1}{l}{$\mathrm{T}$} \\
\cmidrule{1-14}
\endfirsthead
\multicolumn{5}{l}%
{Table 1. Deblurring results ($\mathrm{F}$, $\mathrm{T}$ \& $\mathrm{PSNR}$) (\textit{continued})}\\[5pt]
\hline
\endhead
\hline
\endfoot
\endlastfoot
Aerial              & $256 \times 256$ &53380.99&24.99&5.80&52966.30&25.14&3.56&53513.56&25.10&3.33&52832.73&25.66&3.72\\
Airplane            & $256 \times 256$ &16527.67&33.75&6.50&16563.06&33.72&3.39&16549.09&33.69&3.28&16630.05&34.36&3.62\\
Cameraman           & $256 \times 256$ &37077.60&27.39&4.94&37230.96&27.11&3.28&37096.87&27.44&3.28&37089.05&27.82&3.80\\
Chemical plant      & $256 \times 256$ &37890.81&26.80&6.64&37902.64&26.87&3.48&37934.72&26.80&3.32&37562.79&27.25&3.72\\
Clock               & $256 \times 256$ &31581.71&31.07&6.04&31740.38&30.95&4.12&31604.08&31.18&3.33&31502.30&31.64&4.28\\
Fingerprint 1       & $256 \times 256$ &54289.58&20.98&7.06&55070.52&21.12&3.68&54374.17&20.98&3.46&54256.89&21.22&4.23\\
Lena                & $256 \times 256$ &30700.61&28.33&7.98&30610.62&28.31&4.12&30730.93&28.39&3.37&30468.62&28.73&4.19\\
Moon surface        & $256 \times 256$ &16648.60&30.22&8.76&16552.11&30.29&4.39&16718.75&30.27&3.68&16505.01&30.54&4.67\\
Pattern 1           & $256 \times 256$ &231150.20&18.07&4.48&232285.41&17.85&4.32&231937.14&18.09&3.61&232025.68&18.17&4.13\\
Pattern 2           & $256 \times 256$ &176.53&48.90&7.27&190.83&48.92&3.86&186.27&48.90&3.41&172.69&48.89&4.26\\
Star                & $256 \times 256$ &23482.48&24.31&6.06&23596.24&24.26&4.04&23514.17&24.37&3.81&23370.80&24.49&4.37\\
Aerial              & $512 \times 512$ &155827.40&25.94&23.55&166304.12&25.93&17.01&156015.30&26.01&12.86&154008.47&26.51&16.53\\
Airfield            & $512 \times 512$ &169428.35&26.06&23.87&169301.74&25.97&16.91&169528.18&26.14&13.53&168303.91&26.63&17.68\\
Airplane 1          & $512 \times 512$ &105947.22&32.00&23.43&105942.26&32.06&14.29&106078.89&32.06&13.57&105076.82&32.70&18.18\\
Airplane 2          & $512 \times 512$ &29318.44&35.86&19.22&29609.95&35.14&14.83&29329.10&35.90&13.44&29285.75&35.79&18.10\\
APC                 & $512 \times 512$ &45314.07&33.18&24.09&45128.61&33.20&14.60&45425.04&33.22&13.48&44985.53&33.48&17.97\\
Barbara             & $512 \times 512$ &114725.54&25.45&27.10&114235.16&25.47&14.39&115004.02&25.43&13.84&113697.78&25.60&18.17\\
Blobs               & $512 \times 512$ &191940.70&20.63&24.31&191945.61&20.59&15.92&192059.89&20.63&14.03&192043.10&20.85&17.21\\
Blonde woman        & $512 \times 512$ &96330.41&29.63&30.11&95947.22&29.71&15.40&96567.05&29.65&13.66&95219.10&30.02&17.74\\
Boat                & $512 \times 512$ &118404.12&29.87&27.41&118194.91&29.82&16.12&118590.52&29.90&14.02&117624.44&30.33&16.84\\
Cameraman           & $512 \times 512$ &111579.37&32.98&29.51&111655.52&32.87&15.86&111656.63&33.00&13.30&110900.70&33.53&17.56\\
Car \& APCs 1       & $512 \times 512$ &75501.10&32.67&22.48&74826.32&32.90&13.88&75868.96&32.59&13.65&74614.10&32.98&17.24\\
Car \& APCs 2       & $512 \times 512$ &69193.23&32.52&30.35&68645.53&32.65&15.55&69636.88&32.43&13.84&68233.64&32.85&17.23\\
Clown               & $512 \times 512$ &137437.13&31.71&23.04&137241.10&31.77&13.31&137550.26&31.75&12.79&136517.93&32.12&16.65\\
Crowd 1             & $512 \times 512$ &127463.77&30.55&25.33&127563.32&30.72&13.58&127708.39&30.61&12.21&125715.43&31.33&15.61\\
Crowd 2             & $512 \times 512$ &213890.70&28.16&18.55&214814.30&28.35&12.29&214494.66&28.16&12.56&210665.21&29.00&15.37\\
Darkhair woman      & $512 \times 512$ &85530.72&37.31&24.19&85441.92&37.49&12.47&85854.12&36.96&13.09&84848.24&37.77&15.31\\
Dollar              & $512 \times 512$ &249508.76&21.49&18.39&273059.31&21.42&12.34&249686.20&21.50&12.52&248817.22&21.68&15.34\\
Elaine              & $512 \times 512$ &80406.78&32.01&20.28&79667.61&32.13&12.45&80765.32&31.94&14.08&79206.17&32.24&14.84\\
Fingerprint         & $512 \times 512$ &256043.16&26.09&22.58&256803.08&26.19&12.34&256811.98&25.85&12.66&253486.89&26.49&15.71\\
Flinstones          & $512 \times 512$ &245122.90&25.64&17.50&246272.55&25.45&12.31&245410.87&25.66&11.99&243375.64&26.47&15.50\\
Girlface            & $512 \times 512$ &82611.48&33.00&24.73&82842.10&33.06&12.18&82896.35&32.89&12.50&81871.62&33.44&15.63\\
Goldhill            & $512 \times 512$ &110067.92&30.69&23.78&109700.41&30.78&12.49&110273.58&30.71&12.27&109353.67&31.06&15.66\\
Head CT             & $512 \times 512$ &178265.55&30.91&16.10&179561.46&30.05&12.63&178400.71&30.93&11.94&177648.72&31.52&15.58\\
Houses              & $512 \times 512$ &250441.66&25.79&17.79&262200.38&25.69&12.47&250589.16&25.83&11.62&250375.51&25.97&15.74\\
Kiel                & $512 \times 512$ &155925.23&25.82&16.35&156039.11&25.76&12.83&156013.11&25.89&11.88&155847.17&26.01&15.81\\
Lake                & $512 \times 512$ &154117.24&28.88&24.72&154535.79&28.86&13.30&154292.14&28.95&13.10&152700.95&29.42&15.49\\
Lena                & $512 \times 512$ &91492.93&31.80&24.05&92149.44&31.78&12.50&91620.97&31.80&12.02&90544.16&32.27&14.87\\
Lena numbers        & $512 \times 512$ &258205.38&20.82&21.23&260738.06&20.57&15.70&258430.66&20.88&12.07&258452.85&20.92&15.51\\
Liftingbody         & $512 \times 512$ &43626.15&37.40&24.98&43713.46&37.24&14.06&43668.81&37.47&11.44&43420.69&37.69&15.68\\
Lighthouse          & $512 \times 512$ &169572.40&27.50&19.19&169527.68&27.44&14.50&169664.95&27.55&11.82&169081.03&27.72&15.55\\
Livingroom          & $512 \times 512$ &126804.63&29.83&24.84&126502.75&29.84&13.65&127007.96&29.89&11.82&126280.35&30.18&15.78\\
Mandril             & $512 \times 512$ &149477.52&25.90&23.63&150370.73&26.01&12.48&149889.46&25.88&12.82&148035.40&26.30&15.32\\
MRI spine           & $512 \times 512$ &85213.65&35.59&21.32&85179.01&35.61&13.12&85397.57&35.51&12.20&84823.91&36.10&15.75\\
Peppers             & $512 \times 512$ &101362.85&32.89&19.18&102312.47&32.86&13.45&101591.74&32.89&13.44&100470.74&33.44&15.50\\
Pirate              & $512 \times 512$ &108105.61&29.83&20.98&107788.51&29.91&14.87&108301.20&29.90&12.26&107122.20&30.31&15.69\\
Smiling woman       & $512 \times 512$ &110681.19&33.14&29.36&110432.93&33.17&13.95&110858.04&33.16&12.12&109880.54&33.65&15.66\\
Squares             & $512 \times 512$ &154177.76&50.49&17.53&154454.47&48.95&13.19&154228.55&50.21&12.22&154711.06&47.79&15.72\\
Tank 1              & $512 \times 512$ &63474.30&31.45&30.85&63215.81&31.48&14.99&63643.84&31.44&11.73&62876.00&31.70&15.71\\
Tank 2              & $512 \times 512$ &85867.65&29.39&23.93&85277.68&29.49&14.39&86445.51&29.38&11.89&84963.72&29.70&15.55\\
Tank 3              & $512 \times 512$ &93860.07&29.68&30.64&93348.41&29.77&15.40&94147.81&29.66&12.91&93238.90&29.96&15.79\\
Truck               & $512 \times 512$ &73437.12&31.78&23.54&73088.97&31.93&12.35&73768.38&31.77&12.25&72621.40&32.17&15.44\\
Truck \& APCs 1     & $512 \times 512$ &106703.13&28.69&27.71&106087.95&28.79&14.49&107119.20&28.69&12.18&105751.32&29.00&15.74\\
Truck \& APCs 2     & $512 \times 512$ &105112.40&28.74&28.27&104405.27&28.84&14.09&105581.13&28.71&12.44&104103.47&29.06&15.46\\
Washington DC 1     & $512 \times 512$ &255921.91&20.11&23.93&255818.43&20.03&12.68&256223.57&20.10&12.06&255585.39&20.39&15.78\\
Washington DC 2     & $512 \times 512$ &257153.15&20.06&27.90&256665.77&20.02&15.31&257413.32&20.08&12.52&256847.89&20.36&15.50\\
Washington DC 3     & $512 \times 512$ &261572.25&20.01&24.69&266971.93&20.03&12.64&261772.59&20.04&11.78&261064.29&20.33&15.74\\
Washington DC 4     & $512 \times 512$ &222243.70&20.81&22.17&220980.31&20.85&12.40&222443.50&20.82&11.82&222099.26&21.06&15.33\\
Zelda               & $512 \times 512$ &71117.95&35.27&26.02&71246.57&35.60&12.69&71547.26&35.08&12.10&70038.80&35.77&15.65\\
Dark blobs 1        & $600 \times 600$ &201150.09&23.93&29.71&201437.41&23.91&21.04&201282.62&23.95&18.03&201705.89&24.09&21.04\\
Dark blobs 2        & $600 \times 600$ &197734.24&23.85&27.26&198225.38&23.83&19.98&197847.53&23.86&17.43&198123.72&24.00&21.49\\
House               & $600 \times 600$ &129529.71&33.80&28.46&129835.46&33.82&17.79&129693.50&33.83&17.41&129342.37&33.92&21.52\\
Ordered matches     & $600 \times 600$ &208297.87&29.22&35.19&207720.39&29.30&19.67&208547.46&29.18&17.09&207095.69&29.40&21.63\\
Random matches      & $600 \times 600$ &219117.86&28.00&32.76&220633.95&28.08&17.94&219692.19&27.90&17.45&218101.17&28.30&21.35\\
Rice &92142.42      & $600 \times 600$ &40.12&31.82&92885.81&40.82&17.67&92779.78&39.56&17.34&90891.80&41.15&21.54\\
Shepp-logan phantom & $600 \times 600$ & 83431.60&35.32&21.89&85917.89&31.69&18.52&83775.26&34.38&17.22&85334.45&32.38&21.27\\
Airport             & $1024 \times 1024$ &423652.87&28.25&75.10&422352.54&28.25&53.99&424082.24&28.32&59.41&420651.39&28.63&59.71\\
Pentagon            & $1024 \times 1024$ &324683.85&28.48&97.22&323003.05&28.54&51.01&325787.77&28.51&47.50&320916.50&28.81&59.67\\
Pirate              & $1024 \times 1024$ &462035.52&30.22&88.40&460696.14&30.24&45.67&462855.60&30.23&47.90&457747.96&30.68&59.85\\
Rose                & $1024 \times 1024$ & 289127.64&37.12&86.41&289405.02&37.14&45.00&289515.55&36.99&48.80&288520.76&37.32&59.31\\
Testpat             & $1024 \times 1024$ &921586.14&21.89&56.31&927968.42&21.68&44.94&923132.76&21.88&48.89&927122.49&21.84&61.62\\
Washington DC       & $1024 \times 1024$ &816333.83&23.60&94.63&812073.87&23.67&49.28&817252.58&23.61&51.87&811272.73&23.95&63.54\\
\hline
Average             &  & 159823.01 & 29.36 & 26.54 & 160619.72 & 29.30 & 15.68 & 160106.63 & 29.34 & 14.86 & 159106.67 & 29.65 & 18.47 \\
    
\hline
\end{longtable}
\end{landscape}

\subsection{Comparisons with first-order methods} \label{s.opt}
As in the first section discussed, there is an excessive gap between the optimal complexity of first-order methods for solving smooth and nonsmooth convex problems to get an $\varepsilon$-solution, where it is $O(1/\sqrt{\varepsilon})$ for smooth problems and $O(1/\varepsilon^2)$ for nonsmooth ones. Inspired by this fact, we consider problems of both cases in the form (\ref{e.genf1}) and compare the best-known first-order algorithms with OSGA.

In spite of the fact that some first-order methods use smoothing techniques to solve nonsmooth problems \cite{BotH1, BotH2, Nes.Sm}, we only consider those that directly solve the original problem. In the other words, we only consider algorithms exploiting the nonsmooth first-order black-box oracle or employing the proximity operators along with the smooth first-order black-box oracle. Hence, during our implementations, the following algorithms are considered
\begin{itemize}
\item PGA  : Proximal gradient algorithm \cite{ParB};
\item NSDSG: Nonsummable diminishing subgradient algorithm \cite{BoyXM};
\item FISTA: Beck and Tebolle's fast proximal gradient algorithm \cite{BecT2};
\item NESCO: Nesterov's composite optimal algorithm \cite{Nes.Co};
\item NESUN: Nesterov's universal gradient algorithm \cite{Nes.Un};
\item NES83: Nesterov's 1983 optimal algorithm \cite{Nes.83};
\item NESCS: Nesterov's constant step optimal algorithm \cite{Nes.Bo};
\item NES05: Nesterov's 2005 optimal algorithm \cite{Nes.Sm}.
\end{itemize}  

The algorithms NES83, NESCS and NES05 originally proposed to solve smooth convex problems, where they use the smooth first-order black-box oracle and attain the optimal complexity of order $O(1/\sqrt{\varepsilon})$. Although, obtaining the optimal complexity bound of smooth problems is computationally very interesting, the considered problem (\ref{e.genf1}) is commonly nonsmooth. Recently, {\sc Lewis} and {\sc Overton} in \cite{LewO} investigated the behaviour of the BFGS method for nonsmooth nonconvex problems, where their results are interesting. Furthermore, {\sc Nesterov} in \cite{Nes.Co} shows that the subgradient of composite functions preserve all important properties of the gradient of smooth convex functions. These facts, motivate us to do some computational investigations on the behaviour of optimal smooth first-order methods in the presence of the nonsmooth first-order black-box oracle, where the oracle passes a subgradient of composite objective functions instead of the gradient of smooth parts. In particular, we here consider Nesterov's 1983 algorithm, \cite{Nes.83}, and adapt it to solve (\ref{e.genf1}) by simply preparing the subgradient of the functions as follows:\\

\begin{algorithm}[H]
\DontPrintSemicolon 
\KwIn{select $z$ such that $z \neq y_0$ and $g_{y_0} \neq g_z$;~ local parameters: $y_0$;~ $\eps>0$;}
\KwOut{$x_k$; $\Psi_{x_k}$;}
\Begin{
    $a_0 \leftarrow 0$;~ $x_{-1} \leftarrow y_0$;\;
    compute $g_{y_0} = g_{\Psi}(y_0)$ using NFO-G;\;
    compute $g_{z} = g_{\Psi}(z)$ using NFO-G;\;
    $\alpha_{-1} \leftarrow \|y_0 - z\| / \|g_{y_0} - g_z \|$;\;
    \While {stopping criteria do not hold}{
        $\hat{\alpha}_k \leftarrow \alpha_{k-1}$;\;
        $\hat{x}_k \leftarrow y_k - \hat{\alpha}_k g_{y_k}$;\;
        compute $\Psi_{\hat{x}_k}$ using NFO-F;\;
        \While {$\Psi_{\hat{x}_k} < \Psi_{y_k} - \frac{1}{2} \hat{\alpha}_k \|g_{y_k}\|^2$}{
            $\hat{\alpha}_k \leftarrow \rho \hat{\alpha}_k$;\;
            $\hat{x}_k \leftarrow y_k - \hat{\alpha}_k g_{y_k}$;\;
            compute $\Psi_{\hat{x}_k}$ using NFO-F;\;
        }        
        $x_k \leftarrow \hat{x}_k$;~ $\Psi_{x_k} \leftarrow \Psi_{\hat{x}_k}$;~ $\alpha_k \leftarrow \hat{\alpha}_k$;\;
        $a_{k+1} \leftarrow \left(1 + \sqrt{4a_k^2 + 1} \right) / 2$;\;
        $y_{k + 1} \leftarrow x_k + (a_k - 1)(x_k - x_{k-1})/a_{k + 1}$;\;
        compute $g_{y_{k+1}} = g_{\Psi}(y_{k+1})$ using NFO-G;\;
    }
}
\caption{ {\sc NES83} Nesterov's 1983 algorithm for multi-term affine composite functions}
\label{algo:max}
\end{algorithm}

In a similar way considered for NES83, the algorithms NESCS and NES05 are respectively adapted from {\sc Nesterov}'s constant step \cite{Nes.Bo} and {\sc Nesterov}'s 2005 algorithms \cite{Nes.Sm}. Using this technique, NES83, NESCS and NES05 are able to solve nonsmooth problems as well, however there is no theory to support their convergence at the moment.  

Among the algorithms considered, PGA, NSDSG, FISTA, NESCO and NESUN are originally introduced to solve nonsmooth problems, where NSDSG only employs the nonsmooth first-order black-box oracle and the others use the smooth first-order black-box oracle together with the proximal mapping. In the sense of complexity theory, NSDSG is just optimal for nonsmooth problems with the order $O(1/\varepsilon^2)$, PGA is not optimal with the complexity of the order $O(1/\varepsilon)$, and FISTA, NESCO and NESUN are optimal with the order $O(1/\sqrt{\varepsilon})$. These complexity bounds clearly imply that NSDSG and PGA are not theoretically comparable with the others, however we still consider them in our comparisons to see how much the optimal first-order methods are computationally interesting than the classical first-order methods, especially for solving large-scale problems in applications. 

To get a clear assessment of OSGA, we compare it with the algorithms only using the nonsmooth first-order oracle (NSDSG, NES83, NESCS, NES05) and with those employing the smooth first-order oracle together with the proximal mapping (PGA, NSDSG, FISTA, NESCO, NESUN) in separate comparisons. Most of these algorithms require the global Lipschitz constant $L$ to determine a step size. While NESCS, NES05, PGA and FISTA use the constant $L$ to determine a step size, NESCO and NESUN get a lower approximation of $L$ and adapt it by  backtracking line searches, for more details see \cite{Nes.Co} and \cite{Nes.Un}. In our implementations of NESCO and NESUN, similar to \cite{BecCG}, we determine the initial estimate $L_0$ by 
\[
L_0 = \frac{\|g(x_0) - g(z_0) \|_*}{\|x_0 - z_0\|},
\] 
where $x_0$ and $z_0$ are two arbitrary points such that $x_0 \neq z_0$ and $g(x_0) \neq g(z_0)$. In \cite{Nes.83}, {\sc Nesterov} used an equivalent scheme to determine an initial step size that can be seen in NES83. For NSDSG, the step size is calculated by $\alpha_0/\sqrt{k}$, where $k$ is the current iteration number and $\alpha_0$ is a positive constant that should be specified as big as possible such that the algorithm is not divergent, see \cite{BoyXM}. All the other parameters of the algorithms are set to those reported in the corresponding papers. 

We here consider solving an underdetermined system $A x = y$, where $A$ is a $m \times n$ random matrix and $y$ is a random $m$-vector. This problem is frequently appeared in applications, and the goal is to determine $x$ by optimization techniques. Considering the ill-conditioned feature of this problem, the most popular optimization models are (\ref{e.tikh}) and (\ref{e.BPD}), where (\ref{e.tikh}) is clearly smooth and (\ref{e.BPD}) is nonsmooth. In each case, we consider both dense and sparse data. More precisely, the dense data is constructed by setting $A = \mathsf{rand(m,n)}$, $y = \mathsf{rand(1,m)}$ and the initial point $x_0 = \mathsf{rand(1,n)}$, and the sparse data is created by $A = \mathsf{sprand(m,n,0.05)}$, $y = \mathsf{sprand(1,m,0.05)}$ and the initial point $x_0 = \mathsf{rand(1,n,0.05)}$. The problems use $m = 5000$ and $n = 10000$ and employ $\lambda = 1$ similar to \cite{Nes.Co}. For the dense problems, the algorithms are stopped after 60 seconds of the running time, and for the sparse problems, they are stopped after 45 seconds. Let to define 
\[
\hat{L} = \max_{1 \leq i \leq n} \|a_i\|^2,
\]
where $a_i$ for $i = 1, 2, \cdots, n$ are columns of $A$. For the dense and sparse problems, NESCS, NES05, PGA and FISTA respectively use $L = 10^4 \hat{L}$ and $L = 10^2 \hat{L}$. Moreover, NSDSG employs $\alpha_0 = 10^{-7}$ and $\alpha_0 = 10^{-4}$ for the dense and sparse problems, respectively. To assess the considered algorithms, we now solve (\ref{e.tikh}) and (\ref{e.BPD}) for both dense and sparse data, where the results are respectively illustrated in Figures 8 and 9.

\begin{figure}[h]\label{f.opt1}
\centering
\subfloat[][Function values vs. iterations ]{\includegraphics[width=7.7cm]{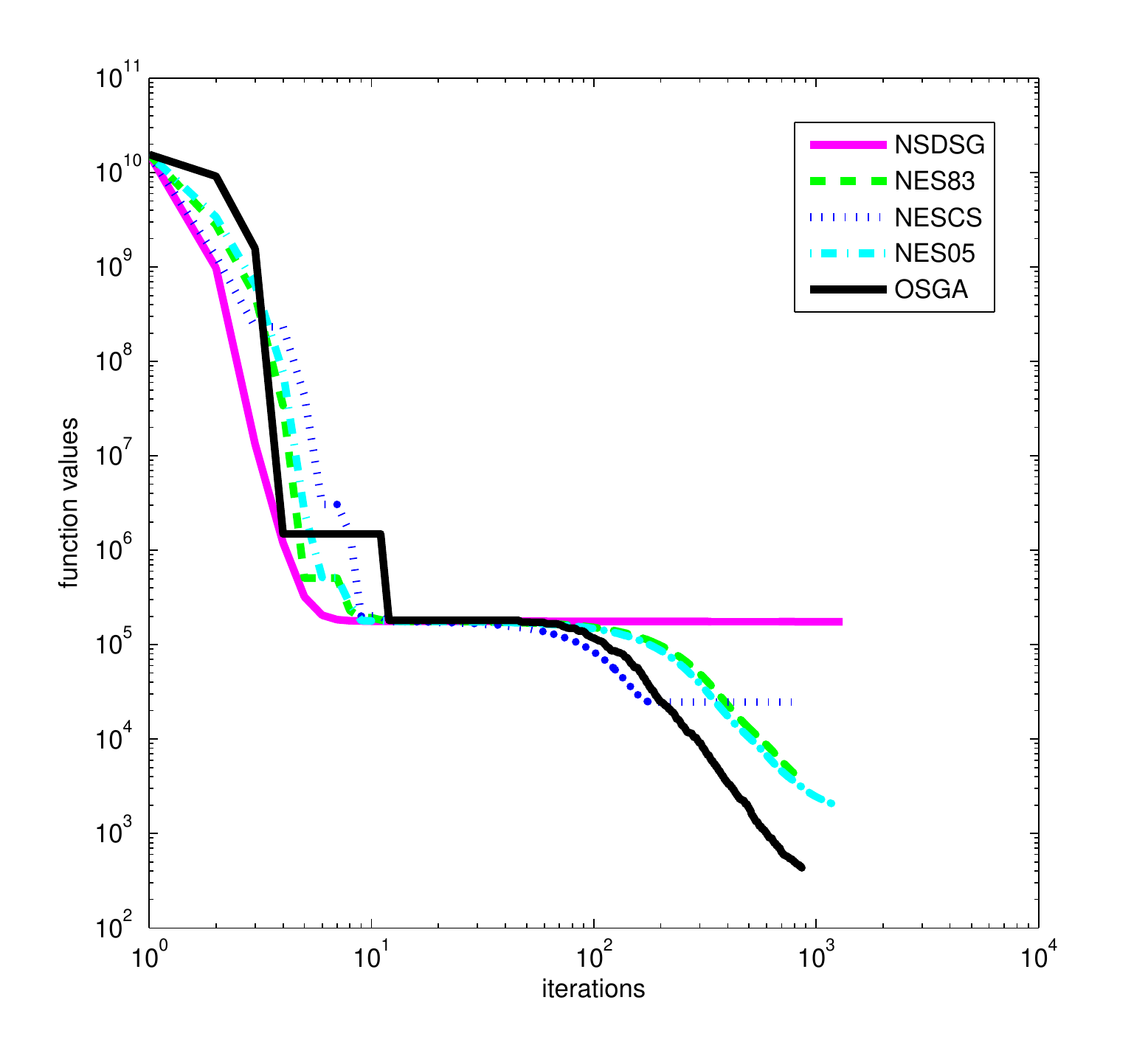}}%
\qquad
\subfloat[][Function values vs. iterations]{\includegraphics[width=7.7cm]{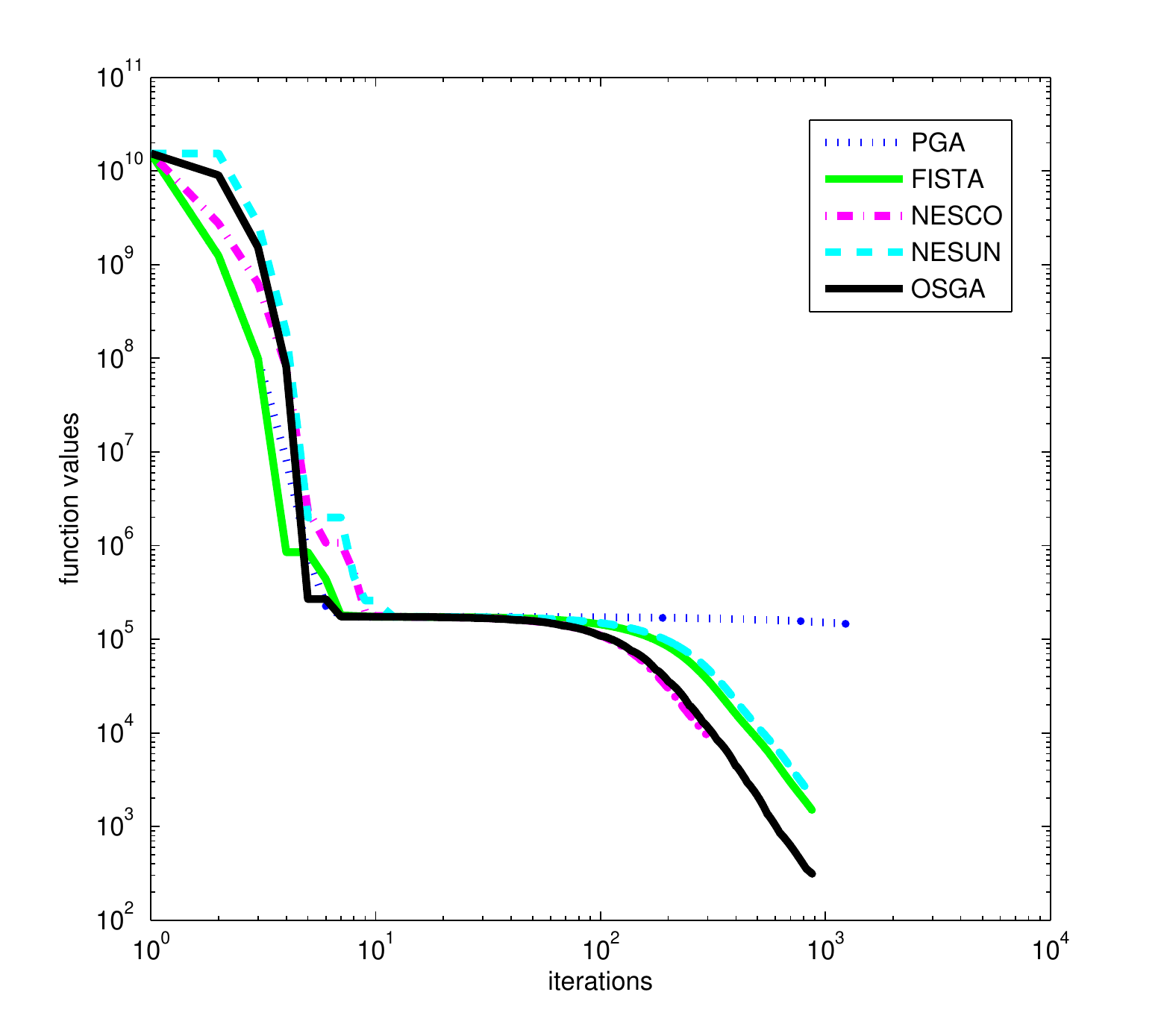}}
\qquad
\subfloat[][Function values vs. iterations ]{\includegraphics[width=7.7cm]{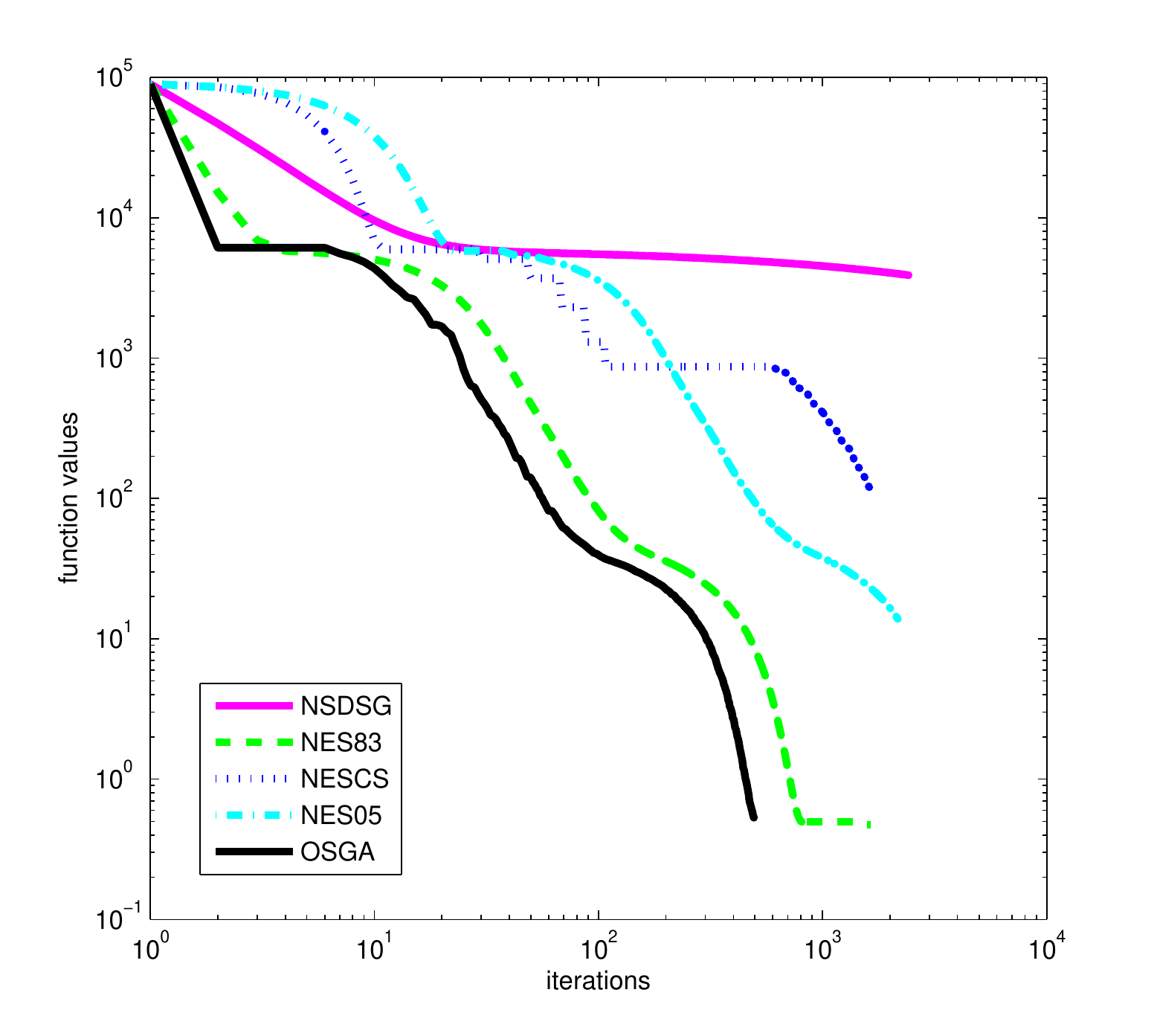}}%
\qquad
\subfloat[][Function values vs. iterations]{\includegraphics[width=7.7cm]{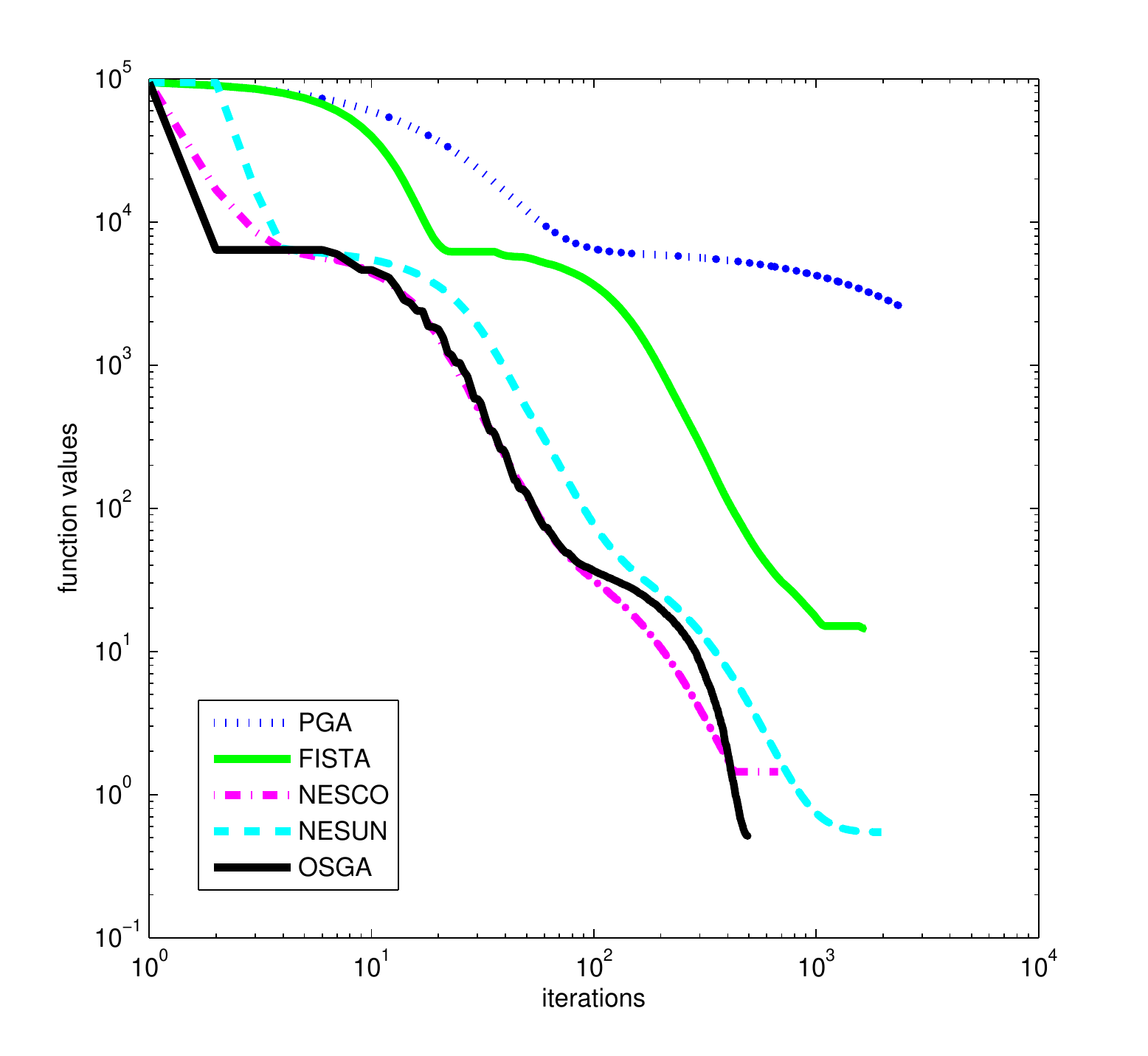}}
\caption{A comparison among the considered algorithms for solving the $l_2^2-l_2^2$ problem (\ref{e.tikh}): subfigures (a) and (b) stand for the dense data, where the algorithms stopped after 60 seconds; subfigures (c) and (d) show the results of the sparse data, where the algorithms stopped after 45 seconds. }
\label{fig:cont}%
\end{figure}

\begin{figure}[h]\label{f.opt2}
\centering
\subfloat[][Function values vs. iterations ]{\includegraphics[width=7.7cm]{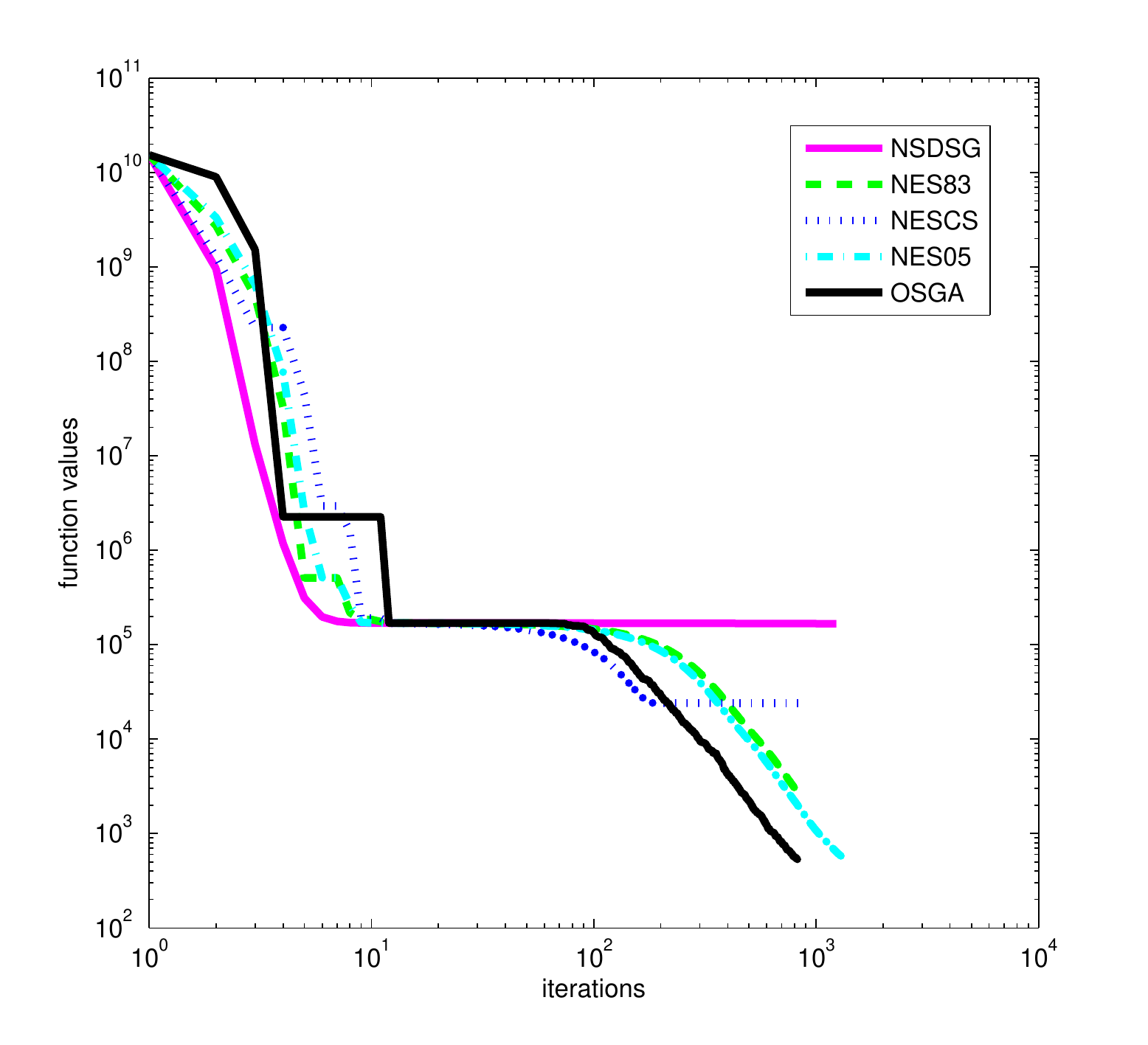}}%
\qquad
\subfloat[][Function values vs. iterations]{\includegraphics[width=7.7cm]{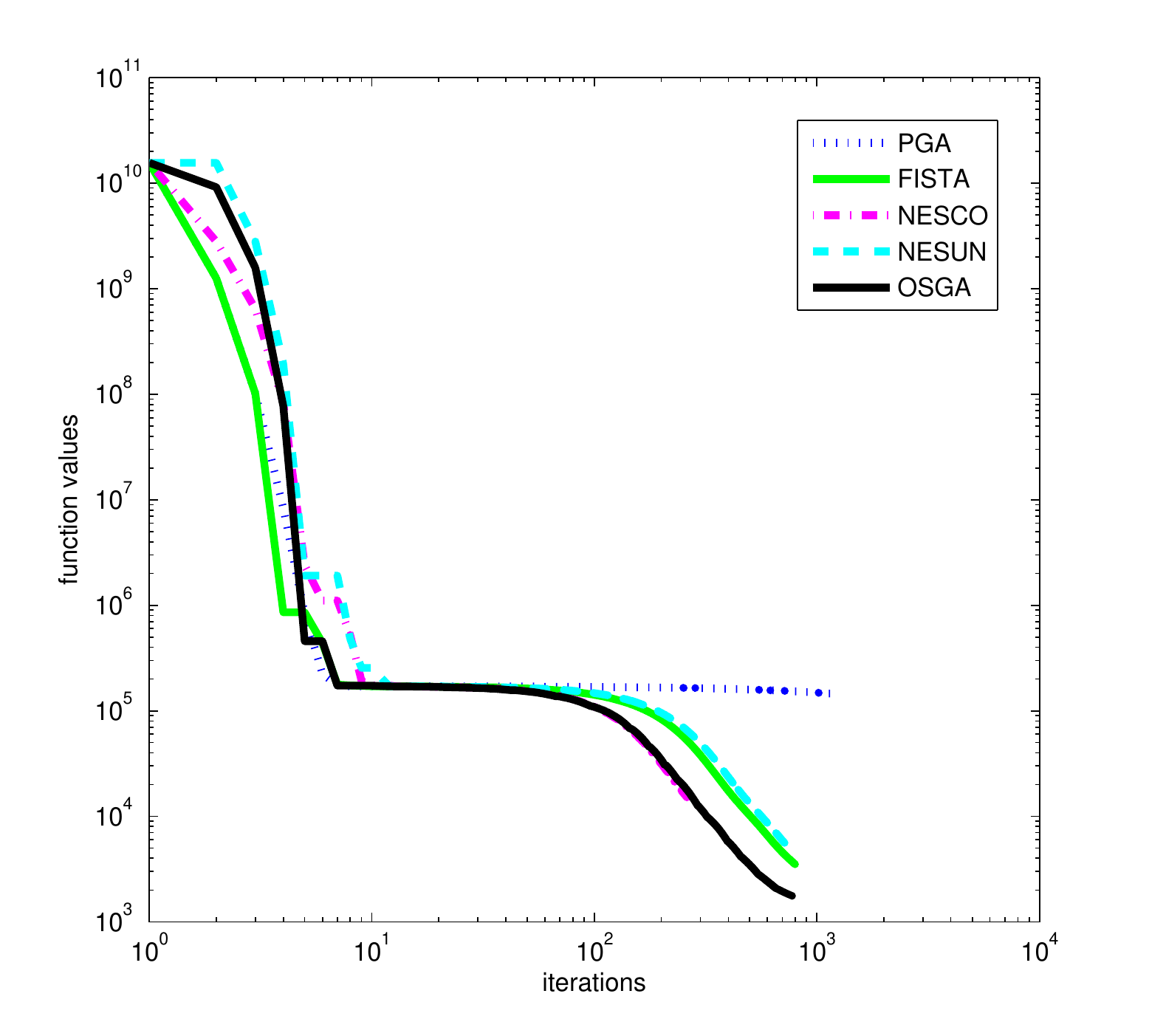}}
\qquad
\subfloat[][Function values vs. iterations ]{\includegraphics[width=7.7cm]{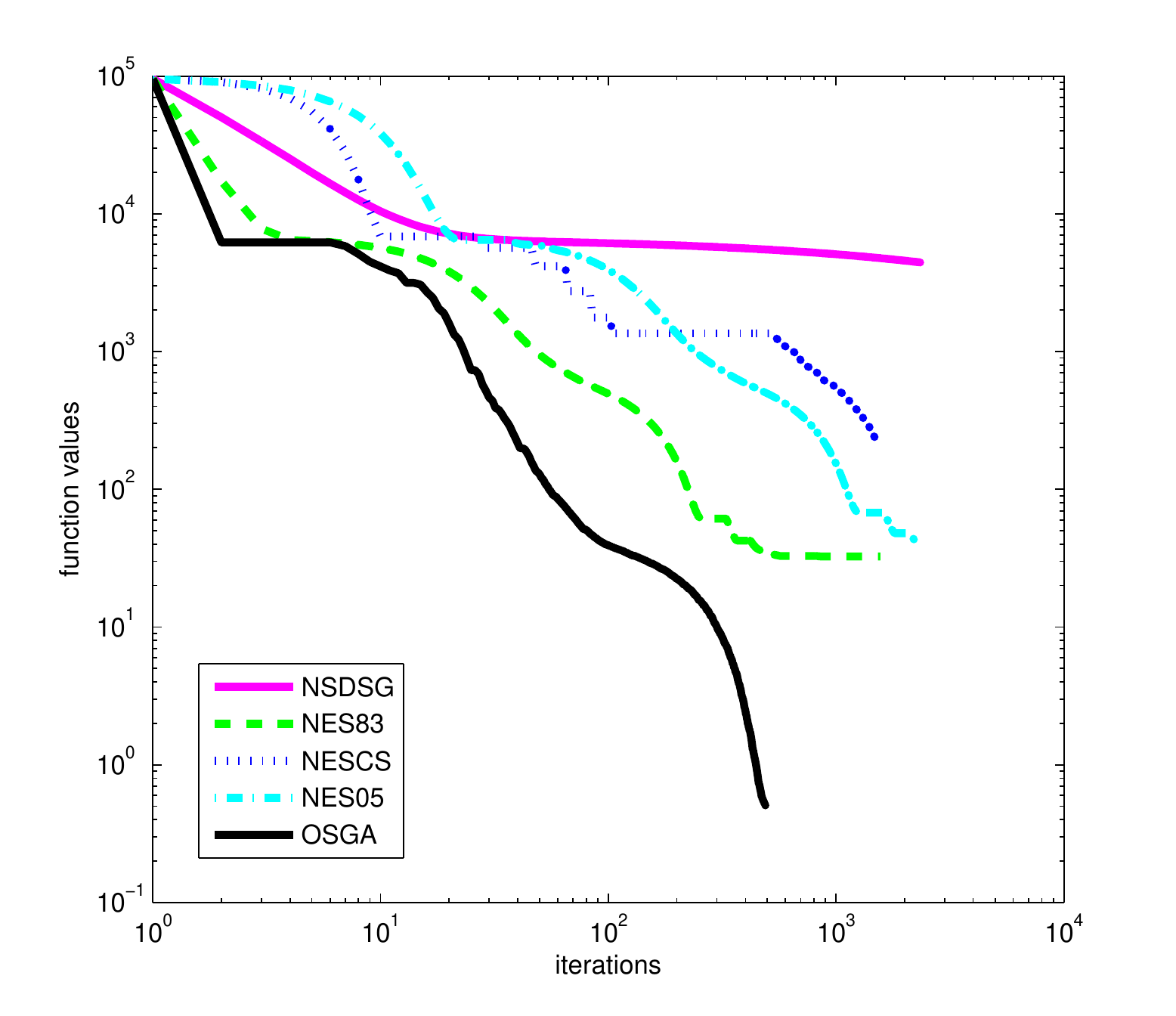}}%
\qquad
\subfloat[][Function values vs. iterations]{\includegraphics[width=7.7cm]{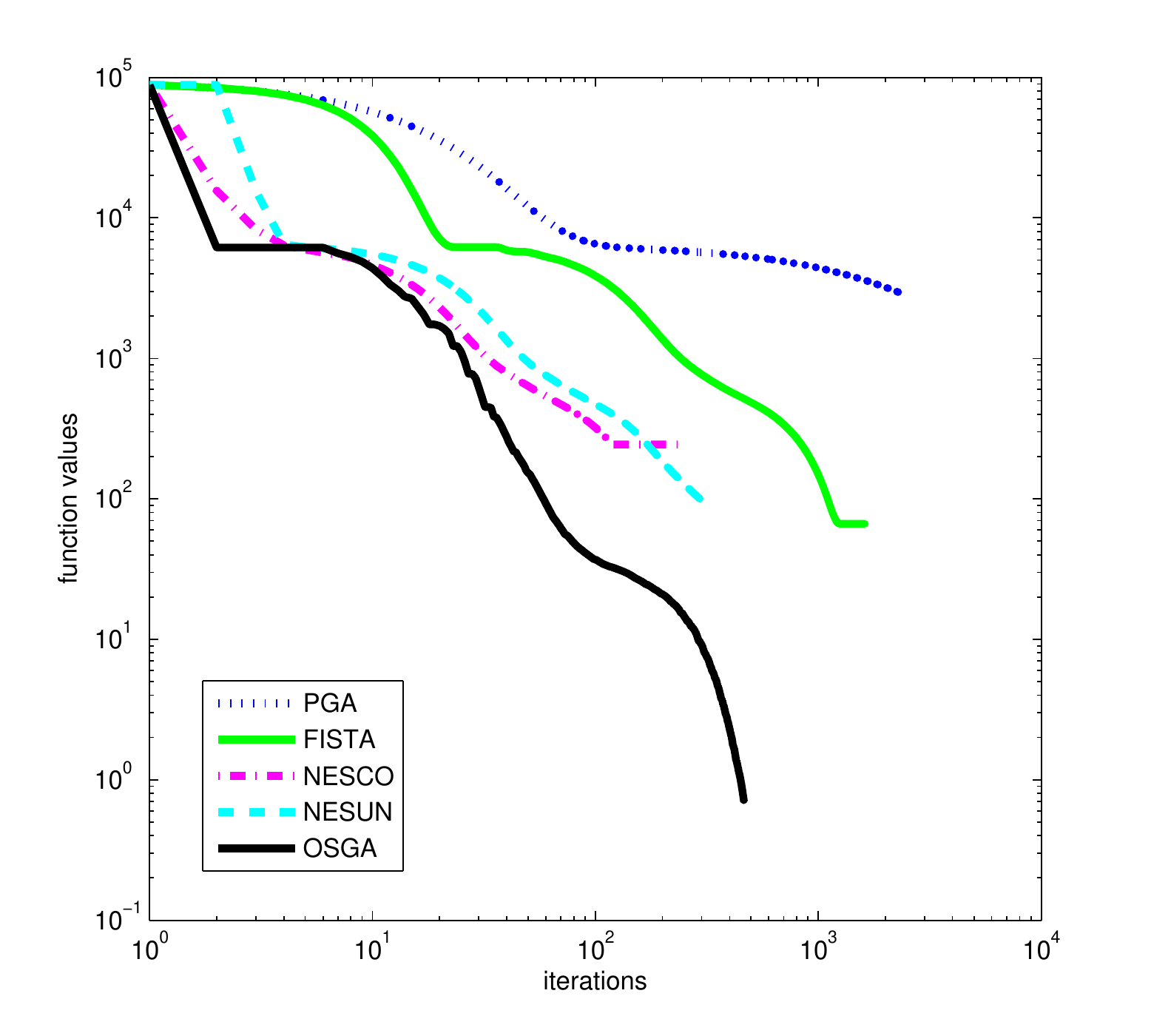}}
\caption{A comparison among the considered algorithms for solving the $l_2^2-l_1$ problem (\ref{e.BPD}): subfigures (a) and (b) stand for the dense data, where the algorithms stopped after 60 seconds; subfigures (c) and (d) show the results of the sparse data, where the algorithms stopped after 45 seconds.}
\end{figure}

In Figure 8, subfigures (a) and (b) illustrate the results of dense data, while subfigures (c) and (d) demonstrate the results of sparse data. From this figure, it can be seen that NSDSG and PGA are not competitive with the others confirming the theoretical results about their complexity. More specifically, Figure 8 (a) and (c) imply that NES83 and NES05 are competitive and produce better results than NESCS. In this case, OSGA achieves the best results, especially for sparse data. Figure 8 (b) and (d) suggest that NESCO, NESUN and FISTA are comparable, but still OSGA generates better results. Figure 9 depicts the results of reconstructing an approximation solution for the nonsmooth problem (\ref{e.BPD}) using the considered algorithms. The same interpretation for Figure 9 indicates that most of the optimal methods produce better results than classical ones and are competitive with each other. However, OSGA commonly obtains the best results among them. Summarizing the results of both figures, we see that the adapted algorithms NES83, NESCS and NES05 are comparable or even better than the the other optimal algorithms, especially for NES83. 

\subsection{Sparse optimization}\label{s.comp}
In recent decades, taking advantage of sparse solutions by using the structure of problems has become a common task in various applications of applied mathematics and particularly optimization, where applications are mostly arising in the fields of signal and image processing, statistics, machine learning and data mining. In the most cases, the problem involves high-dimensional data, and a small number of measurements is accessible, where the core of these problems involves an optimization problem. Thanks to the sparsity of solutions and the structure of  problems, these optimization problems can be solved in reasonable time even for the extremely high-dimensional data sets. Basis pursuit, lasso, wavelet-based deconvolution and compressed sensing are some examples, where the last one receives lots of attentions during the recent years. 

Among fields involving sparse optimization, {\bf \emph{compressed sensing}}, also called {\bf \emph{compressive sensing}} and {\bf \emph{compressive sampling}}, is a novel sensing/sampling framework for acquiring and recovering objects like a sparse image or signal in the most efficient way possible with employing an incoherent projecting basis. On the one hand, conventional processes in the image/signal acquisition from frequency data  follow the Nyquist-Shannon density sampling theorem declaring that the number of samples required for reconstructing an image matches the number of pixels in the image and for recovering a signal without error is devoted by its bandwidth. On the other hand, it is known that most of the data we acquire can be thrown away with almost no perceptual loss. Then, the question is: why all the data should be acquired while most of them will be thrown away? Also, can we only recover parts of the data that will be useful in the final reconstruction? In response to these questions, a novel sensing/sampling approach was introduced which goes against common wisdom in data acquisition, namely compressed sensing, for example see \cite{CanT,CanRT,Don} and references therein. 

It is known that compressed sensing supposes that the considered object, image or signal, has a sparse representation in some bases or dictionaries, and the considered representation dictionary is incoherent with the sensing basis, see \cite{Can,Don}. The idea behind this framework is centred on how many measurements are necessary to reconstruct an image/signal and the related nonlinear reconstruction techniques needed to recover this image/signal. It means that the object is recovered reasonably well from highly undersampled data, in spite of violating the traditional Nyquist-Shannon sampling constraints. In the other words, once an object is known to be sparse in a specific dictionary, one should find a set of measurements or sensing bases supposed to be incoherent with the dictionary. Afterwards, an underdetermined system of linear equations of the form (\ref{e.inv2}) is emerged, which is usually ill-conditioned. Some theoretical results in \cite{CanT,CanRT,Don} yield that the minimum number of measurements required to reconstruct the original object provides a specific pair of measurement matrices and optimization problems. Thus, the main challenge is how to reformulate this inverse problem as an appropriate minimization problem involving regularization terms and solve it by a suitable optimization technique.  

We now consider the inverse problem (\ref{e.inv2}) in the matrix form with $A \in \Rz^{m \times n}$, $x \in \Rz^n$ and $y \in \Rz^m$, where the observation is deficient, $m < n$, as a common interest in compressed sensing. Therefore, the object  $x$ can not be recovered from the observation $y$, even in noiseless system $A x = y$, unless some additional assumptions like sparsity of $x$ is presumed. We here consider a sparse signal $x \in \Rz^n$ that should be recovered by the observation $y \in \Rz^m$, where $A$ is obtained by first filling it with independent samples of the standard Gaussian distribution and then orthonormalizing the rows, see \cite{FigNW}. In our experiments, we consider $m = 5000$ and $n = 10000$, where the original signal $x$ consist of 300 randomly placed $\pm 1$ spikes, and $y$ is generated using (\ref{e.inv2}) by $\sigma^2 = 10^{-6}$. We reconstruct $x$ using (\ref{e.BPD}), where the regularization parameter is set to $\lambda = 0.1 \|\mathcal{A}y\|_\infty$ or $\lambda = 0.001 \|\mathcal{A}y\|_\infty$. The algorithms stopped after 20 seconds of the running time, and the results are illustrated in Figure 10.

\begin{figure}[h]\label{f.spa1}
\centering
\subfloat[][Function values vs. iterations ]{\includegraphics[width=7.6cm]{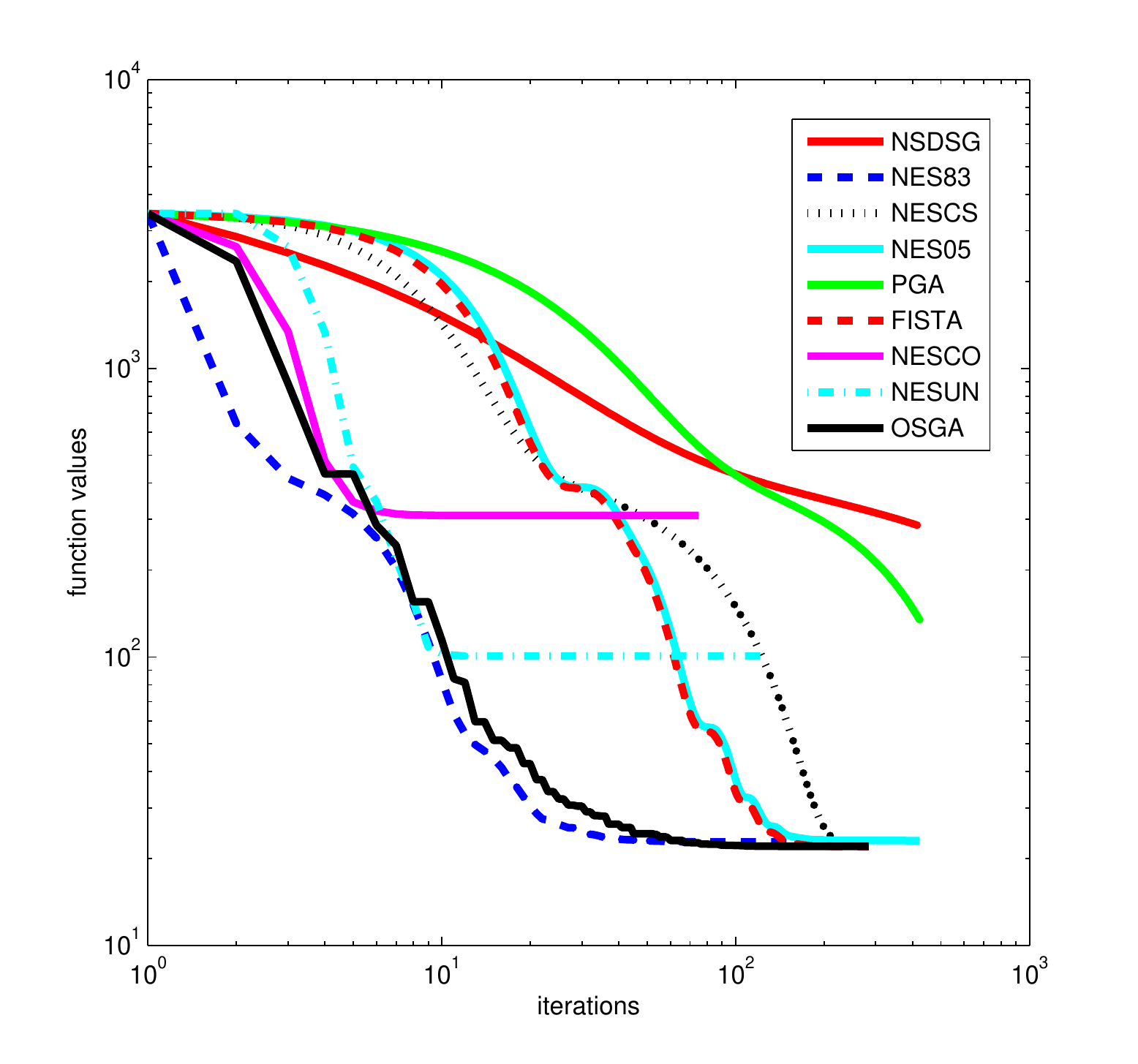}}
\qquad
\subfloat[][MSE vs. iterations ]{\includegraphics[width=7.6cm]{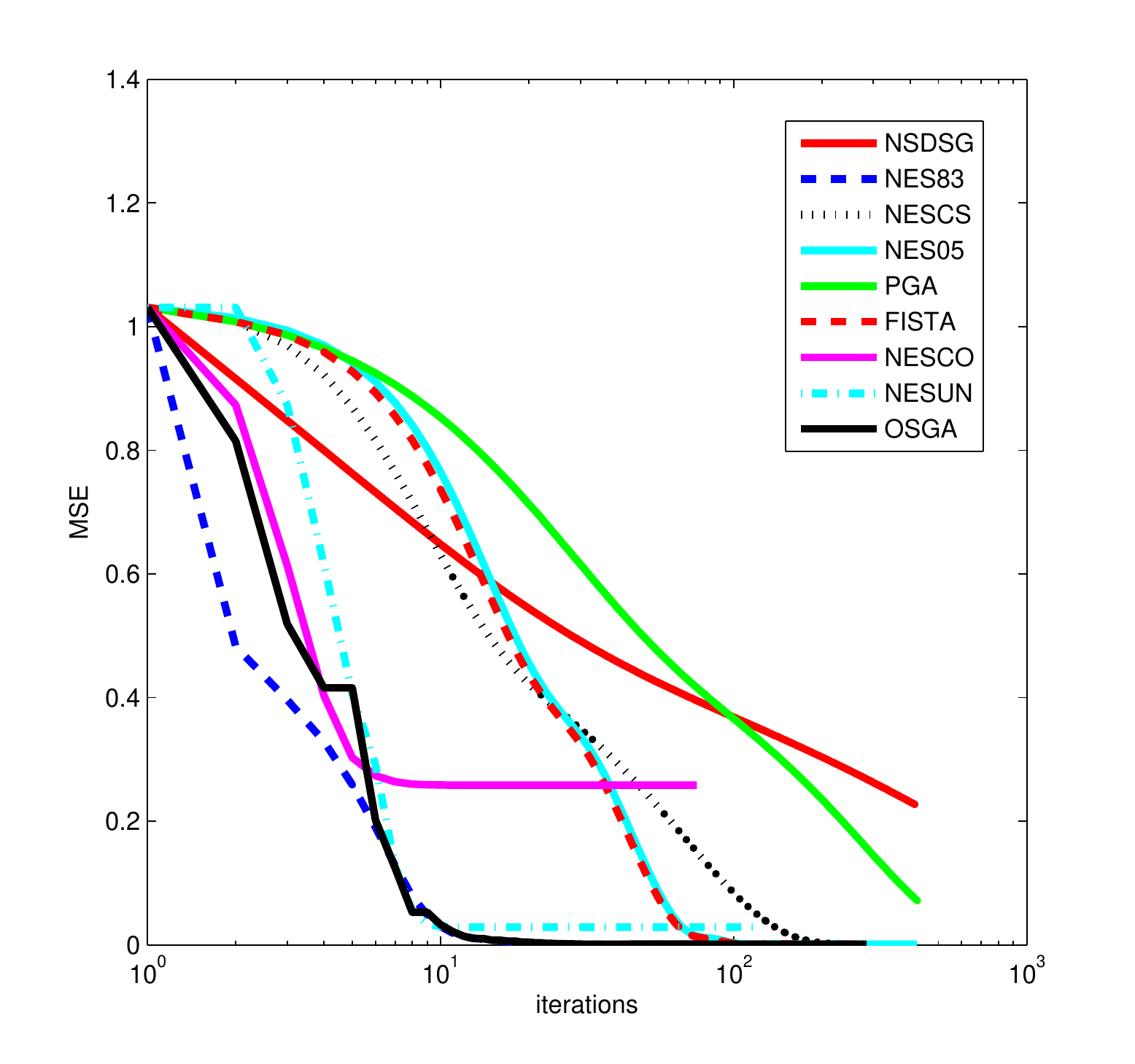}}
\qquad
\subfloat[][Function values vs. iterations ]{\includegraphics[width=7.6cm]{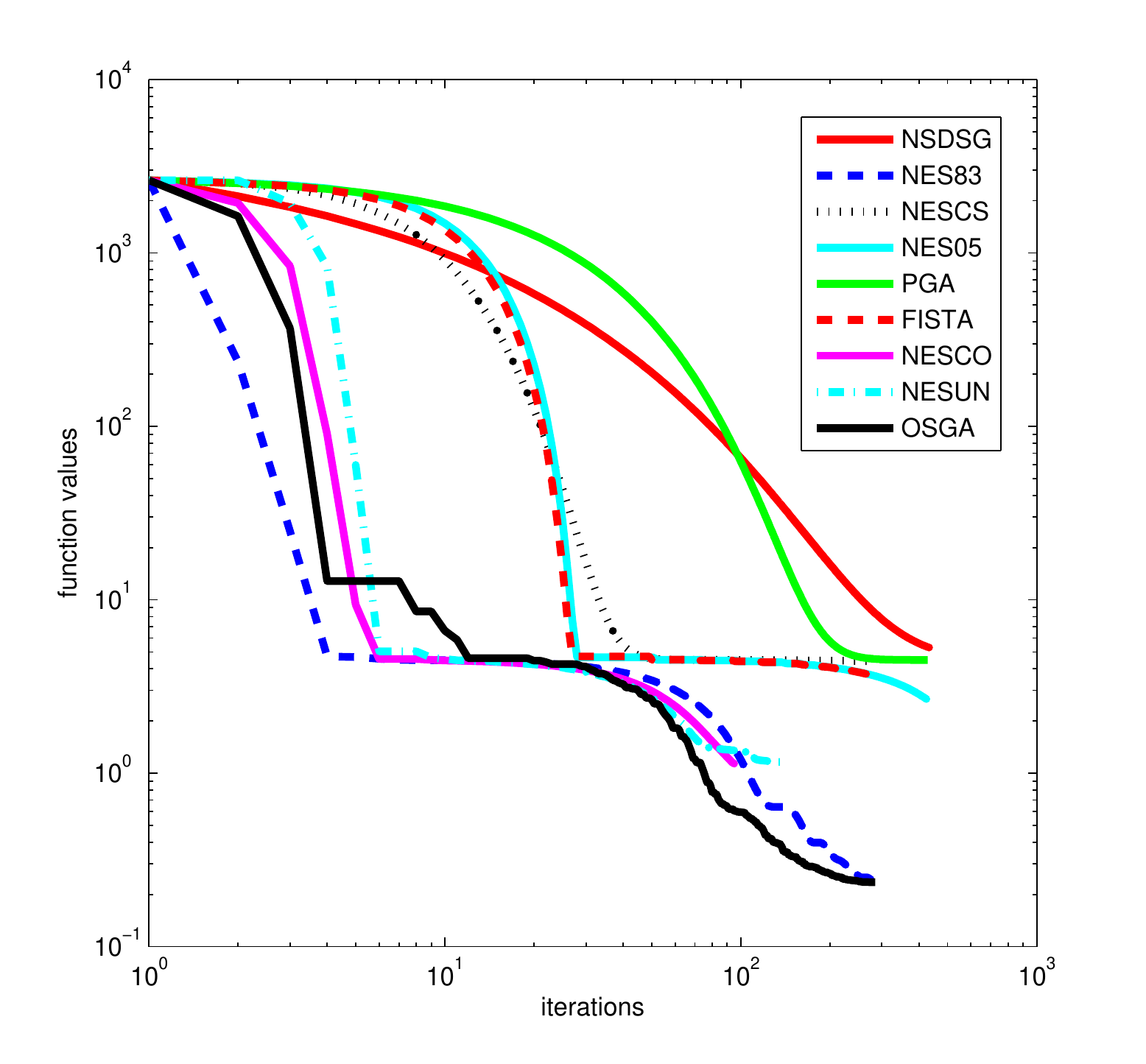}}
\qquad
\subfloat[][MSE vs. iterations ]{\includegraphics[width=7.6cm]{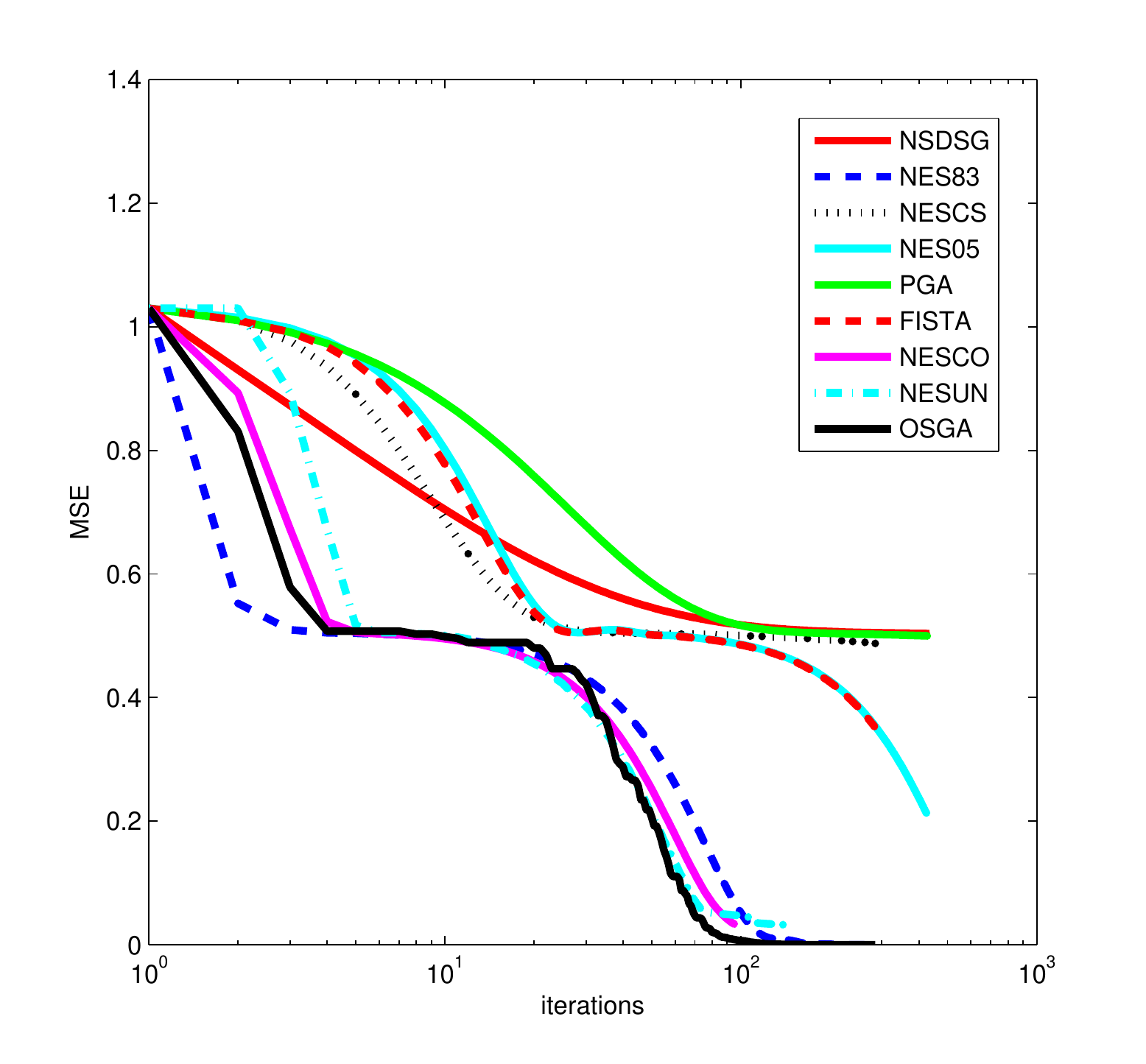}}
\caption{Recovering a noisy sparse signal using $l_2^2-l_1$ problem (\ref{e.BPD}), where the algorithms stopped after 20 seconds. Subfigures (a) and (b) show the results for the regularization parameter $\lambda = 0.1 \|\mathcal{A}y\|_\infty$, while subfigures (c) and (d) indicate the results for the regularization parameter $\lambda = 0.001 \|\mathcal{A}y\|_\infty$.}
\end{figure}

\begin{figure}[h] \label{f.spa2}
    \begin{center}
    \includegraphics[width=16cm]{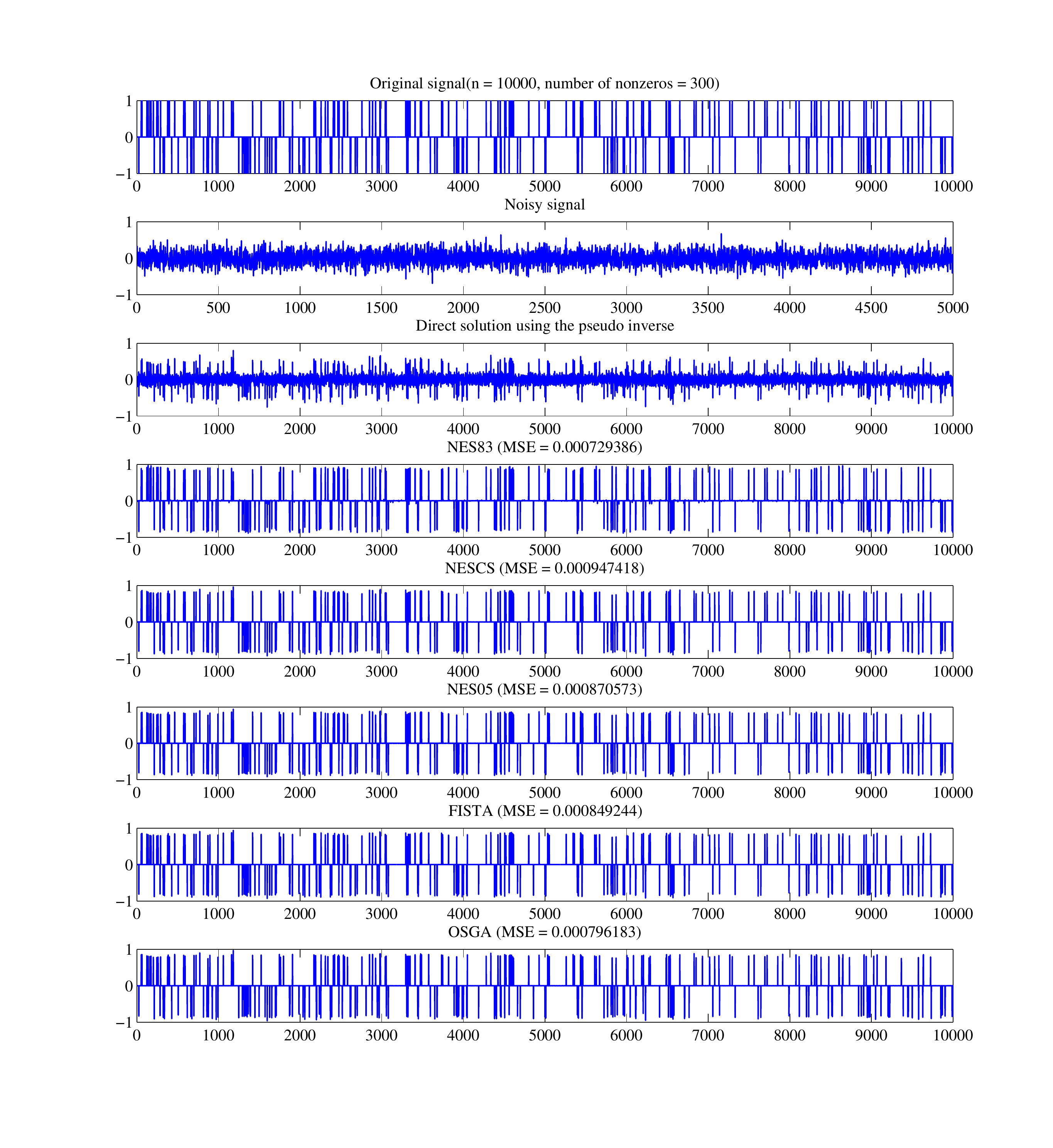}
    \vspace{-1.5cm}
    \caption{Recovering a noisy sparse signal using $l_2^2-l_1$ problem (\ref{e.BPD}) by NES83, NESCS, NES05, FISTA and OSGA. The algorithms stopped  
             after 20 seconds, and the quality measure MSE is defined by (\ref{e.mse}). }
    \end{center}
\end{figure}

Figure 10 illustrates the results of the sparse signal recovery by NSDSG, NSE83, NESCS, NES05, PGA, FISTA, NESCO, NESUN and OSGA, where the the results are compared in the sense of function values and MSE. Indeed, MSE is a usual measure of comparisons in signal reconstruction defined by
\begin{equation}\label{e.mse}
\mathrm{MSE}(x) = \frac{1}{n} \|x - x_0\|^2,
\end{equation}
where $x_0$ is the original signal. The results of Figure 10 indicate that NSDSG, PGA, NESCO and NESUN are unsuccessful to recover the signal accurately, where NESCO and NESUN could not leave the backtracking line searches in a reasonable number of inner iterations which dramatically decrease the step sizes. In Figure 10 (a) and (b), we see that NSE83, NESCS, NES05, FISTA and OSGA can reconstruct the signal accurately, however NES83 and OSGA are much faster than the others regarding both function values and MSE. To see the sensitivity of the algorithms to the regularization parameter, Figure 10 (c) and (d) demonstrate the results of recovery for the smaller regularization parameter $\lambda = 0.001 \|\mathcal{A}y\|_\infty$. It is clear that NESCS, NES05 and FISTA are more sensitive than NES83 and OSGA to the regularization parameter, where they could not get an accurate approximation of the original signal when the regularization parameter decreased. 

To show the results of our test visually, we illustrate the recovered signal of the experiment reported in Figure 10 (a) and (b) in Figure 11. Since NSDSG, PGA, NESCO and NESUN could not recover the signal accurately, we do not consider them in Figure 11. We also add the direct solution $x = A^T(AA^T)^{-1}y$ to our comparison indicating that this solution is very inaccurate. The results suggest that NSE83, NESCS, NES05, FISTA and OSGA recover the signal visually acceptable. Summarizing the results of Figures 10 and 11, it can be seen that NES83 and OSGA just need about 15 iterations to achieve acceptable results, while NESCS, NES05 and FISTA require about 100 iterations to reach the same accuracy.

\section{OSGA software package}
The first version of the OSGA software package for solving unconstrained convex optimization problems is publicly available at 
\begin{center}
\url{http://homepage.univie.ac.at/masoud.ahookhosh/uploads/OSGA_v1.1.tar.gz}
\end{center}
where it is implemented in MATLAB programming language. Some examples are available to show how the user can implement OSGA. The interface to each subprogram in the package is fully documented in the corresponding file. Moreover, the OSGA user's manual \cite{AhoUM} describes the design of the package and how the user can solve his/her own problems.

\section{Concluding remarks}
Simple gradient and subgradient procedures are historically first numerical schemes proposed for solving smooth and nonsmooth convex optimization, respectively. Theoretical and computational results indicate that they are very slow to be able to solve large-scale practical problems in applications. However, these algorithms have a simple structure and need a low amount of memory in implementation. These features cause that first-order algorithms have received much attention during last few decades. In particular, both theoretical and computational aspects of optimal first-order methods are very interesting. 

In this paper, we study and develop an optimal subgradient algorithm called OSGA for the class of multi-term affine composite functions. The algorithm is flexible enough to handle a broad range of structured convex optimization problems and has a simple structure. It also does not need global parameters like Lipschitz constants except for the strong convexity parameter, which can be set to zero if it is unknown. We consider examples of linear inverse problems and report extensive numerical results by comparing OSGA with some state-of-the-art algorithms and software packages. The numerical results suggest that OSGA is computationally competitive or even better than the considered state-of-the-art solvers designed specifically for solving these special problems. It can often reach the fast rate of convergence even beyond the theoretical rate suggesting that OSGA is a promising algorithm for solving a wide range of large-scale convex optimization problems.

We also develop some optimal first-order methods originally proposed for solving smooth convex optimization to deal with nonsmooth problems as well. To this end, we simply pass a subgradient of the nonsmooth composite objective function to these algorithms in place of the gradient. This simply means that we change their original oracle by the nonsmooth black-box oracle. The numerical results show that the adapted algorithms are promising and can produce competitive or even better results than the other considered algorithms specifically proposed for nonsmooth problems. However, they are not supported by any theoretical results at the moment, so it might be an interesting topic for a future research. 

The current paper only considers unconstrained convex optimization problems, however, the constrained version for some simple constraints developed in \cite{AhoN1,AhoN2,AhoN3}. Furthermore, for problems involving costly direct and adjoint operators, a version of OSGA using the multi-dimensional subspace search is established in \cite{AhoN3,AhoN4}. Finally, many more applications like sparse optimization with the $l_1/l_2$ or $l_1/l_\infty$ norm could be considered, but this is beyond the scope of the current work and remains for a future research.\\

{\bf Acknowledgement.} We would like to thank {\sc Arnold Neumaier} for a careful reading of the manuscript and his insightful comments and feedback. We are also grateful to the authors of TwIST, SpaRSA and FISTA software packages for generously making their codes available on the web.



\end{document}